\pdfoutput=1
\documentclass[
fontsize=11pt,
a4paper,
twoside,
english,
headsepline=true,
footsepline=false,
automark,
headings=normal,
parskip=half-,
appendixprefix,
open=any,
cleardoublepage=plain,
abstract=true,
index=totoc,
listof=totoc,
bibliography=totoc
]{scrartcl}
\usepackage[utf8]{inputenc}
\usepackage[english]{babel}
\usepackage[margin=2cm]{geometry}
\usepackage{acronym} 
\usepackage{calc}
\usepackage[table]{xcolor}

\usepackage{graphicx}
\usepackage{epstopdf}
\usepackage{subfigure}

\usepackage{amsmath}
\usepackage{amssymb}
\usepackage{pdfpages}
\usepackage{todonotes}
\usepackage{listings}

\setkomafont{disposition}{\normalcolor\bfseries}
\usepackage[plainpages=false, pdfpagelabels, hidelinks, colorlinks=true, citecolor=red, linkcolor=red, pdfpagelayout=TwoPageRight]{hyperref}

\usepackage{scrhack}

\usepackage{transparent} 

\usepackage{amsfonts, amsthm}

\usepackage{newtxtext}
\usepackage{newtxmath}
\usepackage{pifont}
\usepackage[
        activate={true,nocompatibility},
        final,
        tracking=true,
        kerning=true,
        factor=1100,
        stretch=10,
        shrink=10
]{microtype}

\usepackage{datenumber}
\usepackage{array}
\usepackage{xcolor}
\colorlet{shadecolor}{yellow!20}

\usepackage{mathtools}

\DeclarePairedDelimiter\abs{\lvert}{\rvert}
\DeclarePairedDelimiter\norm{\lVert}{\rVert}
\makeatletter
\let\oldabs\abs
\def\abs{\@ifstar{\oldabs}{\oldabs*}}
\let\oldnorm\norm
\def\norm{\@ifstar{\oldnorm}{\oldnorm*}}
\makeatother

\usepackage{graphicx,wrapfig,lipsum}
\newtheorem{theorem}{Theorem}[section]
\newtheorem{lemma}[theorem]{Lemma}
\newtheorem{corollary}[theorem]{Corollary}
\newtheorem{definition}[theorem]{Definition}

\newtheorem{proposition}[theorem]{Proposition}
\newtheorem{remark}[theorem]{Remark}

\DeclareMathOperator*{\esssup}{ess\,sup}

\usepackage{accents}

\usepackage{enumitem}
\usepackage{enumitem, hyperref}
\makeatletter
\def\namedlabel#1#2{\begingroup
    #2
    \def\@currentlabel{#2}
    \phantomsection\label{#1}\endgroup
}
\makeatother
\DeclareMathOperator{\supp}{supp}

\DeclareMathOperator*{\argmin}{arg\,min}

\DeclareMathOperator{\diff}{diff}
\DeclareMathOperator{\grad}{grad}
\DeclareMathOperator{\CE}{CE}
\DeclareMathOperator{\AC}{AC}
\DeclareMathOperator{\id}{id}
\usepackage{bbold}

\usepackage{mathrsfs}

\usepackage{mathtools}

\newcommand{\babla}{\overline{\nabla}}
\newcommand{\beq}{\begin{equation}}
\newcommand{\eeq}{\end{equation}}
\newcommand{\baq}{\begin{equation}\begin{aligned}}
\newcommand{\eaq}{\end{aligned}\end{equation}}
\newcommand{\beqs}{\begin{equation*}}
\newcommand{\eeqs}{\end{equation*}}
\newcommand{\baqs}{\begin{equation*}\begin{aligned}}
\newcommand{\eaqs}{\end{aligned}\end{equation*}}
\newcommand{\Rd}{{\mathbb{R}^d}}
\newcommand{\R}{{\mathbb{R}}}

\newcommand{\RdRd}{{\mathbb{R}^d\times\mathbb{R}^d}}

\newcommand{\bs}{\boldsymbol}
\newcommand{\dd}{\mathrm{d}}

\newcommand{\A}{\mathcal{A}}
\newcommand{\Mloc}{\mathcal{M}}
\newcommand{\jup}{\textnormal{j}}
\newcommand{\jupbold}{\textnormal{\textbf{j}}}

\newcommand{\varjup}{\textnormal{\j}}
\newcommand{\varjupbold}{\textnormal{\textbf{\j}}}

\newcommand{\vup}{\textnormal{v}}
\newcommand{\vupbold}{\textnormal{\textbf{v}}}

\newcommand{\dup}{\textnormal{d}}

\newcommand{\rhoinit}{\varrhoup_0}
\newcommand{\rhofin}{\varrhoup_1}

\usepackage[many]{tcolorbox}
\tcbuselibrary{breakable, skins}

\graphicspath{{img/},{img-results/}}

\setkomafont{pageheadfoot}{\small\scshape} 
\newcommand{\ignore}[1]{} 
\setdatetoday

\newcommand{\dckeywords}{gradient flow, nonlocal, Fokker-Planck equation, optimal transport}

\hypersetup{
 pdftitle = {Nonlocal cross-interaction systems on graphs},
 pdfkeywords = {\dckeywords}, 
 pdfcreator = {pdfTeX with Hyperref and Thumbpdf}, 
 pdfproducer = {LaTeX, hyperref, thumbpdf}, 
}

\titlehead{
  \vspace*{0.0cm}
}

\title{\huge
  Nonlocal cross-interaction systems on graphs:
  \\
  \vspace*{0.0cm} 
  { Nonquadratic Finslerian structure and nonlinear mobilities}
}
\author{Georg Heinze\thanks{corresponding author}$\;\;^,$\thanks{Fakult\"at f\"ur Mathematik, Technische Universit\"at Chemnitz Reichenhainer Stra{\ss}e 41, Chemnitz, Germany. (\{georg.heinze,jfpietschmann\}@math.tu-chemnitz.de)}\,\,,\, Jan-Frederik Pietschmann$^{\dag}$, Markus Schmidtchen\thanks{Institute of Scientific Computing, Technische Universit\"at Dresden, Zellescher Weg 25,
01069 Dresden, Germany. (markus.schmidtchen@tu-dresden.de).}}

\begin{document}

\maketitle
\begin{abstract}
We study the evolution of a system of two species with nonlinear mobility and nonlocal interactions on a graph whose vertices are given by an arbitrary, positive measure.
To this end, we extend a recently introduced $2$-Wasserstein-type quasi-metric on generalized graphs, which is based on an upwind-interpolation, to the case of two-species systems, concave, nonlinear mobilities, and $p\ne 2$. 
We provide a rigorous interpretation of the interaction system as a gradient flow in the Finslerian setting, arising from the new quasi-metric. 
\end{abstract}

\vskip .4cm

\begin{flushleft}
    \noindent{\makebox[1in]\hrulefill}
\end{flushleft}
\begin{flushright}
    \noindent{\makebox[1in]\hrulefill}
\end{flushright}

\pagenumbering{arabic}




\section{Introduction}

The goal of this paper is the study of a two-species nonlocal interaction system with nonlinear mobility on a graph. It is well known that, in a local and continuous setting, evolution equations for a single species $\rho_t(x)$ of the form
\begin{align}
    \label{eq:aggregation}
    \partial_t \rho_t =\nabla \cdot(\rho_t \nabla K \ast \rho_t),
\end{align}
can be cast into a Wasserstein gradient flow framework, cf. \cite{Topaz2006,Carrillo2011,d2006self, carrillo2010particle,mogilner1999non}. 
Here, the corresponding functional
\begin{align*}
    \mathcal{E}(\rho) = \frac12 \iint_{\RdRd} K(x-y)\dd \rho(x) \dd \rho(y),
\end{align*}
denotes the interaction energy which encodes the nature of the interactions among members of the species. The so-called aggregation equation,  \eqref{eq:aggregation}, can be obtained as the mean-field limit of a particle system associated to it by letting the number of particles, $N$, tend to infinity \cite{dobrushin1979vlasov, neunzert1984introduction, golse2016dynamics, bonaschi2015equivalence}. It is straightforward to introduce a second species to the dynamics such that the energy functional becomes
\baq
    \label{eq:Energy}
    \mathcal{E}(\rhoup) &= \frac{1}{2}\sum_{i,k=1}^2\iint_\RdRd K^{(ik)}(x,y) \dd\rho^{(i)}(x) \dd\rho^{(k)}(y),
\eaq
where $\rhoup = (\rho^{(1)},\rho^{(2)})$ and $\rho^{(i)}\in \mathcal{P}(\Rd)$, with $i=1,2$, denote the two species. Throughout we refer to $K^{(11)}, K^{(22)}$ as the \emph{self-interaction} potentials and we call $K^{(12)}, K^{(21)}$ the \emph{cross-interaction} potentials, respectively. It is worthwhile to highlight that, under the condition that $K^{(12)} = \beta K^{(21)}$, for some $\beta > 0$, the two-species interaction energy gives rise to a Wasserstein gradient flow on the product space and the evolution of the two densities is governed by the equations
\begin{align}
    \label{eq:interaction_system}
    \partial_t \rho_t^{(i)} = \nabla \cdot \left(\rho_t^{(i)} \nabla \left( K^{(i1)}\ast \rho_t^{(1)} + K^{(i2)}\ast \rho_t^{(2)} \right)\right),
\end{align}
with $i=1,2$, cf. \cite{DiFrancesco2013, di2016nonlocal, di2021many, 1078-0947_2020_2_1191, DEF2018}. For a suitable product space metric, the system is the epitome  of interaction models found in many applied contexts for instance in cell-cell adhesion models \cite{ armstrong2006continuum, painter2003modelling, carrillo2019population, barre2019modelling}, chromatophore interactions in the skin of zebrafish \cite{volkening2015modelling, volkening2018iridophores, volkening2020modeling}, and multi-species systems with volume exclusion effects that result in cross-diffusion interaction systems,  \cite{berendsen2017cross, carrillo2018zoology, di2018nonlinear,burger2020segregation}. 

\subsection{Graph setting, non-linear mobility, and $p\neq 2$}

While the space-time continuous dynamics and the Wasserstein gradient flow structure are well-understood, the situation is much more delicate when considering the flow of two densities, one per species, on graphs. In our work, we extend the recent work by Esposito et al. \cite{EPSS2021}, which has established a graph analog of the aggregation equation. Their work shows that an appropriate definition of the geometry of the underlying space allows to understand a class of interaction equations on graphs as gradient flows, albeit in a Finslerian framework rather than the usual Riemannian setting.

The starting point is the dynamic formulation of the $2$-Wasserstein distance due to \cite{BenamouBrenier2000}. There it was shown that the $2$-Wasserstein distance can be characterized equivalently by minimizing (twice) the kinetic energy over all connecting paths
\baq\label{eq:BB}
W_2^2(\rhoinit, \rhofin) = \inf_{(\rho_t, j_t)_t} \int_0^1 \int_{\R^d} \frac{j_t^2}{\rho_t} \;\dd\mu \dd t,
\eaq
where $\mu\in\Mloc^+(\Rd)$ is a suitable reference measure and the infimum is taken over all pairs $(\rho_t,j_t)_{t\in[0,1]}$, $\mu$-a.e. satisfying
\baqs 
\partial_t \rho_t + \nabla \cdot j_t = 0,
\eaqs
as well as $\rho_0 = \rhoinit$, and $\rho_1 = \rhofin$.
While developed with numerical applications in mind, \eqref{eq:BB} turned out to be a starting point for various adaptations and generalizations of the classical Wasserstein distance. Many of these Wasserstein-type distances modify the action density $\A(\rho, j)\coloneqq\int_\Rd\abs{j}^2/\rho \;\dd\mu$ that appears inside the time integral in \eqref{eq:BB}, introducing nonlinear mobilities \cite{Dolbeault_2008}, reaction terms that allow for initial and final measures having different masses \cite{chizat2018interpolating} or even different actions in the interior and on the boundary of a given domain \cite{Monsaingeon2021}. Of particular interest here are the works which identify nonlocal equations such as the nonlocal heat equation \cite{MaasGradFlowEntropyFiniteMarkov2011}, nonlocal adaptations of the Fokker-Planck equation \cite{Chow2012} and a nonlocal porous medium equation \cite{ErbarMaasGradientFlowPorousMedium2014} on finite graphs or finite Markov chains as gradient flows for suitable Wasserstein-type metrics. Also an extension to treat the nonlocal heat equation on $(\Rd,\mu)$ for some Radon measure $\mu$ was considered in \cite{erbar2012gradient}.
An important challenge faced when transferring the notion of a gradient flow to the setting of graphs is the need to compare fluxes or velocities (edge-based quantities) with densities (vertex-based quantities).
This difficulty can be remedied by introducing a weight function
$\theta:[0,\infty)\times[0,\infty)\to[0,\infty)$, which acts as an interpolation of quantities defined on connected vertices. More precisely, if $\rho(x)$ and $\rho(y)$ denote the densities on the vertices $x$ and $y$, respectively, which share a connecting edge, then we shall think of $\theta(\rho(x),\rho(y))$, as the edge density. This definition allows for an adaptation of the action density to the graph setting \cite{EPSS2021}, i.e., 
\baqs
    \mathcal{A}(\rho,j)= \iint_G \frac{\abs{j(x,y)}^2}{\theta(\rho(x),\rho(y))}\dd\mu(x)\dd\mu(y),
\eaqs
where $G$ is the set of edges and $\mu\in\Mloc^+(G)$ is again a reference measure, which, is arbitrary if $\theta$ is 1-homogeneous. There are different choices for $\theta$ that all seem reasonable but have a strong impact on the metric structure derived from $\mathcal{A}$. In \cite{MaasGradFlowEntropyFiniteMarkov2011} and \cite{Chow2012} the logarithmic mean $\theta_\text{l}(r,s)=\frac{r-s}{\log r-\log s}$ is shown to be a suitable choice for equations involving diffusive terms as it allows for a discrete chain rule, a fact that has already been observed and used in the finite volume community, cf. \cite{chang1970practical, PZ2018, carrillo2020convergence}. However, as was pointed out in \cite{EPSS2021}, this choice does not allow for an increase of the support of the solution in the absence of diffusion. Thus, a different choice must be made to obtain physically reasonable solutions. Similar problems occur with the geometric mean $\theta_\text{g}(r,s)=\sqrt{rs}$, while dynamics using the arithmetic mean $\theta_\text{a}(r,s)=\frac{r+s}{2}$ as an interpolation function may lead to negative densities and is therefore also not a reasonable choice. However, it is known that for transport equations, upwind schemes yield stable and structure-preserving discretizations, which motivates the choice
\begin{align*}
    \theta_j(r,s)=r\mathbb{1}_{\{j>0\}}(r,s)+s\mathbb{1}_{\{j<0\}}(r,s)
\end{align*}
as an interpolation function. This comes at the price of losing the antisymmetry of the action density, obtaining a Finslerian structure instead of a Riemannian structure. Yet, this structure is sufficient to define a notion of gradient flows as curves of maximal slope on graphs and leads to stability of gradient flows under narrow convergence, obtaining the existence of gradient flows for a large set of base measures $\mu$ via approximation with finite graphs, \cite{EPSS2021}.

In the context of finite volume discretizations, this upwinding has plenty of precedent, cf. \cite{eymard2000finite}, and references therein. In particular, for equations exhibiting an entropy-dissipation structure, it has been observed that certain finite volume discretizations can be used to preserve the structure at the discrete level, cf. \cite{bessemoulin2012finite} for a general class of drift-diffusion equations,  \cite{carrillo_chertock_huang_2015, bailo2020convergence, B.C.H2020, schlichting2020scharfetter, CJLV2016, BCH2021} for extensions to nonlocal drift-diffusion equations. The strategy was then extended to systems of cross-interaction species in  \cite{carrillo2018zoology, carrillo2020convergence, BCH2021, LZ2021}.

Even though cross-interactions between opposing species introduce a coupling in the velocity fields, using this upwinding, we can show that the evolution on the graph is, indeed, a gradient flow in the set of probability measures with respect to an appropriately defined Finsler product metric for the interaction energy
\begin{align}\tag{\eqref{eq:Energy} revisited}
  \mathcal{E}(\rhoup) = \frac{1}{2}\sum_{i,k=1}^2\iint_\RdRd K^{(ik)}(x,y) \dd\rho^{(i)}(x) \dd\rho^{(k)}(y).
\end{align}
The interpretation of the dynamics as a gradient flow on a graph is established by noting that the quasi metric in the two species case decomposes into the sum of two instances of the metric for the single species, introduced in \cite{EPSS2021} --- in analogy with the strategy for the continuous dynamics in \cite{DiFrancesco2013}.

A formal Finslerian structure also naturally appears in the space-time continuous setting when studying the $p$-Wasserstein distance $W_p$, see \cite{Agueh2012_finsler}. The dynamic formulation of $W_p$ is given as 
\baq\label{eq:BBp_neq_2}
    W_p^p(\rhoinit, \rhofin) = \inf_{(\rho_t, j_t)_t} \int_0^1 \int_{\R^d} \frac{\abs{j_t}^p}{(\rho_t)^{p-1}} \;\dd\mu \dd t.
\eaq
Agueh still provides a norm in this setting, yet not an inner product. This allows to define the differential and gradient of functions on the set of probability measures with this distance, giving rise to a notion of gradient flows in this space.

Moreover, \eqref{eq:BBp_neq_2} can be generalized to the case $p\neq 2$, \cite{Dolbeault_2008} studied generalizations of \eqref{eq:BBp_neq_2} including unbounded, concave mobilities, see also \cite{CARRILLO20101273}. These nonlinear mobilities give rise to evolution equations of the form
\begin{align*}
    \partial_t \rho = \nabla \cdot (m(\rho)\nabla K \ast \rho).
\end{align*}
Depending on the specific choice of the mobility, it becomes necessary to add recession terms to the action functional to ensure its lower semicontinuity. Finally, \cite{Lisini2010} extended this study to bounded mobilities, which naturally appear in models with volume filling. 
\subsection{Our contribution}
In this paper, we show that the setting of \cite{EPSS2021} can be carried over to systems of two (or potentially multiple) interacting species with a non-linear mobility and remains valid in the case $p\neq 2$. This way, we derive a corresponding two-species interaction equation as a gradient flow. It is given by
\baq
	\label{eq:NL2CIE_intro}
	\partial_t\rho_t^{(i)}(x) 
	&=  - (\babla\cdot j_t^{(i)})(x),\\
	j_t^{(i)}(x,y) 
	&= \left[m(\rho_t^{(i)}(x),\rho_t^{(i)}(y))(v_t^{(i)})_+(x,y)\right]^{\frac{1}{p-1}} -\left[m(\rho_t^{(i)}(y),\rho_t^{(i)}(x))((v_t^{(i)})_-(x,y)\right]^{\frac{1}{p-1}},\\
	v_t^{(i)}(x,y) 
	&= - \babla \left[\left(K^{(i1)}\ast\rho_t^{(1)}\right)(x,y)+ \left(K^{(i2)}\ast\rho_t^{(2)}\right)(x,y)\right],
\eaq
where $p\in(1,\infty)$, $i=1,2$, $x,\,y  \in \R^d$ and $t\geq0$. Here, the quantities $\mu$ and $\eta$ encode the structure of the graph, $m$ is a concave mobility, and the operators $\babla$ and $\babla\cdot$ are discrete analogues of gradient and divergence. Precise definitions will be given in Section \ref{sec:analytical} below.

The core novelties of our work are:
\begin{itemize}
    \item Introduction of an action functional, which incorporates an upwind structure and a concave (un)bounded mobility.
    \item Extension of the Finslerian structure from \cite{EPSS2021} to the case $p\in(1,\infty)$, weakening the notion of Finsler metric, in the spirit of \cite{Agueh2012_finsler}, and employing the notion of the metric gradient.
    \item Derivation of all the core results from \cite{EPSS2021} in the generalized framework, in particular a chain rule, a stability result for gradient flows, and an existence result beyond finite graphs.
    \item Extension of our framework to multiple species and derivation of a gradient flow structure for energies with symmetric cross-interaction.
\end{itemize}


The remainder of the paper is organized as follows. In section \ref{sec:analytical}, we define the graph setting and introduce the notion of action functional and continuity equation, as well as the induced quasimetric. In Section \ref{sec:FinslerGrad} we discuss the Finsler geometry, give the interpretation of our system as a gradient flow using a suitable variational characterization and provide existence and stability results. 

\section{Analytical setting of the two species graph structure}
	\label{sec:analytical}
This section provides the necessary extension of the dynamic $2$-Wasserstein distance \eqref{eq:BB} to the graph setting, including a nonlinear mobility and for $p\neq 2$. To this end, we give a precise definition of the corresponding action functional and study some of its properties. Then, after introducing a notion of generalized continuity equations, we can define and analyze the corresponding quasimetric.
We start by introducing the graph setting. Throughout, $\Mloc(X)$ ($\Mloc_+(X)$) denotes the space of (nonnegative) Radon measures on the space $X$. The vertices of our graph are defined by the base measure $\mu\in \Mloc^+(\Rd)$. We define a nonnegative weight function $\eta:\Rd\times\Rd\to[0,\infty)$ and thereby the edges of the undirected graph as
\baqs
    G\coloneqq\{(x,y)\in\Rd\times\Rd:x\neq y, \eta(x,y)>0\}.
\eaqs
Setting $1<p=q/(q-1)<\infty$, throughout we shall make use of the following set of technical assumptions on $\mu$ and the $\eta$:
\begin{align}
	\label{W}\tag{W} &\text{(continuous symmetric weight) } &\eta|_G \in C(G,[0,\infty)),\;\forall x,y\in\Rd \text{ it holds }\eta(x,y) = \eta(y,x),\\
	\label{MB2}\tag{MB1} &\text{(individual moment bound) } &\int_\Rd (1+|x|^p)\dd\mu(x) \leq C_\mu,\\
	\label{MB}\tag{MB2} &\text{(joint moment bound) } &\sup_{x\in\Rd} \int_\Rd\abs{x-y}^q\lor\abs{x-y}^{pq}\eta(x,y)\dd\mu(y)\leq C_\eta,\\
	\label{BC}\tag{BC} &\text{(local blow-up control) } &\lim_{\varepsilon\rightarrow 0}\sup_{x\in\Rd} \int_{B_\varepsilon(x)\setminus\{x\}}\abs{x-y}^q\eta(x,y)\dd\mu(y) = 0,
\end{align}
for some constants $C_\eta, C_\mu > 0$ and where $B_\varepsilon(x) \coloneqq \{y\in\Rd:\abs{x-y}<\varepsilon\}$. 
\begin{remark}\label{rem:MB2}
While the assumptions \eqref{W}, \eqref{BC} and \eqref{MB} are similar to the assumptions in \cite{EPSS2021}, \eqref{MB2} is a new assumption required due to the added nonlinear mobilities. It is needed, whenever we apply Lemma \ref{lem:bound_by_A} in the sequel. However, for a large class of mobilities Lemma \ref{lem:bound_by_A} can be refined to allow dropping assumption \eqref{MB2}. More details on this will be given in Remark \ref{rem:bound_by_A_refined}. 
\end{remark}

Next, we establish nonlocal analogues for the gradient and the divergence as foreshadowed in the introduction.
\begin{definition}[Nonlocal gradient and nonlocal divergence]
Given $\varphi:\Rd\to\R$, we define the nonlocal gradient $\babla\varphi$ by
\baqs
	\babla\varphi(x,y)\coloneqq \varphi(y)-\varphi(x).
\eaqs
Given $j \in \Mloc(\Rd)$ we define the nonlocal divergence $\babla\cdot j$ by
\baqs
	\int_\Rd \varphi(x) \dd \babla\cdot j(x) \coloneqq -\frac{1}{2}\iint_G\babla\varphi(x,y)\eta(x,y)\dd j(x,y)\quad \forall \varphi:\Rd\to\R.
\eaqs
\end{definition}

\subsection{Definition of the action functional}
\begin{definition}[Mobility and density functions]\label{def:mob}
Given two thresholds $R,S\in(0,\infty]$, we define a \emph{mobility function} $m\in C([0,R)\times[0,S))$, which is concave and strictly positive in $(0,R)\times(0,S)$. For $(r,s)\in [0,R)\times[0,S)$ we denote $m(r,S) = \lim_{s\to S}m(r,s)$ and $m(R,s) = \lim_{r\to R}m(r,s)$. We call such a mobility $m$ \emph{upwind-admissible} if for every $s\ge 0$ we have $m(0,s)=0$. Furthermore, if $R=S=\infty$, which implies that $m$ is nondecreasing, we set
    \baq\label{eq_def:m_infty}
        m_\infty(r,s) \coloneqq \lim_{\lambda\to\infty}\frac{1}{\lambda}m(\lambda r, \lambda s).
    \eaq
    We say that the growth of $m$ is \emph{uniformly sublinear}, if $m_\infty\equiv 0$.

Given a mobility $m$ and an exponent $p\in(1,\infty)$, we define the convex, lsc. density function $\alpha:\R^3\to[0,\infty]$ by
\baq
	\label{eq_def:alpha}
	\alpha_m(j,r,s) \coloneqq \begin{cases} \frac{(j_+)^p}{m^{p-1}(r,s)}, &(r,s)\in[0,R]\times[0,S],\\
	\infty, &\text{otherwise},
	\end{cases}
\eaq
where we use the conventions
\baqs
    a/b = \begin{cases} 0,&\text{if }a=b=0,\\ \infty,&\text{if }a\ne b=0. \end{cases}
\eaqs
In the case $R=S=\infty$, we define the recession function 
\baq
    \alpha_{m_\infty}(j,r,s)\coloneqq\lim_{\lambda\to\infty}\frac{1}{\lambda}\alpha_m(\lambda j,\lambda r,\lambda s)=\frac{(j_+)^p}{m_\infty^{p-1}(r,s)}.
\eaq
\end{definition}
\begin{remark}
\begin{enumerate}[label=(\roman*)]
    \item Observe that \eqref{eq_def:alpha} encodes an upwind structure as only the positive part of the flux enters the definition.
    
    \item By definition, $m_\infty$, and thus $\alpha_{m_\infty}$, are positively 1-homogeneous. In particular, for $r>0$ we have $m_\infty(r,s) = rm_\infty(1,s/r)$ and for $s>0$ we have $sm_\infty(r/s,1)$.
    
    \item The assumption $m(0,s)=0$ is crucial as it ensures the non-negativity of $\rho$ when the action is finite. However, note that this assumption excludes, among others, the choice $m\equiv 1$.
    
    \item If $m$ is an upwind-admissible mobility, the homogeneity implies $m_\infty(0,s)=0$.
    
    \item If $m_\infty(r,s)=0$, then $\alpha_{m_\infty}(j,r,s) = \infty$ for every $j\ne 0$. In particular, $\alpha_{m_\infty}(j,r,s)$ is the convex indicator of the set $\{j=0\}$ (it is zero if $j=0$ and infinite otherwise), if the growth of $m$ is uniformly sublinear.
\end{enumerate}
\end{remark}

To define the action density functional, we introduce the following notation.
\begin{definition}\label{def:transpose}
    The transpose of a Borel set $B\in\mathcal{B}(\RdRd)$ is denoted by $B^\top\coloneqq\{(y,x)\in\RdRd:(x,y)\in B\}$. The transpose of a measure $\nu\in\Mloc(\Rd\times\Rd)$ is then defined as $(\nu)^\top(B)\coloneqq \nu(B^\top)$.
    
    Given a pair of probability measures $\rhoup = (\rho^{(1)},\rho^{(2)})\in (\mathcal{P}(\Rd))^2$, abusing notation, we denote the Lebesgue decomposition $\dd\rho^{(i)} =\dd\rho^{(i)\mu}+\dd\rho^{(i)\perp}= \rho^{(i)}\dd\mu+\dd\rho^{(i)\perp}$ for $i=1,2$. Similarly, given a pair of fluxes $\jup \in (\Mloc(G))^2$, for $i=1,2$ we denote $\dd j^{(i)} = \dd j^{(i)\mu} + \dd j^{(i)\perp} = j^{(i)}\dd(\mu\otimes\mu)+\dd j^{(i)\perp}$.
\end{definition}
We emphasize that from now on, we always use non-italic, upright letters such as $\jup$, $\rhoup$, $\betaup$ etc. to indicate pairs of quantities indexed by superscript $(i)$, e.g. $\jup = (j^{(1)},j^{(2)})$.
\begin{definition}[Single species action density functional]\label{def:A}
Let $\mu\in \Mloc^+(\Rd)$, $\eta$ satisfy \eqref{W}. Now, given $p\in(1,\infty)$ and an upwind-admissible $m$, for any $(\rho,j)\in\mathcal{P}(\Rd)\times\Mloc(G)$, we define the single-species action density functional $\bar\A_m$ as follows:
\begin{enumerate}
    \item If $R=S=\infty$, for any $\sigma\in\Mloc^+(\Rd\times\Rd)$ such that $\sigma = \sigma^\top$, $\dd(\rho^\perp\otimes\mu)=\rho^{\perp}\otimes\tilde\mu\dd\sigma$, $\dd(\mu\otimes\rho^{\perp})=\tilde\mu\otimes\rho^{\perp}\dd\sigma$ and $\dd j^{\perp}= j^{\perp}\dd\sigma$, we define
    \baq\label{eq_def:A_bar_case:R=S=infty}
        \bar\A_m(\mu;\rho,j) \coloneqq&\, \frac{1}{2}\iint_G\Big[\alpha_m\left(j,\rho\otimes\mu,\mu\otimes\rho\right)+\alpha_m\left(-j,\mu\otimes\rho,\rho\otimes\mu\right)\Big]\eta \dd(\mu\otimes\mu)\\
        +&\,\frac{1}{2}\iint_G\Big[\alpha_{m_\infty}\left(j^\perp,\rho^\perp\otimes\tilde\mu,\tilde\mu\otimes\rho^\perp\right)+\alpha_{m_\infty}\left(-j^\perp,\tilde\mu\otimes\rho^\perp,\rho^\perp\otimes\tilde\mu\right)\Big]\eta \dd\sigma.
    \eaq
    
    \item If $R\land S<\infty$, similar to \cite{Lisini2010}, we define $\bar\A_m$ as follows:
    \baq\label{eq_def:A_bar_case:R^S<infty}
        \bar\A_m(\mu;\rho,j) \coloneqq\begin{cases}\begin{aligned}\frac{1}{2}\iint_G\Big[&\alpha_m\left(j,\rho\otimes\mu,\mu\otimes\rho\right)\\
        +\,&\alpha_m\left(-j,\mu\otimes\rho,\rho\otimes\mu\right)\Big]\eta \dd(\mu\otimes\mu),\end{aligned} &\text{if } \rho^\perp=0,j^\perp=0,\\
        \infty,&\text{otherwise}. \end{cases}
    \eaq
\end{enumerate}
Let $\betaup=(\beta^{(1)},\beta^{(2)})\in(0,\infty)\times(0,\infty)$ be a pair of positive constants. Then, for any $\rhoup \in (\mathcal{P}(\Rd))^2$, $\jup \in (\Mloc(\Rd))^2$, the two-species action density functional is defined by 
\baq\label{eq_def:A}
    \A_{m, \betaup}(\mu;\rhoup,\jup)\coloneqq\frac{1}{\beta^{(1)}}\bar\A_m(\mu;\rho^{(1)},j^{(1)})+\frac{1}{\beta^{(2)}}\bar\A_m(\mu;\rho^{(2)},j^{(2)}).
\eaq
\end{definition}
\begin{remark}\label{rem:properties_A}
\begin{enumerate}[label=(\roman*)]
    \item $\A_{m, \betaup}$ is well-defined which is clear if $R\land S<\infty$. If $R=S=\infty$, the definition $\bar\A_m$ is independent of the particular choice of $\sigma$ due to the positive 1-homogeneity of $\alpha_{m_\infty}$. An admissible $\sigma$ can always be constructed, e.g. by setting
    \baqs
        \sigma = \big(\rho^\perp\otimes\mu+\rho^\perp\otimes\mu+\big|j^\perp\big|+\big|(j^\perp)^\top\big|\big),
    \eaqs
    for $i=1,2$. The definition may be adapted such that the symmetry assumption $\sigma=\sigma^\top$ is not required. However, since this assumption is nonrestrictive and simplifies notation, we apply it throughout.

    \item If $R=S=\infty$ and $m_\infty\equiv0$, the definition of the single-species action density functional \eqref{eq_def:A_bar_case:R=S=infty} simplifies to
    \baqs
        \bar \A_m(\mu;\rho,j) \coloneqq\begin{cases}\begin{aligned}\frac{1}{2}\iint_G\Big[&\alpha_m\left(j,\rho\otimes\mu,\mu\otimes\rho\right) +\alpha_m\left(-j,\mu\otimes\rho,\rho\otimes\mu\right)\Big]\eta \dd(\mu\otimes\mu),\end{aligned} &\text{if } j^\perp=0,\\
        \infty,&\text{otherwise}. \end{cases}
    \eaqs
    
    \item If $R\land S<\infty$, finiteness of the action density implies both $\rho\ll\mu$ as well as $j\ll\mu\otimes\mu$.
    
    \item The two-species action density functional is fully decoupled with respect to the different species.

\end{enumerate}
\end{remark}

Next, we want to show an antisymmetry property of the action density. To do this, we introduce the following notation. 
\begin{definition}[Antisymmetric velocities and fluxes]\label{def:V^as}
We define the set of \emph{antisymmetric velocities} by
\baqs
\mathcal{V}^\mathrm{as}(G) \coloneqq \big\{& (v,v^{\perp}):G\rightarrow\R^2, \,\mathrm{s.t.}\,
v = -v^\top \text{ and }(v^{\perp})^\top = -(v^{\perp})^\top \big\}.
\eaqs
For $\mu\in \Mloc^+(\Rd)$ and $\rho\in \mathcal{P}(\Rd)$ we define $\varsigma\in\Mloc^+(\Rd)$ as
\baq\label{eq_def:varsigma}
    \varsigma=\rho^\perp\otimes\mu+\mu\otimes\rho^\perp,
\eaq
and denote the corresponding densities by $\dd(\rho^{\perp}\otimes\mu) \eqqcolon \rho^{\perp}\otimes\tilde\mu\dd\varsigma^{}$ and $\dd(\mu\otimes\rho^{\perp}) \eqqcolon \tilde\mu\otimes\rho^{\perp}\dd\varsigma^{}$. With these densities, we further shorten notation by introducing
\baq\label{eq_def:mathfrak(m)}
    \mathfrak{m}(x,y) &\coloneqq m(\rho(x),\rho(y)),\\
    \mathfrak{m}_\infty(x,y) &\coloneqq m_\infty(\tilde\mu(x)\rho^{\perp}(y),\rho^{\perp}(x)\tilde\mu(y)).
\eaq
With this, recalling $q=p/(p-1)$, for $k=1,2$, we define
\baqs
    &\dd\gamma_1(x,y)\coloneqq (\mathfrak{m}(x,y))^{q-1}\dd\mu(x)\dd\mu(y),\quad 
    &&\dd\gamma_2(x,y)\coloneqq (\mathfrak{m}(y,x))^{q-1}\dd\mu(x)\dd\mu(y),\\
    &\dd\gamma_{1}^\perp(x,y)\coloneqq (\mathfrak{m}_\infty(x,y))^{q-1}\dd\varsigma(x,y),\quad
    &&\dd\gamma_{2}^\perp(x,y)\coloneqq (\mathfrak{m}_\infty(y,x))^{q-1}\dd\varsigma(x,y).
\eaqs
These measures satisfy $\gamma_1^\top=\gamma_2$ and vice versa as well as $(\gamma_1^\perp)^\top=\gamma_2^\perp$ and vice versa. By \eqref{W}, this does not change, when multiplied by $\eta$. Hence, it makes sense to define the following set of \emph{antisymmetric fluxes}:
\baqs
    \mathcal{M}^\mathrm{as}_{\eta\gamma_1}(G) \coloneqq \big\{j \in \Mloc(G)\;:\; &j_+ \ll \eta\gamma_1,
    j_- = (j_+)^\top, j^\perp_+ \ll \eta\gamma_1^\perp,
    j^\perp_- = (j^\perp_+)^\top
    \big\}.
\eaqs

\end{definition}
\begin{remark}\label{rem:V^as}
    Any $j\in \mathcal{M}^\mathrm{as}_{\eta\gamma_1}(G)$ satisfies $j_- \ll \eta\gamma_2$ and $j^\perp_- \ll \eta\gamma_2^\perp$.
    
    In the sequel, any indices attached to $\rho$ will be passed to $\mathfrak{m}$, $\mathfrak{m}_\infty$, $\gamma_k$, $\gamma_k^\perp$ and $\varsigma$, i.e. when  $\rho=\rho^{n,(i)\perp}_t$ in \eqref{eq_def:mathfrak(m)}, we write
    \baqs
        \mathfrak{m}_{\infty,t}^{n,(i)}(x,y)\coloneqq m_\infty(\tilde\mu(x)\rho_t^{n,(i)\perp}(y),\rho_t^{n,(i)\perp}(x)\tilde\mu(y)),
    \eaqs
    and similarly for the other expressions.
\end{remark}

We are now in the position to establish a connection between fluxes and velocities.
\begin{lemma}[Dual representation]\label{lem:connection_flux_velocity}
\sloppy Let $\mu \in \Mloc^+(\Rd)$, $\rho$ and $j$ be such that $\A_{m, \betaup}(\mu;\rhoup,\jup)<\infty$. Then, there exists a pair of measurable functions $(v^{(1)},v^{(2)}):G\to\R^2$ such that for $i=1,2$ we have
\baq\label{eq:dj_mit_v}
\dd j^{(i)\mu} &= (v^{(i)}_+)^{q-1}\dd\gamma_1^{(i)}-(v^{(i)}_-)^{q-1}\dd\gamma_2^{(i)}.
\eaq
Further, $\varsigma$ from Definition \ref{def:V^as} is an admissible choice for $\sigma$ in Definition \ref{def:A}, i.e. $j^{(i)\top}\ll\varsigma^{(i)}$, and there exists another pair of measurable functions $(v^{(1)\perp},v^{(2)\perp}):G\to\R^2$ (which is zero if $R\land S<\infty$), such that for $i=1,2$ we have
\baq\label{eq:dj^perp_mit_v}
\dd j^{(i)\perp} &= (v^{(i)\perp}_+)^{q-1}\dd\gamma_1^{(i)\perp}-(v^{(i)\perp}_-)^{q-1}\dd\gamma_2^{(i)\perp}.
\eaq
We can rewrite the action density in terms of $\vup = (v^{(1)},v^{(1)\perp},v^{(2)},v^{(2)\perp})$ as
\baq\label{eq_def:A_tilde}
\A_{m, \betaup}(\mu;\rhoup,\jup) &= \frac{1}{2}\sum_{i=1}^2\frac{1}{\beta^{(i)}}\Bigg[\iint_G \mathfrak{m}^{(i)}\left(\Big(v_+^{(i)}\Big)^q+\Big(\big(v_-^{(i)}\big)^\top\Big)^q\right)\eta\dd\mu\otimes\mu\\ 
&\hphantom{=\frac{1}{2}\sum_{i=1}^2\frac{1}{\beta^{(i)}}\Bigg[}+\iint_G \mathfrak{m}^{(i)}_\infty\left(\Big(v_+^{(i)\perp}\Big)^q+\Big(\big((v_-^{(i)\perp}\big)^\top\Big)^q\right)\eta\dd\varsigma^{(i)}\Bigg]\\
&\eqqcolon \tilde{\A}_{m, \betaup}(\mu;\rhoup,\vup),
\eaq
while for $R\land S<\infty$ we have $\rhoup^\perp=0$ and $\jup^\perp =0$. In particular, if $\vup\in (\mathcal{V}^\mathrm{as}(G))^2$, we have
\baq\label{eq:A mit antisymm v}
    \tilde{\A}_{m, \betaup}(\mu;\rhoup,\vup)&=\sum_{i=1}^2\frac{1}{\beta^{(i)}}\Bigg[\iint_G \mathfrak{m}^{(i)}\Big(v_+^{(i)}\Big)^q\eta\dd\mu\otimes\mu+\iint_G \mathfrak{m}^{(i)}_\infty\Big(v_+^{(i)\perp}\Big)^q\eta\dd\varsigma^{(i)}\Bigg].
\eaq
\end{lemma}
\begin{proof}
Assume $R=S=\infty$. The case $R\land S<\infty$ then works similarly. First, we show that the finiteness of the action implies that $j^{(i)\perp}_+\ll\gamma_1^{(i)\perp}$ and $j^{(i)\perp}_-\ll\gamma_2^{(i)\perp}$. Indeed, assume there exists a set $B\in\mathcal{B}(G)$ such that $j^{(i)\perp}_+(B) > 0 = \gamma_1^{(i)\perp}(B)$. By Remark \ref{rem:properties_A} $\sigma^{(i)}=\gamma_1^{(i)\perp}+\gamma_2^{(i)\perp}+|j^{(i)\perp}|$ is admissible in \eqref{eq_def:A_bar_case:R=S=infty} and $\sigma^{(i)}(B)>0$. Since we have
\baqs
    m_\infty^{q-1}\left(\frac{\dd(\rho^{(i)\perp}\otimes\mu)}{\dd\sigma^{(i)}},\frac{\dd(\mu\otimes\rho^{(i)\perp})}{\dd\sigma^{(i)}}\right)=m_\infty^{q-1}\left(\frac{\dd(\mu\otimes\rho^{(i)\perp})}{\dd\sigma^{(i)}},\frac{\dd(\rho^{(i)\perp}\otimes\mu)}{\dd\sigma^{(i)}}\right)=m_\infty^{q-1}(0,0)=0
\eaqs   
$\sigma^{(i)}$-a.e. in $B$, we obtain a contradiction to $\A_{m, \betaup}(\mu;\rhoup,\jup)<\infty$. Similar arguments hold true for $j^{(i)\perp}_-$ as well as $j^{(i)}_+$ and $j^{(i)}_-$. Therefore, the nonnegative functions $v^{(i)}_+$ and $v^{(i)}_-$ satisfying \eqref{eq:dj_mit_v} are well-defined $\gamma_1^{(i)}$-a.e. and $\gamma_2^{(i)}$-a.e., respectively, which gives us $v^{(i)}=v^{(i)}_+-v^{(i)}_-$. Similarly, we obtain $v^{(i)\perp}$ satisfying \eqref{eq:dj^perp_mit_v}. Finally, \eqref{eq_def:A_tilde}, follows, when we insert
\baqs
    \dd j^{(i)\mu}(x,y) &=  \Big[\left(\mathfrak{m}^{(i)}(x,y)v^{(i)}_+(x,y)\right)^{q-1}-\left(\mathfrak{m}^{(i)}(y,x)v^{(i)}_-(x,y)\right)^{q-1}\Big]\dd\mu(x)\dd\mu(y)
\eaqs
and
\baqs
    \dd j^{(i)\perp}(x,y) &=  \Big[\left(\mathfrak{m}^{(i)}_\infty(x,y)v^{(i)\perp}_+(x,y)\right)^{q-1}-\left(\mathfrak{m}^{(i)}_\infty(y,x)v^{(i)\perp}_-(x,y)\right)^{q-1}\Big]\dd\varsigma^{(i)}(x,y)
\eaqs
into \eqref{eq_def:A_bar_case:R=S=infty}.
\end{proof}
\begin{definition}\label{def:alpha_tilde}
    In light of \eqref{eq_def:A_tilde}, we define
    \baq
	    \label{eq_def:alpha_tilde}
	    \tilde\alpha_m(v,r,s) \coloneqq \begin{cases} m(r,s)(v_+)^q, &(r,s)\in[0,R]\times[0,S],\\
	    \infty, &\text{otherwise},
	    \end{cases}
    \eaq
which gives us a representation of $\tilde{\A}_{m, \betaup}$ similar to the one in Definition \ref{def:A}, only replacing $j$ with $v$ and $\alpha$ with $\tilde\alpha$.
\end{definition}

We observe that antisymmetric fluxes admit lower action densities while preserving their nonlocal divergence.
\begin{corollary}[Antisymmetric vector fields have lower action density]\label{cor_as->lower_action}
Let $\mu \in \Mloc^+(\Rd)$, $\rhoup \in (\mathcal{P}(\Rd))^2$ and $\jup \in(\Mloc(G))^2$ s.t. $\A_{m, \betaup}(\mu;\rhoup,\jup)< \infty$. Then, there exists $\bar{\varjup}\in(\mathcal{M}^\mathrm{as}_{\eta\gamma_1}(G))^2$ such that
\baqs
\babla\cdot j^{(i)} &= \babla\cdot\bar{\jmath}^{(i)},\qquad i=1,2,\\
\eaqs
and
\baqs
\A_{m, \betaup}(\mu;\rhoup,\bar{\varjup}) \leq \A_{m, \betaup}(\mu;\rhoup,\jup).
\eaqs
\end{corollary}
\begin{proof}
Analogous to \cite[Corollary 2.8]{EPSS2021}, one defines $\dd\bar{\jmath}^{(i)}\coloneqq\dd\big( j^{(i)}-(j^{(i)})^\top\big)/2$, compares \eqref{eq_def:A_tilde} with \eqref{eq:A mit antisymm v}, and applies Jensen's inequality.
\end{proof}
Next, we establish properties of the action density as well as important bounds for the subsequent analysis.
\begin{lemma}[Lower semicontinuity of the action density]\label{lem:lsc A}
The action is lower semicontinuous with respect to weak-$\ast$ convergence in $\Mloc^+(\Rd)\times(\Mloc^+(\Rd))^2\times(\Mloc(G))^2$. That is, for $\mu^n \rightharpoonup^\ast \mu$ in $\Mloc^+(\Rd)$, $\rhoup_n \rightharpoonup^\ast \rhoup$ in $(\Mloc^+(\Rd))^2$, and $\jup^n \rightharpoonup^\ast \jup$ in $(\Mloc(G))^2$, we have
\baqs
\liminf_{n\rightarrow\infty} \A_{m, \betaup}(\mu^n;\rhoup^n,\jup^n) \geq \A_{m, \betaup}(\mu;\rhoup,\jup).
\eaqs
\end{lemma}
\begin{proof}
See \cite[Theorem 2.34]{Ambrosio2000FunctionsOB}, while keeping in mind that $\mu^n\otimes\rho^{n,(i)}\rightharpoonup^\ast\mu\otimes\rho^{(i)}$ in $\Mloc(\Rd\times\Rd)$ if and only if both $\mu^n\rightharpoonup^\ast\mu$ in $\Mloc(\Rd)$ and $\rho^{n,(i)}\rightharpoonup^\ast\rho^{(i)}$ in $\Mloc(\Rd)$.
\end{proof}

\begin{lemma}\label{lem:bound_by_A}
Let $\mu \in \Mloc^+(\Rd)$, $\rhoup \in (\mathcal{P}(\Rd))^2$ and $\jup \in(\Mloc(G))^2$, such that $\A_{m, \betaup}(\mu;\rhoup,\jup)< \infty$. Then, there exists a constant $M=M(m,p,\betaup)>0$, such that for any measurable $\Phi:G\rightarrow\R_+$, it holds
\baq\label{eq:lem_bound_by_A}
\iint_G\Phi\eta\dd\abs{\jup} \le M \A_{m, \betaup}^{1/p}(\mu;\rhoup,\jup)\sum_{i=1}^2\left(\iint_G\Phi^q\eta \dd(\mu\otimes\mu+\rho^{(i)}\otimes\mu+\mu\otimes\rho^{(i)})\right)^{1/q}.
\eaq
\end{lemma}
\begin{proof}
First, let $R=S=\infty$ and $\sigma^{(i)}=\varsigma^{(i)}\in \Mloc^+(G)$, $i=1,2$ be as in Definition \ref{def:V^as}. Since 
$\A_{m, \betaup}(\mu;\rhoup,\jup)< \infty$, we have that
\baqs
    B \coloneqq \Big\{&\,(x,y)\in G: \sum_{i=1}^2\Big[\alpha_m\big(j^{(i)}_+(x,y),\rho^{(i)}(x),\rho^{(i)}(y)\big)+\alpha_m\big(j^{(i)}_-(x,y),\rho^{(i)}(y),\rho^{(i)}(x)\big)\\
    +&\,\alpha_{m_\infty}\big(j^{(i)\perp}_+(x,y),\rho^{(i)\perp}(x)\tilde\mu(y),\tilde\mu(x)\rho^{(i)\perp}(y)\big)+\alpha_{m_\infty}\big(j^{(i)\perp}_-(x,y),\tilde\mu(x)\rho^{(i)\perp}(y),\rho^{(i)\perp}(x)\tilde\mu(y)\big)\Big] = \infty\Big\}
\eaqs
is a $\varsigma^{(i)}$-nullset for $i=1,2$. By definition of $\alpha_m$, we have, $\varsigma^{(i)}$-a.e. in $B^c$, the inequality
\baq\label{eq:lem_bound_by_A_j-bound}
\left(j^{(i)}_+(x,y)\right)^p+\left(j^{(i)}_-(x,y)\right)^p 
\le&\,\max\Big\{m^{p-1}\big(\rho^{(i)}(y),\rho^{(i)}(x)\big), m^{p-1}\big(\rho^{(i)}(x),\rho^{(i)}(y)\big)\Big\}\\
\cdot&\left(\alpha_m\big(j^{(i)}_+,\rho^{(i)}(x),\rho^{(i)}(y)\big)+\alpha_m\big(j^{(i)}_-,\rho^{(i)}(y),\rho^{(i)}(x)\right)\big)\\
\le&\, \bar M[1+\rho^{(i)}(x)+\rho^{(i)}(y)]^{p-1}\\
\cdot&\left(\alpha_m\big(j^{(i)_+},\rho^{(i)}(x),\rho^{(i)}(y)\big)+\alpha_m\big(j^{(i)}_-,\rho^{(i)}(y),\rho^{(i)}(x)\big)\right),
\eaq
where $\bar M$ only depends on $m$. Indeed, such an $\bar M$ exists, since $m$ is concave and $m(0,s)=0$ by definition. Similarly, we have the bound
\baqs
    \left(j^{(i)\perp}_+(x,y)\right)^p+\left(j^{(i)\perp}_-(x,y)\right)^p 
    \le&\,\bar M[1+\rho^{(i)\perp}(x)\tilde\mu(y)+\tilde\mu(x)\rho^{(i)\perp}(y)]^{p-1}\\
    \cdot&\,\Big(\alpha_{m_\infty}\big(j^{(i)\perp}_+,\rho^{(i)\perp}(x)\tilde\mu(y),\tilde\mu(x)\rho^{(i)\perp}(y)\big)\\
    \hphantom{\cdot}&\,+\alpha_{m_\infty}\big(j^{(i)\perp}_-,\tilde\mu(x)\rho^{(i)\perp}(y),\rho^{(i)\perp}(x)\tilde\mu(y)\big)\Big).
\eaqs
By the complementarity of the positive and negative parts, this gives
\baqs
\abs{j^{(i)}}(x,y) &\leq \tilde{M}\left(1+\rho^{(i)}(x)+\rho^{(i)}(y)\right)^{1/q}\bigg(\frac{1}{\beta^{(i)}}\alpha_m\big(j^{(i)}_+,\rho^{(i)}(x),\rho^{(i)}(y)\big)\\
&\hphantom{\leq\tilde{M}\left(1+\rho^{(i)}(x)+\rho^{(i)}(y)\right)^{1/q}}+\frac{1}{\beta^{(i)}}\alpha_m\big(j^{(i)}_-,\rho^{(i)}(y),\rho^{(i)}(x)\big)\bigg)^{1/p},\\
\abs{j^{(i)\perp}}(x,y) &\leq \tilde{M}\left(1+\rho^{(i)\perp}(x)\tilde\mu(y)+\rho^{(i)\perp}(y)\tilde\mu(x)\right)^{1/q}\bigg(\frac{1}{\beta^{(i)}}\alpha_{m_\infty}\big(j^{(i)\perp}_+,\rho^{(i)\perp}(x)\tilde\mu(y),\tilde\mu(x)\rho^{(i)\perp}(y)\big)\\
&\hphantom{\leq\tilde{M}\left(\rho^{(i)\perp}(x)\tilde\mu(y)+\rho^{(i)\perp}(y)\tilde\mu(x)\right)^{1/q}}+\frac{1}{\beta^{(i)}}\alpha_{m_\infty}\big(j^{(i)\perp}_-,\tilde\mu(x)\rho^{(i)\perp}(y),\rho^{(i)\perp}(x)\tilde\mu(y)\big)\bigg)^{1/p},
\eaqs
where $\tilde{M}$ depends only on $m$, $p$ and $\betaup$. Since $|\jup| +|\jup^\perp| = \sum_{i=1}^2 (|j^{(i)}|+|j^{(i)\perp}|)$, these estimates together with H{\"o}lder's inequality yield 
\baqs
\iint_G\Phi\eta\dd\abs{\jup}=&\,\sum_{i=1}^2\Bigg(\iint_{B^c}\Phi\eta |j^{(i)}|\dd(\mu\otimes\mu)+\iint_{B^c}\Phi\eta |j^{(i)\perp}|\dd\varsigma^{(i)}\Bigg)\\
\le&\,4\tilde{M} \big(2\A_{m, \betaup}(\mu;\rhoup,\jup)\big)^{1/p}\sum_{i=1}^2\Bigg(\left(\iint_G\Phi^q\eta \dd(\mu\otimes\mu+\rho^{(i)\mu}\otimes\mu+\mu\otimes\rho^{(i)\mu})\right)^{1/q}\\
&\phantom{4\tilde{M} \A_{m, \betaup}^{1/p}(\mu;\rhoup,\jup)\sum_{i=1}^2}+\left(\iint_G\Phi^q\eta \dd(\varsigma^{(i)}+\rho^{(i)\perp}\otimes\mu+\mu\otimes\rho^{(i)\perp})\right)^{1/q}\Bigg).
\eaqs
Thus, recalling that $\varsigma^{(i)}=\rho^{(i)\perp}\otimes\mu+\mu\otimes\rho^{(i)\perp}$, we obtain \eqref{eq:lem_bound_by_A} with $M=16\cdot2^{1/p}\tilde{M}$. If $R\land S<\infty$, we argue similarly, only replacing in \eqref{eq:lem_bound_by_A_j-bound} $m$ by $m_\uparrow(r,s)\coloneqq\sup_{(\tilde{r}, \tilde{s}) \in [0,r]\times[0,s]}m(\tilde{r}, \tilde{s})$, which is still concave and satisfies $m_\uparrow(0,s)=0$.
\end{proof}
\begin{remark}\label{rem:bound_by_A_refined}
If there exists $C > 0$ such that $m$ satisfies 
\begin{equation*}
    \tag{M} m(r,s) \le C(r+s)\qquad \forall r,s\in [0,R)\times[0,S),
\end{equation*}
then \eqref{eq:lem_bound_by_A} can be replaced by the refined bound
\baq\label{eq:lem_bound_by_A_refined}
\iint_G\Phi\eta\dd\abs{\jup} \le M \A_{m, \betaup}^{1/p}(\mu;\rhoup,\jup)\sum_{i=1}^2\left(\iint_G\Phi^q\eta \dd(\rho^{(i)}\otimes\mu+\mu\otimes\rho^{(i)})\right)^{1/q}.
\eaq
Indeed, observe that in this case the summand $1$ on the right-hand side \eqref{eq:lem_bound_by_A_j-bound} can be omitted, which leads to dropping the integral with respect to $\mu\otimes\mu$. It is straightforward to check that replacing \eqref{eq:lem_bound_by_A} by \eqref{eq:lem_bound_by_A_refined}, whenever it is employed in the sequel, allows to drop the assumption \eqref{MB2} altogether.
\end{remark}
\begin{corollary}\label{cor:bound_by_A}
Let $\mu \in \Mloc^+(\Rd)$, $\rhoup \in (\mathcal{P}(\Rd))^2$ and $\jup \in(\Mloc(G))^2$, such that $\A_{m, \betaup}(\mu;\rhoup,\jup)< \infty$. Then, for $\Phi_1(x,y) =2\land\abs{x-y}$ and $\Phi_2=\abs{x-y}\lor\abs{x-y}^p$, we have
\baqs
\iint_G \Phi_k\eta\dd\abs{\jup} \leq M C_\eta^{1/q}\A_{m, \betaup}^{1/p}(\mu;\rhoup,\jup),\qquad k=1,2,
\eaqs
where $M=M(m,p,\betaup)$ is different from that in Lemma \ref{lem:bound_by_A}.
\end{corollary}
\begin{proof}
Note that $\Phi_1(x,y) \leq \abs{x-y} \leq \Phi_2(x,y)$. Therefore, Lemma \ref{lem:bound_by_A} yields for $k =1,2$
\baqs
&\,\iint_G\Phi_k(x,y)\eta(x,y)\dd{\abs{\jup}}(x,y)\\
\le&\, \bar{M} \A_{m, \betaup}^{1/p}(\mu;\rhoup,\jup)\sum_{i=1}^2\left(\iint_G\Phi_k^q(x,y)\eta(x,y) \dd\big(\mu\otimes\mu+\rho^{(i)}\otimes\mu+\mu\otimes\rho^{(i)}\big)(x,y)\right)^{1/q}\\
=&\, \bar{M} \A_{m, \betaup}^{1/p}(\mu;\rhoup,\jup)\sum_{i=1}^2\left(\iint_G{\abs{x-y}}^q\lor{\abs{x-y}}^{pq}\eta(x,y) \dd\big(\mu\otimes\mu+2\rho^{(i)}\otimes\mu\big)(x,y)\right)^{1/q}\\
\le&\, \bar{M} \A_{m, \betaup}^{1/p}(\mu;\rhoup,\jup) 2(C_\mu+2)^{1/q}C_\eta^{1/q},
\eaqs
where we used \eqref{MB} together with \eqref{MB2} and the fact that $\rho^{(i)}\in\mathcal{P}(\Rd)$ for $i=1,2$.
\end{proof}
\begin{lemma}[Convexity of the action]\label{lem:convexity of the action}
Let $\mu_0,\mu_1 \in \Mloc^+(\Rd)$, $\rhoup_0,\rhoup_1 \in (\mathcal{P}(\Rd))^2$ and $\jup_0,\jup_1 \in(\Mloc(G))^2$. For $\tau \in (0,1)$ define $\mu_\tau = (1-\tau)\mu_0+\tau\mu_1$, $\rhoup_\tau = (1-\tau)\rhoup_0+\tau\rhoup_1$ and $\jup_\tau = (1-\tau)\jup_0+\tau \jup_1$. Then, we have
\baqs
\A_{m, \betaup}(\mu_\tau;\rhoup_\tau,\jup_\tau) \leq (1-\tau)\A_{m, \betaup}(\mu^0;\rhoup_0,\jup_0)+\tau\A_{m, \betaup}(\mu_1;\rhoup_1,\jup_1),
\eaqs
\end{lemma}
\begin{proof}
This immediately follows from the convexity of $\alpha_m$ and $\alpha_{m_\infty}$. A detailed argument for one species and $m(r,s)=r$, which upon small adjustments is also applicable here,  can be found in \cite[Lemma 2.12]{EPSS2021}.
\end{proof}

\subsection{Generalized continuity equation and properties}

In this subsection we study the nonlocal continuity equation and properties of its solutions with finite action. 
\begin{definition}[Continuity equation]\label{def:cont eq} We say that the pair $(\bs\rhoup,\jupbold)=((\rhoup_t)_{t\in[0,T]},(\jup_t)_{t\in[0,T]})$ 
with $\rhoup_t \in (\mathcal{P}(\Rd))^2$ and $\jup_t \in (\Mloc(G))^2$, 
is a weak solution of the continuity equation
\baqs
\partial_t \rhoup + \babla\cdot \jup_t = 0 \text{ on }(0,T)\times\Rd,
\eaqs
if we have 
\begin{enumerate}[label=(\roman*)]
\item $\bs\rhoup$ is a weakly continuous curve in $(\mathcal{P}(\Rd))^2$
\item $\jupbold$ is a Borel-measurable curve in $(\Mloc(G))^2$
\item For any $\varphi \in C^\infty_c(\Rd\times(0,T))$ and $i=1,2$, we have
\baq\label{eq:cont}
\int_0^T\int_\Rd \partial_t\varphi_t(x)\dd\rho_t^{(i)}(x) \dd t+\frac{1}{2}\int_0^T\iint_G\babla\varphi_t(x,y)\eta(x,y)\dd j_t^{(i)}(x,y)\dd t &= 0.
\eaq
\end{enumerate}
We denote the set of all weak solutions on the time interval $[0,T]$ by $\CE_T$. For $\rhoinit, \rhofin \in (\mathcal{P}(\Rd))^2$, we write $(\bs\rhoup,\jupbold) \in \CE_T(\rhoinit,\rhofin)$ if $(\bs\rhoup,\jupbold) \in \CE_T$ and, additionally, $\rho_0 =\rhoinit, \rho_T=\rhofin$. We will often shorten notation and write $\CE\coloneqq \CE_1$.
\end{definition}
We make the following observations:
\begin{remark}\label{rem:eq cont}
\begin{enumerate}[label=(\roman*)]
\item Since $|\babla\varphi(x,y)|\leq \norm{\varphi}_{C^1(\Rd)}2\land\abs{x-y}$, the continuity equation is well-defined under the integrability condition
\baq
    \label{eq:integrability}
	\int_0^T\iint_G 2\land\abs{x-y}\eta(x,y)\dd |j^{(i)}_t| (x,y)\dd t < \infty,\qquad\text{for }i=1,2.
\eaq
By Corollary \ref{cor:bound_by_A}, this condition is satisfied for any pair $(\bs\rhoup,\jupbold)$ with $\int_0^T \A_{m, \betaup}(\mu;\rhoup_t,\jup_t)\dd t <\infty$.
\item The continuity equation holds for more general test functions. Indeed, regularizing via convolution, we immediately see that \eqref{eq:cont} also holds for $\varphi \in C^1_c(\Rd\times(0,T))$.  Under the integrability condition \eqref{eq:integrability} we can also consider bounded test functions $\varphi\in C^1_b(\Rd\times(0,T))$, whose support has a compact projection in $(0,T)$. To see this, we approximate $\varphi$ by $\varphi\chi_R$, where $\chi_R\in C_c^\infty(\Rd)$, $0\leq \chi_R \leq 1$ and $\chi_R\equiv 1$ on $B_R(0)$.
\item The continuity equation is decoupled with respect to the different components $i\in\{1,2\}$ of $\bs\rhoup$ and $\jupbold$.
\end{enumerate}
\end{remark}
Since both the action density functional $\A_{m, \betaup}$ as well as the continuity equations are fully decoupled with respect to the different species, previous remarks yield analogues of \cite[Lemma 2.15, Lemma 2.16 and Proposition 2.17]{EPSS2021} for the two-species case. 
\begin{lemma}\label{lem:continuous representative}
Let $\bs\rhoup$ and $\jupbold$ be Borel families of measures in $(\mathcal{P}(\Rd))^2$ and $(\Mloc(\Rd))^2$ satisfying \eqref{eq:cont} and \eqref{eq:integrability}. Then, there exist weakly continuous curves $\bar{\bs\rhoup}\subset(\mathcal{P}(\Rd))^2$ such that $\bar{\rho}^{(i)}_t = \rho^{(i)}_t$ for a.e. $t\in[0,T]$ and $i=1,2$. Moreover, for any $\varphi \in C^1_b([0,T]\times\Rd)$ and any $0\leq s \leq t \leq T$ and $i=1,2$ it holds
\baq\label{eq:continuous representative}
\int_\Rd \varphi_{t}(x)\dd\bar{\rho}^{(i)}_{t}(x)-\int_\Rd \varphi_{s}(x)\dd\bar{\rho}^{(i)}_{s}(x) &= \int_{s}^{t}\int_\Rd \partial_t\varphi_t(x)\dd\rho^{(i)}_t(x)\dd t\\
&+\frac{1}{2}\int_{s}^{t}\iint_G \babla\varphi_t(x,y)\eta(x,y)\dd j^{(i)}_t(x)\dd t.
\eaq
\end{lemma}
\begin{proof}
This an adaptation of \cite[Lemma 8.1.2]{AmbrosioGigliSavare2008}. The required estimate on the time derivatives $\partial_t\rho_t^{(i)}$ is provided by Corollary \ref{cor:bound_by_A} as described in Remark \ref{rem:eq cont} (i). Finally, similar to Remark \ref{rem:eq cont} (ii), we can lower the regularity and compactness assumptions on the test functions $\varphi$.
\end{proof}
\begin{lemma}[Time-uniformly bounded $p$-th moments]\label{lem:Time-uniformly_bounded_moments}
Let $(\mu^n)_{n\in\mathbb{N}}\subset\Mloc^+(\Rd)$ satisfy \eqref{MB2} and \eqref{MB} uniformly in $n$. Let $(\rho_0^{n,(i)})_{n\in\mathbb{N}} \subset\mathcal{P}(\Rd)$ be such that $\sup_{n\in\mathbb{N}} M_p(\rho_0^{n,(i)})<\infty$ and let $(\bs\rhoup^{n},\jupbold^{n})_{n\in\mathbb{N}} \subset\CE_T$
be such that $\sup_{n\in\mathbb{N}} \int_0^T\A_{m, \betaup}(\mu^n;\rhoup^n_t,\jup^n_t)\dd t< \infty$. Then, we have
\baqs
\sup_{n\in\mathbb{N}}\sup_{t\in[0,T]} M_p(\rho_t^{n,(i)})<\infty.
\eaqs
\end{lemma}
\begin{proof}
We argue similarly to \cite{EPSS2021}. Since $|x|^p$ is not admissible in \eqref{eq:cont}, we introduce a smooth cut-off $\varphi_R\in C_c^\infty(\Rd;[0,1])$ satisfying $\varphi_R|_{B_R(0)}\equiv 1$, $\supp\varphi_R\subset B_{2R}(0)$ and $\abs{\nabla\varphi_R}\le 2/R$. Then, we define $\psi_R(x)\coloneqq \varphi_R^p(x)(1+|x|)^p$, which is admissible in \eqref{eq:cont}, giving us
\baqs
    \frac{\dd}{\dd t}\sum_{i=1}^2\int_\Rd \psi_R\dd\rho_t^{n,(i)}&=-\frac{1}{2}\sum_{i=1}^2\iint_G\babla\psi_R\eta\dd j_t^{n,(i)}\\
    &\le \frac{1}{2}M\A_{m, \betaup}^{1/p}(\mu;\rhoup_t,\jup_t)\sum_{i=1}^2\left(\iint_G\abs{\babla\psi_R}^q\eta \dd\big(\mu\otimes\mu+\rho_t^{n,(i)}\otimes\mu+\mu\otimes\rho_t^{n,(i)}\big)\right)^{1/q},
\eaqs
where we used Lemma \ref{lem:bound_by_A} in the last step. To bound the right-hand side by the moment, we shorten notation by introducing $\zeta\coloneqq\varphi_R(x)|x|$ and $xi\coloneqq\varphi_R(y)\abs{y}$, and calculate
\baqs
    |\babla\psi_R(x,y)|^q &= |\varphi_R^p(x)(1+|x|)^p-\varphi_R^p(y)(1+\abs{y})^p|^q\le C_1 \left(|\varphi_R^p(x)-\varphi_R^p(y)|^q+\abs{\zeta^p-\xi^p}^q\right).
\eaqs
Since $\varphi_R\le 1$ and $\abs{\nabla\varphi_R}\le 2/R$, we have 
\baqs
    |\varphi_R^p(x)-\varphi_R^p(y)|^q\le C_2\abs{\varphi_R(x)-\varphi_R(y)}^q\le 2^q C_2/R^q\abs{x-y}^q\le C_2\abs{x-y}^q,
\eaqs
for every $R\ge 2$ and where we can always choose $C_2\le\lceil p\rceil$. To bound the second term we employ the mean value theorem for the function $z\mapsto z^p$ and obtain
\baqs
    \abs{\zeta^p-\xi^p}^q\le \big(p\abs{\zeta-\xi}(\zeta+\xi)^{p-1}\big)^q = p^q\abs{\zeta-\xi}^q(2r-\zeta+\xi)^p \le C_3\big(\abs{\zeta-\xi}^{pq}+\abs{\zeta-\xi}^q \zeta^p\big).
\eaqs
Since $x\mapsto\varphi(x)|x|$ is globally Lipschitz, there exists a constant $C_4>0$ independent of $R$, such that
\baqs
    |\babla\psi_R(x,y)|^q \le C_4(1+|x|^p)\left(\abs{x-y}^{q}\lor\abs{x-y}^{pq}\right).
\eaqs
Thus, sending $R\to\infty$ and using \eqref{MB2} as well as \eqref{MB}, we find
\baqs
    \frac{\dd}{\dd t}\sum_{i=1}^2\int_\Rd (1+|x|^p)\dd\rho_t^{n,(i)}(x)&\le MC_4C_\eta^{1/q}\A_{m, \betaup}^{1/p}(\mu;\rhoup_t,\jup_t)\sum_{i=1}^2\left(\frac{1}{2}C_\mu+\int_\Rd(1+|x|^p) \dd\rho_t^{n,(i)}(x)\right)^{1/q}\\
    &\le \frac{1}{2}(C_\mu+2)MC_4C_\eta^{1/q}\A_{m, \betaup}^{1/p}(\mu;\rhoup_t,\jup_t)\sum_{i=1}^2\left(\int_\Rd(1+|x|^p) \dd\rho_t^{n,(i)}(x)\right)^{1/q}.
\eaqs
Integrating this inequality in time, estimating the $p$-th power of the sum by the sum of $p$-th powers and using H{\"o}lder's inequality, we arrive at
\baqs
    \sum_{i=1}^2\int_\Rd (1+|x|^p)\dd\rho_t^{n,(i)}(x)&\le  C\sum_{i=1}^2\int_\Rd (1+|x|^p)\dd\rho_0^{n,(i)}(x)+CT^{p-1}\int_0^T\A_{m, \betaup}(\mu;\rhoup,\jup)\dd t,
\eaqs
for some constant $C>0$. Taking the supremum over $n\in\mathbb{N}$ and $t\in[0,T]$ finishes the proof.
\end{proof}
\begin{proposition}[Compactness of solutions to the nonlocal continuity equation]\label{prop:compactness}
Let $(\mu^n)_{n\in\mathbb{N}}\subset\Mloc^+(\Rd)$ and suppose that $(\mu^n)_{n\in\mathbb{N}}$ weakly-$^\ast$ converges to $\mu\in\Mloc^+(\Rd)$. Moreover, assume that the base measures $\mu^n$ and $\mu$ satisify \eqref{MB2}, \eqref{MB} and \eqref{BC} uniformly in $n$. Let $((\bs\rhoup^{n},\jupbold^{n}))_{n\in\mathbb{N}} \subset\CE_T$ be such that $\sup_{n\in\mathbb{N}} M_p(\rho_0^{n,(i)})<\infty$ and $\sup_{n\in\mathbb{N}} \int_0^T\A_{m, \betaup}(\mu^n;\rhoup^n_t,\jup^n_t)\dd t< \infty$. Then, there exists  $(\bs\rhoup,\jupbold) \subset\CE_T$ such that, up to a subsequence (still indexed by $n$), as $n\to\infty$, it holds
\baqs
&\rhoup^{n}_t \rightharpoonup\rhoup_t\qquad\text{narrowly for all }t\in[0,T],\\
&\jupbold^{n,}  \rightharpoonup^\ast \jupbold\qquad\text{in }(\Mloc(G\times[0,T]))^2.
\eaqs
Moreover, the action is lower semicontinuous along the above subsequences 
i.e., we have
\baqs
\liminf_{n\to\infty}\int_0^T\A_{m, \betaup}(\mu^n_t;\rhoup^n_t,\jup^n_t)\dd t \geq \int_0^T\A_{m, \betaup}(\mu_t;\rhoup_t,\jup_t)\dd t.
\eaqs
\end{proposition}
\begin{proof}
We argue similarly to \cite[Proposition 2.17]{EPSS2021} and present only the main steps. At first, we employ Lemma \ref{lem:Time-uniformly_bounded_moments}, Corollary \ref{cor:bound_by_A}, H{\"o}lder's inequality, Assumption \eqref{W} and the disintegration theorem (see e.g. \cite[Theorem 5.3.1]{AmbrosioGigliSavare2008}) to obtain a subsequence still denoted by $(\jup^n_t)_n$ which weakly-$^\ast$ converges to a Borel family $(\jup_t)_{t\in[0,T]}$ such that for $i=1,2$ we have $j^{(i)}(K\times I) =\int_I j^{(i)}_t(K)\dd t$ as well as \eqref{eq:integrability} for any compact sets $I\subset [0,T]$, $K\subset G$. Then, for $0\leq s\leq t\leq T$ and $\varphi \in C_c^\infty(\Rd)$, we obtain the equality
\baqs
\lim_{n\to\infty} \int_{s}^{t}\iint_G \babla \varphi(x,y)\eta(x,y)\dd j_t^{n,(i)}(x,y)\dd t = \int_{s}^{t}\iint_G \babla \varphi(x,y)\eta(x,y)\dd j_t^{(i)}(x,y)\dd t
\eaqs
by employing a truncation argument and using the Assumptions \eqref{MB2}, \eqref{MB}, \eqref{BC} and \eqref{W} as well as H\"older's inequality and Lemma \ref{lem:bound_by_A}. Since $(\rho_0^{n,(i)})_n$ has uniformly bounded $p$-th moments, it is uniformly tight. Hence, Prokhorov's theorem (see e.g. \cite[Theorem 5.1.3]{AmbrosioGigliSavare2008}), the above convergence result and \eqref{eq:continuous representative}, where we choose $\varphi_{t} = \xi$, yield local narrow convergence of $(\rho_t^{n,(i)})_n$ to some $\rho_t^{(i)}\in \Mloc^+(\Rd)$ for $i=1,2$. In the last step we use \eqref{eq:integrability} and Corollary \ref{cor:bound_by_A} to find that $\rho_t^{(i)}\in\mathcal{P}(\Rd)$ and employ Lemma \ref{lem:Time-uniformly_bounded_moments} to obtain that the narrow convergence of $(\rho_t^{n,(i)})_n$ towards $\rho_t^{(i)}$ is in fact global. Therefore, we have $(\bs\rhoup,\jupbold)\in \CE_T$. Finally, since narrow convergence implies weak-$^\ast$ convergence, Lemma \ref{lem:lsc A} shows the claim of lower semicontinuity of the action.
\end{proof}
\begin{remark}\label{rem:compactness_mu}
    Note that in Proposition \ref{prop:compactness} we have compactness of the action density not only in $\rhoup_t$ and $\jup_t$, but also in the base measure $\mu_t$. This will play a crucial role later in the proof of existence.
\end{remark}
\subsection{Definition of a quasimetric}\label{Quasimetric}
Having defined an action density and a continuity equation, we are now ready to define the induced quasimetric:
\begin{definition}[Nonlocal upwind transportation cost for two species]\label{def:T_beta,mu}
For $\mu \in \Mloc^+(\Rd)$, $\eta$ satisfying \eqref{MB2}, \eqref{MB}, \eqref{BC}, and $\rhoinit, \rhofin \in (\mathcal{P}(\Rd))^2$, the nonlocal upwind transportation cost between  $\rhoinit$ and $\rhofin$ is defined as
\baq\label{eq_def:T_beta,mu}
\mathcal{T}_{m,\betaup,\mu}(\rhoinit,\rhofin) = \left(\inf\left\{\int_0^1 \A_{m, \betaup}(\mu;\rhoup_t,\jup_t)\dd t: (\bs\rhoup,\jupbold)\in\CE(\rhoinit, \rhofin)\right\}\right)^{1/p}.
\eaq 
\end{definition}
\begin{remark}[Decoupling of the transportation cost]\label{rem:T_decoupled}
Let us denote the nonlocal upwind transportation cost for one species by $\overline{\mathcal{T}}_{m,\mu}$. Then, since both the action and the continuity equation are decoupled with respect to the components of $\rhoup$ and $\jup$, the infima are also independent of each other, which implies that for any $\rhoinit, \rhofin \in (\mathcal{P}(\Rd))^2$ we have
\baq\label{eq:T_decoupled}
\mathcal{T}^p_{m,\betaup,\mu}\left(\rhoinit,\rhofin\right) =\frac{1}{\beta^{(1)}}\overline{\mathcal{T}}^p_{m,\mu}\left(\varrho_0^{(1)},\varrho_1^{(1)}\right)+\frac{1}{\beta^{(2)}}\overline{\mathcal{T}}^p_{m,\mu}\left(\varrho_0^{(2)},\varrho_1^{(2)}\right).
\eaq
\end{remark}

\begin{theorem}[Optimal curves exist and are constant speed geodesics]\label{thm:constant speed geodesic}
For any $\mu \in \Mloc^+(\Rd)$ satisfying \eqref{MB2}, \eqref{MB} and \eqref{BC}, any $T\geq 0$ and any $\rhoinit,\rhofin \in (\mathcal{P}(\Rd))^2$ with $\mathcal{T}_{m,\betaup,\mu}(\rhoinit,\rhofin) < \infty$, the infimum in \eqref{eq_def:T_beta,mu} is attained by a curve $(\bs\rhoup,\jupbold) \in \CE(\rhoinit,\rhofin)$ with $ \A_{m, \betaup}(\mu; \rhoup_t,\jup_t)=\mathcal{T}^2_{\betaup,\mu}(\rhoinit,\rhofin)$ for a.e. $t \in [0,1]$. This curve is a constant-speed geodesic, i.e., it satisfies
\baqs
\mathcal{T}_{m,\betaup,\mu}(\rho_s,\rho_t) = \abs{t-s}\mathcal{T}_{m,\betaup,\mu}(\rhoinit,\rhofin),\text{ for every }s,t\in [0,1].
\eaqs
\end{theorem}
\begin{proof}
This can be proved by using \cite[Theorem 2.20]{EPSS2021} and the decoupling of $\A_{m, \betaup}$
Alternatively, one could infer this from Lemma \ref{lem:reparametrization} similar to the proof of \cite[Theorem 4.3]{erbar2012gradient}.
\end{proof}
\begin{lemma}[Reparametrization]\label{lem:reparametrization} For any $\mu \in \Mloc^+(\Rd)$ satisfying \eqref{MB2}, \eqref{MB} and \eqref{BC}, any $T\geq 0$ and any $\rhoinit, \rhofin\in (\mathcal{P}(\Rd))^2$ it holds
\baqs
\mathcal{T}_{m,\betaup,\mu}(\rhoinit,\rhofin) = \inf\Bigg\{\int_0^T \A_{m, \betaup}^{1/p}(\mu;\rhoup_t,\jup_t)\dd t: (\bs\rhoup, \jupbold)\in\CE_T(\rhoinit,\rhofin)\Bigg\}.
\eaqs
\end{lemma}
\begin{proof}
This immediately follows from Theorem \ref{thm:constant speed geodesic}. Alternatively, one can argue via a reparametrization argument similar to the one in the proof of \cite[Theorem 5.4]{Dolbeault_2008}.
\end{proof}

\begin{proposition}\label{prop:W_1 bound}
Let $\mu \in \Mloc^+(\Rd)$ satisfy \eqref{MB2} and \eqref{MB}. Then, for any $\rhoinit, \rhofin \in (\mathcal{P}(\Rd))^2$, again denoting by $\overline{\mathcal{T}}_{m,\mu}$ the transportation cost for one species, there exists $C>0$ such that
\baqs
    W_1(\rhoinit, \rhofin) \coloneqq \sum_{i=1}^2 W_1\left(\varrho_0^{(i)},\varrho_1^{(i)}\right) \le C\mathcal{T}_{m,\betaup,\mu}^{1/p}(\rhoinit, \rhofin).
\eaqs
\end{proposition}
\begin{proof}
By Remark \ref{rem:eq cont}, in \eqref{eq:continuous representative} we can choose a test function $\psi$, which is constant in time, 1-Lipschitz in space and satisfies $0\le\psi\le 1$. Then, a quick calculation yields a uniform bound, which allows us to take the supremum over all $\psi$ and employ the Kantorovich-Rubinstein formula to obtain the result. A detailed proof for $p=2$ and $m(r,s)=r$ (which does not change the argument) can be found in \cite[Proposition 2.21]{EPSS2021}.
\end{proof}
Observe that, by the previous Proposition, Young's inequality and \eqref{eq:T_decoupled}, we have
\baqs
W_1\left(\rhoinit,\rhofin\right) &\le \bar{C}\sum_{i=1}^2\overline{\mathcal{T}}_{m,\mu}^{1/p}\left(\varrho_0^{(i)}, \varrho_1^{(i)}\right)
\le \tilde{C}\left(\sum_{i=1}^2\overline{\mathcal{T}}_{m,\mu}^p\left(\varrho_0^{(i)},\varrho_1^{(i)}\right)\right)^{1/p}\le C\mathcal{T}_{m,\betaup,\mu}^{1/p}(\rhoinit,\rhofin),
\eaqs
where $C$ depends only on the $m$, $p$, $\betaup$, $\eta$ and $\mu$.
Hence, $\mathcal{T}_{m,\betaup,\mu}$ defines a quasimetric on $(\mathcal{P}(\Rd))^2$ and induces a topology stronger than the $W_1$-topology:
\begin{theorem}\label{thm:T_beta,mu quasimetik}
    Let $\mu \in \Mloc^+(\Rd)$ satisfy \eqref{MB2}, \eqref{MB} and \eqref{BC}. Then, the nonlocal upwind transportation cost for two species $\mathcal{T}_{m,\betaup,\mu}$ defines a quasimetric on $(\mathcal{P}_p(\Rd))^2$ and the map $(\rhoinit,\rhofin)\mapsto \mathcal{T}_{m,\betaup,\mu}(\rhoinit,\rhofin)$ is lower semicontinuous with respect to the narrow convergence. The topology induced by $\mathcal{T}_{m,\betaup,\mu}$ is stronger than the $W_1$-topology and the narrow topology. In particular, bounded sets in $((\mathcal{P}_p(\Rd))^2,\mathcal{T}_{m,\betaup,\mu})$ are narrowly relatively compact.
\end{theorem}
\begin{proof}
    Similar to \cite[Theorem 2.22]{EPSS2021} we have that if $\mathcal{T}_{m,\betaup,\mu}(\rhoinit,\rhofin)=0$, then the minimizing pair $(\bs\rhoup,\jupbold)\in\CE(\rhoinit,\rhofin)$ satisfies $\A_{m, \betaup}(\mu;\rhoup_t,\jup_t)=0$ for a.e. $t\in[0,T]$. Thus, for $i=1,2$ we have $j_t^{(i)} \equiv 0$, $(\mu\otimes\mu+\varsigma^{(i)}_t)$-a.e. and hence $\rho^{(i)}_0\equiv\rho^{(i)}_1$. The triangle inequality follows from Lemma \ref{lem:reparametrization} by concatenating the solutions of the nonlocal continuity equation. The compactness and lower semicontinuity of $\A_{m, \betaup}$ shown in Proposition \ref{prop:compactness} are inherited by $\mathcal{T}_{m,\betaup,\mu}$. Lastly, the claims about the topology immediately follow from Proposition \ref{prop:W_1 bound}.
\end{proof}
Now, we adapt the definition of absolutely continuous curves to our setting. 
\begin{definition}[Absolutely continuous curves]\label{def:AC-curves}
Let $\mu \in \Mloc^+(\Rd)$ satisfy \eqref{MB2}, \eqref{MB} and \eqref{BC}. A curve $\bs\rhoup \subset  (\mathcal{P}(\Rd))^2$ belongs to $\AC^p([0,T];((\mathcal{P}_p(\Rd))^2,\mathcal{T}_{m,\betaup,\mu}))$ if there exists $f\in L^p(0,T)$ such that for any $0<s \leq t< T$ we have
\baq\label{eq_def:AC}
	\mathcal{T}_{m,\betaup,\mu}(\rho_{s},\rho_{t})\leq \int_{s}^{t}f(t)\dd t.
\eaq
Such a curve is called \emph{($p$-)absolutely continuous}. For any $\bs\rhoup\in\AC^p([0,T];((\mathcal{P}_p(\Rd))^2,\mathcal{T}_{m,\betaup,\mu}))$ and a.e. $t\in[0,T]$ the limit
\baqs
\abs{\rhoup'_t}\coloneqq \lim_{h\to 0} \frac{\mathcal{T}_{m,\betaup,\mu}(\rhoup_t,\rhoup_{t+h})}{\abs{h}}
\eaqs
is well-defined.\footnote{For details see e.g. \cite[Theorem 1.1.2]{AmbrosioGigliSavare2008}.} It is called the \emph{metric derivative} of $\rho$ at $t$. The map $t\mapsto\abs{\rhoup'_t}$ belongs to $L^p(0,T)$ and satisfies $\abs{\rhoup'_t} \le f(t)$ for any $m$ satisfying \eqref{eq_def:AC}, making it the minimal intgerand in \eqref{eq_def:AC}.
\end{definition}
\begin{proposition}[Metric velocity]\label{prop:met vel} Let $\mu \in \Mloc^+(\Rd)$ satisfy \eqref{MB2}, \eqref{MB} and \eqref{BC}. A curve $\bs\rhoup \subset (\mathcal{P}_p(\Rd))^2$ belongs to $\AC^p([0,T];((\mathcal{P}_p(\Rd))^2,\mathcal{T}_{m,\betaup,\mu}))$ if and only if there exists a family $\jupbold$ such that $(\bs\rhoup,\jupbold)\in\CE_T$ and
\baqs
\int_0^T\A_{m, \betaup}^{1/p}(\mu;\rhoup_t,\jup_t)\dd t < \infty.
\eaqs
In this case, the metric derivative satisfies $\abs{\rhoup'}^p(t) \leq \A_{m, \betaup}(\mu;\rhoup_t,\jup_t)$ for a.e. $t$. Additionally, there exists a unique family $\varjupbold$ such that $(\bs\rhoup,\varjupbold)\in\CE_T$ and
\baqs
\abs{\rhoup'_t}^p =\A_{m, \betaup}(\mu;\rhoup_t,\varjup_t)\text{ for a.e. } t\in[0,T].
\eaqs
This identity holds if and only if $\varjup_t\in T_\rhoup(\mathcal{P}_p(\Rd))^2$ for a.e. $t$, where we define the tangent space at $\rho$ as
\baq\label{eq_def:T_rho}
T_\rhoup(\mathcal{P}_p(\Rd))^2 \coloneqq \left\{\jup \in (\mathcal{M}^\mathrm{as}_{\eta\gamma_1}(G))^2:\A_{m, \betaup}(\mu;\rhoup,\jup)\leq \A_{m, \betaup}(\mu;\rhoup,\jup+ \dup),\forall\dup\in (\mathcal{M}_\mathrm{div}(G))^2\right\}
\eaq
and the space of divergence free-fluxes as
\baqs
    \mathcal{M}_\mathrm{div}(G) \coloneqq \Bigg\{d\in \Mloc(G):\iint_G \babla \varphi\eta\dd d = 0 \text{ for any } \varphi \in C^\infty_c(\Rd)\Bigg\}.
\eaqs
\end{proposition}
\begin{proof}
The first statement about the characterization of absolutely continuous curves follows from \cite[Theorem 5.17]{Dolbeault_2008}, due to Theorem \ref{thm:constant speed geodesic}, Lemma \ref{lem:reparametrization} and Proposition \ref{prop:compactness}. Since, by Corollary \ref{cor_as->lower_action}, we have that antisymmetric fluxes have lower action, it is not restrictive to require the minimizing flux to lie in $(\mathcal{M}^\mathrm{as}_{\eta\gamma_1}(G))^2$. For the converse statement, we  argue as in the proof of \cite[Proposition 2.25]{EPSS2021}. Here we use that the map $j\mapsto\A_{m, \betaup}(\mu;\rhoup,\jup)$ is strictly convex for $\jup \in (\mathcal{M}^\mathrm{as}_{\eta\gamma_1}(G))^2$ with $\A_{m, \betaup}(\mu;\rhoup,\jup)<\infty$ and that the set $\{\jup \in (\mathcal{M}^\mathrm{as}_{\eta\gamma_1}(G))^2:\babla \jup =\babla \jup_t\}$ is closed with respect to weak-$^\ast$ convergence. Additionally, we employ Corollary \ref{cor:bound_by_A} and obtain that $j\mapsto\A_{m, \betaup}(\mu;\rhoup,\jup)$ has locally relatively compact sublevel sets with respect to narrow convergence, by arguing as in the proof of Proposition \ref{prop:compactness}. Finally, applying the direct method of calculus of variations, we see that $\varjupbold$ is well-defined.
\end{proof}

\begin{definition}\label{def:tilde_T_rho}
    Recall that for any $\rhoup\in(\mathcal{P}_p(\Rd))^2$, by Lemma  \ref{lem:connection_flux_velocity}, we can identify any $\jup \in (\Mloc(G))^2$ such that $\A_{m,\betaup}(\mu;\rhoup,\jup)<\infty$ with a velocity $\vup = (v^{(1)},v^{(1)\perp},v^{(2)},v^{(2)\perp})$ given as 
    \baqs
        \dd j^{(i)\mu} &= (v^{(i)}_+)^{q-1}\dd\gamma_1^{(i)}-(v^{(i)}_-)^{q-1}\dd\gamma_2^{(i)},\\
        \dd j^{(i)\perp} &= (v^{(i)\perp}_+)^{q-1}\dd\gamma_1^{(i)\perp}-(v^{(i)\perp}_-)^{q-1}\dd\gamma_2^{(i)\perp}.
    \eaqs
    We define as $\tilde{T}_\rhoup(\mathcal{P}_p(\Rd))^2$ the set of velocities $\vup$ associated this way to $\jup\in T_\rhoup(\mathcal{P}_p(\Rd))^2$.
\end{definition}
In the following proposition we present a characterization of tangent velocities in cases, when $R\land S<\infty$ or $m_\infty\equiv 0$. In these cases they lie in the closure of the set of gradients of smooth functions.
\begin{proposition}[Tangent velocities are almost gradient]\label{prop:tan_flux}
Assume that either $R\land S<\infty$ or $m_\infty\equiv 0$. Let $\mu \in \Mloc^+(\Rd)$ satisfy \eqref{MB2}, \eqref{MB} and \eqref{BC}. Let $\rhoup \in (\mathcal{P}(\Rd))^2$ and $\vup = (v^{(1)},0,v^{(2)},0):G\to\R^4$ be associated to $\jup \in (\Mloc(G))^2$ satisfying $\A_{m,\betaup}(\mu;\rhoup,\jup)<\infty$ as before. Then, we have $\vup\in\tilde{T}_\rhoup(\mathcal{P}_p(\Rd))^2$ if and only if
\baqs
v^{(i)} &\in \overline{\Big\{\babla\varphi:\varphi\in C^\infty_c(\Rd)\Big\}}^{L^q(\eta\hat{\gamma}_\vup^{(i)})}, &&\text{ where } \dd \hat{\gamma}_\vup^{(i)}=\mathbb{1}_{\{v^{(i)}>0\}}\dd\gamma_1^{(i)}+\mathbb{1}_{\{v^{(i)}<0\}}\dd\gamma_2^{(i)}.
\eaqs
\end{proposition}
\begin{proof}
The boundedness of the action implies that the singular part vanishes. Thus, we can argue similarly to the proof of \cite[Proposition 2.26]{EPSS2021}. Let $\jup \in (\Mloc(G))^2$ be the flux associated to $v$. We define $J^{(i)}_+ \coloneqq\supp j^{(i)}_+$ and $\gamma^{(i)}_+\coloneqq\gamma^{(i)}_1|_{J^{(i)}_+}$. Observe that by the antisymmetry of $j^{(i)}$ we have $(J^{(i)}_+)^\top = \supp j^{(i)}_-$. Therefore, if $\A_{m, \betaup}(\mu;\rhoup,\jup)<\infty$, by Lemma \ref{lem:connection_flux_velocity} and the assumptions on $m$, for $i=1,2$ there exist antisymmetric functions $f^{(i)}$ such that $\dd j^{(i)} = f^{(i)}\dd\big(\gamma^{(i)}_++(\gamma^{(i)}_+)^\top\big)$. These functions satisfy
\baqs
    \bar\A_{m}(\mu;\rho^{(i)},j^{(i)}) = ||f^{(i)}_+||^p_{L^p(\eta\gamma_1^{(i)})} = ||f^{(i)}||^p_{L^p(\eta\gamma_+^{(i)})}.
\eaqs
By symmetry, we can rewrite the divergence as
\baqs
    \frac{1}{2}\iint_G\babla\phi\eta\dd j^{(i)} = \iint_G\babla\phi\eta\dd j^{(i)}_+ = \iint_G\babla\phi f^{(i)}\eta\dd\gamma^{(i)}_+.
\eaqs
Now, we observe that \eqref{eq_def:T_rho} is equivalent to
\baq\label{eq:tan_flux_proof_1}
    \iint_G |f^{(i)}|^p-|f^{(i)}+g^{(i)}|^p\eta\dd\gamma_+^{(i)}\le 0,
\eaq
for all antisymmetric $g^{(i)}\in L^p(\eta\gamma_+^{(i)})$, which satisfy $\iint_G\babla\cdot\psi g^{(i)}\eta\dd\gamma_+^{(i)}=0$ for every $\psi \in C_c^\infty(\Rd)$. Since the sign of $g^{(i)}$ may be negative, \eqref{eq:tan_flux_proof_1} is equivalent to
\baqs
    (f^{(i)})^{p-1}g^{(i)}=0,\qquad \eta\gamma_+^{(i)}\text{-a.e.},
\eaqs
which is equivalent to
\baqs
    (f_+^{(i)})^{p-1}g^{(i)}=0,\qquad \eta\gamma^{(i)}\text{-a.e.}
\eaqs
Now, note that we have $v^{(i)}_+=(f^{(i)}_+)^{p-1}$. Hence, $v^{(i)}_+$ belongs to the closure of $\{\babla\varphi:\varphi\in C^\infty_c(\Rd)\}$ in $L^q(\eta\gamma^{(i)}_1)$. Finally, recalling that $v^{(i)}$ are antisymmetric and that $(\gamma_1^{(i)})^\top=\gamma_2^{(i)}$, the claim follows.
\end{proof}
\begin{remark}\label{rem:dense_subset_T_rho}
Proposition \ref{prop:tan_flux} shows that if $R\land S<\infty$ or $m_\infty\equiv 0$ and for $\mu$ and $\rhoup$ as in the statement, for $\jup$ chosen from a dense subset of $T_\rhoup(\mathcal{P}_p(\Rd))^2$, there exists a measurable function $\varphiup=(\varphi^{(1)},0,\varphi^{(2)},0):\Rd \rightarrow\R^4$, such that we have
\baqs
\A_{m, \betaup}(\mu;\rhoup,\jup) = \tilde{\A}_{m, \betaup}(\mu;\rhoup,\babla\varphiup),
\eaqs
and we can then write
\baq\label{eq:dense flux structure}
    \dd j^{(i)} = \big(\big(\babla\varphi^{(i)}\big)_+\big)^{q-1}\dd\gamma_1^{(i)}-\big(\big(\babla\varphi^{(i)}\big)_-\big)^{q-1}\dd\gamma_2^{(i)}.
\eaq
\end{remark}
\begin{proposition}[Absolutely continuous curves stay supported in $\supp\mu$]\label{prop:Absolutely continuous curves stay supported in supp mu}
    Let $\mu \in \Mloc^+(\Rd)$ satisfy \eqref{MB2}, \eqref{MB} and \eqref{BC} and let $\bs\rhoup \in\AC^p([0,T];((\mathcal{P}_p(\Rd))^2,\mathcal{T}_{m,\betaup,\mu}))$ be such that $\supp\rho_0^{(i)}\subseteq\supp\mu$, $i=1,2$. Additionally, assume that $m$ satisfies the following condition:
    \begin{equation}\label{A}\tag{A}
        R\land S<\infty \quad\text{or}\quad m_\infty \equiv 0 \quad\text{or}\quad m(r,s)=0 \iff r=0.
    \end{equation}
    Then, we have $\supp\rho_t^{(i)}\subseteq\supp\mu$ for all $t\in[0,T]$, $i=1,2$.
\end{proposition}
\begin{proof}
If $R\land S<\infty$ or $m_\infty \equiv 0$, this is immediate from the finiteness of the action. Thus, let $R=S=\infty$ and assume $m(r,s)>0$ for every $r>0$. This allows us to argue similarly to \cite[Proposition 2.28]{EPSS2021} by employing Proposition \ref{prop:met vel}, Lemma \ref{lem:connection_flux_velocity} and Lemma \ref{lem:continuous representative} to obtain a pair $(\bs\rhoup,\jupbold)\in(\mathcal{P}_p(\Rd))^2\times(\mathcal{M}^\mathrm{as}_{\eta\gamma_1}(G))^2$ satisfying \eqref{eq:continuous representative} and we have for $i=1,2$
\baqs
    \dd j^{(i)\mu}_t &= (f^{(i)}_t)_+\dd(\mu\otimes\mu)-(f^{(i)}_t)_-\dd(\mu\otimes\mu),\\
    \dd j^{(i)\perp}_t &= (f^{(i)\perp}_t)_+\dd(\rho_t^{(i)\perp}\otimes\mu)-(f^{(i)\perp}_t)_-\dd(\mu\otimes\rho_t^{(i)\perp}),
\eaqs
with suitable antisymmetric $f^{(i)}_t$ and $f^{(i)\perp}_t$. Here we used that $\gamma_{1,t}^{(i)\perp}\ll\rho_t^{(i)\perp}\otimes\mu$ and $\gamma_{2,t}^{(i)\perp}\ll\mu\otimes\rho_t^{(i)\perp}$ since we assumed $m(r,s)=0$ if and only if $r=0$.
Inserting $\varphi\in C^\infty_c(\Rd)$ with $\supp \varphi\subset\Rd\setminus\supp\mu$ and $\varphi\geq 0$ into \eqref{eq:continuous representative}, by the antisymmetry of $v^{(i)}_t$ and $v^{(i)\perp}_t$, we obtain for both $i=1,2$
\baqs
\int_\Rd\varphi(x)\dd\rho_t^{(i)}(x)&=\int_\Rd\varphi(x)\dd\rho_0^{(i)}(x)+\int_0^t\iint_G\babla\varphi(x,y)(f^{(i)}_\tau)_+(x,y)\eta(x,y)\dd\mu(x)\dd\mu(y)\dd \tau\\
&\hphantom{=\int_\Rd\varphi(x)\dd\rho_0^{(i)}(x)}+\int_0^t\iint_G(\varphi(y)-\varphi(x))(f^{(i)\perp}_\tau)_+(x,y)\eta(x,y)\dd\rho^{(i)\perp}(x)\dd\mu(y)\dd \tau\\
&\le-\int_0^t\iint_G\varphi(x)(f^{(i)\perp}_\tau)_+(x,y)\eta(x,y)\dd\rho^{(i)\perp}(x)\dd\mu(y)\dd \tau\le 0.
\eaqs
Since $\rho_t^{(i)\perp}$ and $\mu$ are nonnegative measures this finishes the proof.
\end{proof}
\begin{remark}\label{rem:importance_of_A}
    \begin{enumerate}[label=(\roman*)]
        \item Assumption \eqref{A} plays an important role in the proof of existence of gradient flows. By Proposition \ref{prop:Absolutely continuous curves stay supported in supp mu} it guarantees that when $\mu$ is a counting measure and $\supp\rho_0^{(i)}\subseteq\supp\mu$, the same is true for all times. This will reduce the continuity equation to a finite system of ordinary differential equations.
        \item On a finite graph $m(r,s)=0\iff r=0$ means that the mobility vanishes if and only if there is no mass on the node from which the mass is flowing away. This is reasonable from a model point of view.
    \end{enumerate}
\end{remark}

\section{Two nonlocally interacting species as Finsler gradient flows}\label{sec:FinslerGrad}

In this section we define a Minkowski norm on $T_\rhoup(\mathcal{P}_p(\Rd))^2$, thereby inducing a Finslerian structure. Moreover, the inner product gives rise to a notion of gradient and divergence. Subsequently, we show that this gradient of the nonlocal cross-interaction energy, 
\begin{equation}
\tag{\eqref{eq:Energy} revisited}
    \mathcal{E}(\rhoup) = \frac{1}{2}\sum_{i,k=1}^2\iint_\RdRd K^{(ik)}(x,y)\dd\rho^{(i)}(x)\dd\rho^{(k)}(y),
\end{equation}
exists and is unique, whenever $K^{(12)}$ and  $K^{(21)}$ are positive multiples of one another. Then, system \eqref{eq:NL2CIE_intro} reads
\baq\label{eq:NL2CIE}
    \partial_t\rho_t^{(i)}(x)
    & + \beta^{(i)} \int_\Rd \big(\mathfrak{m}_t^{(i)}(x,y)\babla\big(K^{(i1)} \ast \rho_t^{(1)} + K^{(i2)} \ast \rho_t^{(2)}\big)(x,y)_-\big)^{q-1} \eta(x,y) \dd\mu(y)\\
    &-\beta^{(i)} \int_\Rd \big(\mathfrak{m}_t^{(i)}(y,x)\babla\big(K^{(i1)}\ast \rho_t^{(1)}+K^{(i2)} \ast \rho_t^{(2)}\big)(x,y)_+\big)^{q-1} \eta(x,y) \dd\mu(y)\\
    & + \beta^{(i)} \int_\Rd \big(\mathfrak{m}_{\infty,t}^{(i)}(x,y)\babla\big(K^{(i1)} \ast \rho_t^{(1)} + K^{(i2)} \ast \rho_t^{(2)}\big)(x,y)_-\big)^{q-1} \eta(x,y) \varsigma_t^{(i)}(x,\dd y)\\
    &-\beta^{(i)} \int_\Rd \big(\mathfrak{m}_{\infty,t}^{(i)}(y,x)\babla\big(K^{(i1)}\ast \rho_t^{(1)}+K^{(i2)} \ast \rho_t^{(2)}\big)(x,y)_+\big)^{q-1} \eta(x,y) \varsigma_t^{(i)}(x,\dd y)=0,
\eaq
is a gradient flow of $\mathcal{E}$ with respect to the Finslerian structure. Here $\beta^{(1)}, \beta^{(2)}>0$ and, after a rescaling, we shall assume that $\beta^{(1)} = 1$ and $K^{(12)}= K^{(21)}$.

This will finally allow us to deduce that weak solutions of \eqref{eq:NL2CIE} exist for a large family of base measures $\mu$ via approximation with finite graphs. These considerations are based on known results for one species \cite{EPSS2021}.


Before we construct the Finslerian structure on the product space, let us introduce our notion of weak solutions.
\begin{definition}\label{def:weak solution to NL2CIE}
A curve $\bs\rhoup:[0,T]\to(\mathcal{P}_p(\Rd))^2$ is called a \emph{weak solution} to \eqref{eq:NL2CIE} if the pair $(\bs\rhoup,\jupbold)$ is a weak solution of the continuity equation
\baqs
\partial_t\rhoup_t+\babla\cdot \jup_t = 0\text{ on }[0,T]\times\Rd,
\eaqs
in the sense of Definition \ref{def:cont eq}, where the flux $\jup:[0,T]\to(\Mloc(G))^2$ for $i=1,2$ is given by
\baq
\label{eq_def:flux_weak_solution}
    \dd j_t^{(i)\mu}&=\big(\beta^{(i)}(\babla\delta_{\rho^{(i)}}\mathcal{E}(\rhoup_t))_-\big)^{q-1}\dd\gamma_{1,t}^{(i)}-\big(\beta^{(i)}(\babla\delta_{\rho^{(i)}}\mathcal{E}(\rhoup_t))_+\big)^{q-1}\dd\gamma_{2,t}^{(i)},\\
    \dd j_t^{(i)\perp}&=\big(\beta^{(i)}(\babla\delta_{\rho^{(i)}}\mathcal{E}(\rhoup_t))_-\big)^{q-1}\dd\gamma_{1,t}^{(i)\perp}-\big(\beta^{(i)}(\babla\delta_{\rho^{(i)}}\mathcal{E}(\rhoup_t))_+\big)^{q-1}\dd\gamma_{2,t}^{(i)\perp}.
\eaq
\end{definition}
\begin{remark}\label{rem:coupling}
There was no coupling between the two species until this point, neither in the definition of the action densities nor in the continuity equations. Only now, via the dependence of $\frac{\delta\mathcal{E}}{\delta\rho^{(i)}}$ on both $\rho^{(1)}$ and $\rho^{(2)}$, cross-interaction occurs.
\end{remark}

Throughout the rest of this paper, we make the following assumptions on the kerneles $K^{(ik)}:\Rd\times\Rd\to\R$, $i,k=1,2$:
\begin{align}
	\label{K2}\tag{K1} &\text{(Symmetry) } &\forall x,y\in \Rd\text{ it holds } K^{(ik)}(x,y)=K^{(ik)}(y,x)  ,\\
	\label{K3}\tag{K2} &\text{(Growth) } &\exists L_K\in(0,\infty) \text{ such that }\; \forall(x,y),(x',y')\in \Rd\times\Rd, \text{ and $i,k=1,2$, we have }\\
	\notag& &\abs{K^{(ik)}(x,y)-K^{(ik)}(x',y')}\leq L_K\big(\abs{(x,y)-(x',y')}\lor\abs{(x,y)-(x',y')}^p\big)
\end{align}
In particular, \eqref{K3} implies continuity and guarantees that the proper domain of $\mathcal{E}$ contains $(\mathcal{P}_p(\Rd))^2$. Indeed, by \eqref{K3} there exists $C>0$ s.t. for all $x,y\in\Rd$ we have $|K^{(ik)}(x,y)|\leq C(1+|x|^p+\abs{y}^p)$ (see \cite[Remark 3.2]{EPSS2021} for details).
\begin{proposition}[Continuity of the energy]\label{prop:Continuity of the energy} Let the potentials $K^{(ik)}$, $i,k=1,2$ satisfy \eqref{K2}, \eqref{K3}. Then, for any sequence $(\rhoup^n)_{n\in\mathbb{N}}\subset(\mathcal{P}_p(\Rd))^2$ narrowly converging to some $\rhoup \in (\mathcal{P}_p(\Rd))^2$, we have
\baqs
\lim_{n\to\infty}\mathcal{E}(\rhoup^n)=\mathcal{E}(\rhoup).
\eaqs
\end{proposition}
\begin{proof}
Keeping in mind that $\rho^n\rightharpoonup\rho$ if and only if $\rho^{n,(i)}\rightharpoonup\rho^{(i)}$ for both $i=1,2$ and using the assumptions on $K^{(ik)}$, $i,k=1,2$, we can argue as in \cite[Proposition 3.3]{EPSS2021}; we truncate the kernels to obtain bounded continuous test functions. Then, we employ Lebesgue's dominated convergence theorem and a diagonal argument.
\end{proof}
\subsection{Finslerian geometry}


\begin{definition}\label{def:finsler_metric}(Finsler metric).
    Given $\rhoup \in (\mathcal{P}_p(\Rd))^2$ we define the function $l_\rhoup:T_\rhoup(\mathcal{P}_p(\Rd))^2\to (T_\rhoup(\mathcal{P}_p(\Rd))^2)^\ast$ as follows: for any $\jup, \bar{\varjup}\in T_\rhoup(\mathcal{P}_p(\Rd))^2$, in the case $R=S=\infty$, we set
    \baqs
        l_\rhoup(\jup)[\bar{\varjup}] \coloneqq&\, \frac{1}{2}\sum_{i=1}^2\frac{1}{\beta^{(i)}}\Bigg[\iint_G \bar{\jmath}^{(i)}(x,y)\left(\frac{j^{(i)}_+(x,y)}{\mathfrak{m}^{(i)}(x,y)}-\frac{j^{(i)}_-(x,y)}{\mathfrak{m}^{(i)}(y,x)}\right)^{p-1}\eta(x,y)\dd\mu(x)\dd\mu(y)\\
        &\hphantom{\,\frac{1}{2}\sum_{i=1}^2\frac{1}{\beta^{(i)}}\Bigg[}+\iint_G \bar{\jmath}^{(i)\perp}(x,y)\Bigg(\frac{j^{(i)\perp}_+(x,y)}{\mathfrak{m}^{(i)}_\infty(x,y)} - \frac{j^{(i)\perp}_-(x,y)}{\mathfrak{m}^{(i)}_\infty(y,x)}\Bigg)^{p-1}\eta(x,y)\dd\varsigma^{(i)}(x,y)\Bigg],
    \eaqs
    where $\varsigma^{(i)}$ are as in Lemma \ref{lem:connection_flux_velocity}. In the case $R\land S<\infty$, we analogously define
    \baqs
        l_\rhoup(\jup)[\bar{\varjup}] \coloneqq&\, \frac{1}{2}\sum_{i=1}^2\frac{1}{\beta^{(i)}}\iint_G \bar{\jmath}^{(i)}(x,y)\left(\frac{j^{(i)}_+(x,y)}{\mathfrak{m}^{(i)}(x,y)}-\frac{j^{(i)}_-(x,y)}{\mathfrak{m}^{(i)}(y,x)}\right)^{p-1}\eta(x,y)\dd\mu(x)\dd\mu(y).
    \eaqs
    Next, we define the Finsler metric $F_\rhoup:T_\rhoup(\mathcal{P}_p(\Rd))^2\to\R$ as
    \baqs
        F_\rhoup(\jup) \coloneqq (l_{\rho}(\jup)[\jup])^{1/p}=\A_{m, \betaup}^{1/p}(\mu;\rhoup,\jup).
    \eaqs
\end{definition}
Our first goal is to show that $F$ is a Minkowski norm. To this end, we establish a H{\"o}lder-type inequality:
\begin{lemma}[H{\"o}lder-type inequality]\label{lem:Holder-type_inequality_for_l}
For $\jup, \bar{\varjup}\in T_\rhoup(\mathcal{P}_p(\Rd))^2$ it holds
\baq\label{eq:Holder}
l_\rhoup(\jup)[\bar{\varjup}] \le (l_\rhoup(\bar{\varjup})[\bar{\varjup}])^{1/p} (l_\rhoup(\jup)[\jup])^{1/q}.
\eaq
We have equality in \eqref{eq:Holder} if and only if there exists $\lambda \geq 0$ such that for $i=1,2$ we have $\bar{\jmath}^{(i)} = \lambda j^{(i)}$, $\eta(\mu\otimes\mu)$-a.e. and, if $R=S=\infty$, $\bar{\jmath}^{(i)\perp} = \lambda j^{(i)\perp}$, $\eta\varsigma^{(i)}$-a.e.
\end{lemma}
\begin{proof}
We only consider the case $R=S=\infty$. A proof in the case $R\land S<\infty$ can then be obtained by setting the recession terms to zero. First note that we have the simple estimate
\baq\label{eq:Holder_for_l_proof_1}
    \frac{\bar{\jmath}(j_+)^{p-1}}{m^{p-1}(r,s)}-\frac{\bar{\jmath}(j_-)^{p-1}}{m^{p-1}(s,r)} \le \frac{\bar{\jmath}_+(j_+)^{p-1}}{m^{p-1}(r,s)}+\frac{\bar{\jmath}_-(j_-)^{p-1}}{m^{p-1}(s,r)},
\eaq
with equality if and only if $\bar{\jmath}$ is nonegative, where $j$ is positive and nonpositive, where $j$ is negative. Recalling that $p-1=p/q$, we shorten the notation by introducing 
\baqs
    a^{(i)}_1(x,y)&\coloneqq \frac{\bar{\jmath}^{(i)}_+(x,y)}{(\mathfrak{m}(x,y))^{1/q}},\qquad b^{(i)}_1(x,y)\coloneqq \frac{(j^{(i)}_+(x,y))^{p/q}}{(\mathfrak{m}(x,y))^{p/q^2}},\\
    a^{(i)}_2(x,y)&\coloneqq \frac{\bar{\jmath}^{(i)}_-(x,y)}{(\mathfrak{m}(y,x))^{1/q}},\qquad b^{(i)}_2(x,y)\coloneqq \frac{(j^{(i)}_-(x,y))^{p/q}}{(\mathfrak{m}(y,x))^{p/q^2}},
\eaqs
and similarly the recession terms $a^{(i)}_{k,\infty}$ and $b^{(i)}_{k,\infty}$ for $i,k=1,2$. Note that $a^{(i)}_1 a^{(i)}_2 = 0$ and similar for all the other terms. We use this fact, \eqref{eq:Holder_for_l_proof_1}, and apply H{\"o}lder's inequality for sums and integrals, to obtain
\baq\label{eq:Holder_for_l_proof_2}
    l_\rhoup(\jup)[\bar{\varjup}] &\le \frac{1}{2}\sum_{i=1}^2\frac{1}{\beta^{(i)}}\Bigg[\iint_G \sum_{k=1}^2 a^{(i)}_k b^{(i)}_k \eta\dd(\mu\otimes\mu)+\iint_G \sum_{k=1}^2 a^{(i)}_{k,\infty} b^{(i)}_{k,\infty} \eta\dd\varsigma^{(i)}\Bigg]\\
    &\le \frac{1}{2}\sum_{i=1}^2\frac{1}{\beta^{(i)}}\Bigg[\iint_G \Bigg(\sum_{k=1}^2 (a^{(i)}_k)^p\Bigg)^{1/p} 
    \Bigg(\sum_{k=1}^2 (b^{(i)}_k)^q\Bigg)^{1/q} \eta\dd(\mu\otimes\mu)\\
    &\hphantom{\le \frac{1}{2}\sum_{i=1}^2\frac{1}{\beta^{(i)}}\Bigg[}+ 
    \iint_G \Bigg(\sum_{k=1}^2 (a^{(i)}_{k,\infty})^p\Bigg)^{1/p} \Bigg(\sum_{k=1}^2 (b^{(i)}_{k,\infty})^q\Bigg)^{1/q} \eta\dd\varsigma^{(i)}\Bigg]\\
    &\le \frac{1}{2}\sum_{i=1}^2\frac{1}{\beta^{(i)}}\Bigg[\Bigg(\iint_G \sum_{k=1}^2 (a^{(i)}_k)^p \eta\dd(\mu\otimes\mu)\Bigg)^{1/p} \Bigg(\iint_G \sum_{k=1}^2 (b^{(i)}_k)^q \eta\dd(\mu\otimes\mu)\Bigg)^{1/q}\\
    &\hphantom{\le \frac{1}{2}\sum_{i=1}^2\frac{1}{\beta^{(i)}}\Bigg[}+ \Bigg(\iint_G \sum_{k=1}^2 (a^{(i)}_{k,\infty})^p \eta\dd\varsigma^{(i)}\Bigg)^{1/p} \Bigg(\iint_G \sum_{k=1}^2 (b^{(i)}_{k,\infty})^q \eta\dd\varsigma^{(i)}\Bigg)^{1/q}\Bigg]\\
    &\le \frac{1}{2}\sum_{i=1}^2\frac{1}{\beta^{(i)}}\Bigg[\Bigg(\iint_G \sum_{k=1}^2 (a^{(i)}_k)^p \eta\dd(\mu\otimes\mu) + \iint_G \sum_{k=1}^2 (a^{(i)}_{k,\infty})^p \eta\dd\varsigma^{(i)}\Bigg)^{1/p}\\
    &\hphantom{\le \frac{1}{2}\sum_{i=1}^2\frac{1}{\beta^{(i)}}\Bigg[\Bigg(} \Bigg(\iint_G \sum_{k=1}^2 (b^{(i)}_k)^q \eta\dd(\mu\otimes\mu) + \iint_G \sum_{k=1}^2 (b^{(i)}_{k,\infty})^q \eta\dd\varsigma^{(i)}\Bigg)^{1/q}\Bigg]\\
    &\le (l_\rhoup(\bar{\varjup})[\bar{\varjup}])^{1/p} (l_\rhoup(\jup)[\jup])^{1/q}.
\eaq
The third inequality in \eqref{eq:Holder_for_l_proof_2} is satisfied with equality if and only if there exist $\lambda^{(i)}, \lambda^{(i)}_\infty >0$, $i=1,2$, s.t. $\sum_{k=1}^2 (a^{(i)}_k)^p = \lambda^{(i)} \sum_{k=1}^2 (b^{(i)}_k)^q$, $\eta(\mu\otimes\mu)$-a.e. and $\sum_{k=1}^2 (a^{(i)}_{k,\infty})^p = \lambda^{(i)} \sum_{k=1}^2 (b^{(i)}_{k,\infty})^q$, $\eta\varsigma^{(i)}$-a.e. Assuming equality in the third inequality, equality in the first and second inequality is achieved if and only if for $i,k=1,2$ we have $(a^{(i)}_k)^p = \lambda^{(i)} (b^{(i)}_k)^q$, $\eta(\mu\otimes\mu)$-a.e. and $(a^{(i)}_{k,\infty})^p = \lambda^{(i)} (b^{(i)}_{k,\infty})^q$, $\eta\varsigma^{(i)}$-a.e. Assuming equality in the first three inequalities, equality in the fourth inequality is obtained if and only if for $i=1,2$ we have $\lambda^{(i)}=\lambda^{(i)}_\infty$, while equality in the last inequality is given if and only if $\lambda^{(1)} = \lambda^{(2)}$ and $\lambda^{(1)}_\infty = \lambda^{(2)}_\infty$. Combining these considerations finishes the proof.
\end{proof}
\begin{theorem}[Minkowski norm]\label{thm:F_ist_Minkowski-norm}
    For any $\rhoup\in(\mathcal{P}_p(\Rd))^2$ we have that $F_\rhoup$ is a Minkowski-norm on $T_\rhoup(\mathcal{P}_p(\Rd))^2$, i.e. it is smooth away from zero and satisfies
    \begin{enumerate}[label=(\roman*)]
        \item Positivity: $F_\rhoup(\jup) > 0$ for all $\rhoup\in(\mathcal{P}_p(\Rd))^2$ and all $0\ne \jup\in T_\rhoup(\mathcal{P}_p(\Rd))^2$. 
        \item Positive 1-homogeneity: $F_\rhoup(\lambda \jup) = \lambda F_\rhoup(\jup)$ for all $\lambda > 0$, $\rhoup\in(\mathcal{P}_p(\Rd))^2$ and $\jup\in T_\rhoup(\mathcal{P}_p(\Rd))^2$.
        \item Strong convexity: $F_\rhoup(\jup+\bar{\varjup})\le F_\rhoup(\jup)+F_\rhoup(\bar{\varjup})$ for $\rhoup\in(\mathcal{P}_p(\Rd))^2$ and $\jup,\bar{\varjup}\in T_\rhoup(\mathcal{P}_p(\Rd))^2$, with equality if and only if there exists $C\ge 0$ such that $\jup = C \bar{\varjup}$.
    \end{enumerate}
\end{theorem}
\begin{proof}
    Positivity, positive 1-homogeneity and smoothness, when $j^{(i)} \ne 0$, $\eta(\mu\otimes\mu)$-a.e. (and, for $R=S=\infty$, $j^{(i)\perp} \ne 0$, $\eta\varsigma^{(i)}$-a.e.) are immediate from the definition. Strong convexity is obtained from Lemma \ref{lem:Holder-type_inequality_for_l} as follows:
    \baqs
        (F_\rhoup(\jup+\bar{\varjup}))^p &= l_\rhoup(\jup+\bar{\varjup})[\jup+\bar{\varjup}] = l_\rhoup(\jup+\bar{\varjup})[\jup]+l_\rhoup(\jup+\bar{\varjup})[\bar{\varjup}]\\
        &\le \big((l_\rhoup(\jup)[\jup])^{1/p}+(l_\rhoup(\bar{\varjup})[\bar{\varjup}])^{1/p}\big)(l_\rhoup(\jup+\bar{\varjup})[\jup+\bar{\varjup}])^{1/q} \\
        &= (F_\rhoup(\jup)+F_\rhoup(\bar{\varjup})) (F_\rhoup(\jup+\bar{\varjup}))^{p-1}.
    \eaqs
    Dividing by $F^{p-1}$ yields the statement.
\end{proof}
\begin{remark}\label{rem:Finsler-Vergleich}
    We use a different notion of Minkowski norm, compared to \cite{EPSS2021}. There, (i), (ii) and (iii) are replaced by the stronger assumption that the second variation of $F_\rhoup(\jup)$ is a symmetric positive definite bilinear form if $\jup$ is nonzero $(\rho\otimes\mu+\mu\otimes\rho)$-a.e. However, the notion used here is better suited to the case $p \ne 2$ and also commonly used, e.g. in \cite{Agueh2012_finsler}.
\end{remark}
\begin{proposition}\label{prop:l_derives_from_F}
Let $\rhoup \in (\mathcal{P}_p(\Rd))^2$ and $\jup\in T_\rhoup(\mathcal{P}_p(\Rd))^2$ such that for $i=1,2$ we have density of $j^{(i)}\ne 0$, $\eta(\mu\otimes\mu)$-a.e. and $j^{(i)\perp}\ne 0$, $\eta\varsigma^{(i)}$-a.e., if $R=S=\infty$. Then, for any $\bar{\varjup}\in T_\rhoup(\mathcal{P}_p(\Rd))^2$, we have
\baqs
    \frac{1}{p}\frac{\dd}{\dd \tau}(F_\rhoup[\jup+\tau\bar{\varjup}])^p\Big|_{\tau=0}=l_\rhoup(\jup)[\bar{\varjup}].
\eaqs
\end{proposition}
\begin{proof}
Recalling that $(a+\tau b)^p = \sum_{k=0}^\infty\binom{p}{k} a^{p-k}\tau^kb^k = a^p+ pa^{p-1}\tau b +\mathcal{O}(\tau^2)$ and that the different species as well as the absolutely continuous and the singular parts can be treated individually, this follows as in \cite[Appendix A]{EPSS2021}.
\end{proof}
\begin{definition}[Differential and metric gradient]\label{def:diff and grad}
Given a functional $\mathcal{F}:(\mathcal{P}(\Rd))^2\to\R\cup\{+\infty\}$, we define its \emph{differential} at $\rhoup \in (\mathcal{P}_p(\Rd))^2$ in direction $\bar{\varjup} \in T_\rhoup(\mathcal{P}_p(\Rd))^2$ by
\baqs
\diff \mathcal{F}(\rhoup)[\bar{\varjup}] \coloneqq \frac{\dd}{\dd t}\mathcal{F}(\tilde{\rhoup}_t)\Big|_{t= 0},
\eaqs
where $\tilde{\rhoup}_t$ solves $\frac{\dd}{\dd t}\tilde{\rhoup}_t = -\babla\cdot \bar{\varjup}$ on a small interval according to Definition \ref{def:cont eq} and satisfies $\tilde{\rhoup}_0=\rhoup$.

We further define the metric gradient (if it exists) via the equation
\baqs
\diff \mathcal{F}(\rhoup)[\bar{\varjup}] = l_\rhoup(\grad \mathcal{F}(\rhoup))[\bar{\varjup}],\qquad \text{for any } \bar{\varjup} \in T_\rhoup(\mathcal{P}_p(\Rd))^2.
\eaqs
\end{definition}
\begin{theorem}[Uniqueness of the gradient]\label{thm:grad unique}
Given a functional $\mathcal{F}:(\mathcal{P}(\Rd))^2\to\R\cup\{+\infty\}$, if for $\rhoup \in (\mathcal{P}_p(\Rd))^2$ the differential $\diff \mathcal{F}(\rhoup)$ exists, then it is unique.
\end{theorem}
\begin{proof}
Similar to \cite[Subsection 3.1]{EPSS2021}, the uniqueness of the gradient is an immediate consequence of the injectivity of the map $j\mapsto l_\rhoup(\jup)[\bar{\varjup}]$ for given $\bar{\varjup}\in T_\rhoup(\mathcal{P}_p(\Rd))^2$. To show this injectivity, let $\jup,\tilde{\varjup} \in T_\rhoup(\mathcal{P}_p(\Rd))^2$ with $l_\rhoup(\jup) =l_\rhoup(\tilde{\varjup})$. If either $j^{(i)} = 0$ or $\tilde{\jmath}^{(i)} = 0$, $\eta(\mu\otimes\mu)$-a.e., then $l_\rhoup(\jup) =l_\rhoup(\tilde{\varjup})$ implies $j^{(i)} = \tilde{\jmath}^{(i)} = 0$, $\eta(\mu\otimes\mu)$-a.e. Similarly, if either $j^{(i)\perp} = 0$ or $\tilde{\jmath}^{(i)\perp} = 0$, $\eta\varsigma^{(i)}$-a.e., then $l_\rhoup(\jup) =l_\rhoup(\tilde{\varjup})$ implies $j^{(i)\perp} = \tilde{\jmath}^{(i)\perp} = 0$, $\eta\varsigma^{(i)}$-a.e. Now, for at least one $i\in\{1,2\}$ let $j^{(i)}(A) \ne 0$ for some $A\subset G$ with $\mu\otimes\mu(A)>0$ or $j^{(i)\perp}(B)\ne 0$ for some $B\subset G$ with $\varsigma^{(i)}(B)>0$, and let $\tilde{\jmath}^{(i)}(\tilde{A}) \ne 0$ for some $\tilde{A}\subset G$ with $\mu\otimes\mu(\tilde{A})>0$ or $\tilde{\jmath}^{(i)\perp}(\tilde{B})\ne 0$ for some $\tilde{B}\subset G$ with $\varsigma^{(i)}(\tilde{B})>0$. Then, by \eqref{eq:Holder} we obtain
\baqs
0 < l_\rhoup(\jup)[\jup] = l_\rhoup(\tilde{\jmath})[\jup] \leq (l_\rhoup(\tilde{\jmath})[\tilde{\jmath}])^{1/p}(l_\rhoup(\jup)[\jup])^{1/q},
\eaqs
which gives us $l_\rhoup(\jup)[\jup] \leq l_\rhoup(\tilde{\jmath})[\tilde{\jmath}]$. Inverting the roles of $\jup$ and $\tilde{\varjup}$, we obtain $l_\rhoup(\tilde{\varjup})[\tilde{\varjup}] \le l_\rhoup(\jup)[\jup]$, i.e. we have $l_\rhoup(\jup)[\jup] = l_\rhoup(\tilde{\varjup})[\tilde{\varjup}]$. This implies equality in the H{\"o}lder-type inequality \eqref{eq:Holder}, thus yielding $j^{(i)} = C\tilde{\jmath}^{(i)}$, $\eta(\mu\otimes\mu)$-a.e. and $j^{(i)\perp} = C \tilde{\jmath}^{(i)\perp}$, $\eta\varsigma^{(i)}$-a.e. for some $C \geq0$ and both $i=1,2$. Using the positive $1$-homogeneity of $l_\rhoup$ we obtain $l_\rhoup(\jup) = l_\rhoup(C \tilde{\varjup}) = C l_\rhoup(\tilde{\varjup})= C l_\rhoup(\jup)$, which yields $C =1$ since $l_\rhoup(\jup)[\jup] \neq 0$. This proves the claimed injectivity.
\end{proof}
Since the map $\jup \mapsto l_\rhoup(\jup)[\bar{\varjup}]$ is not antisymmetric, i.e. $l_\rhoup(\jup)[\bar{\varjup}] \ne - l_\rhoup(-\jup)[\bar{\varjup}]$, we separately need to define the negative gradient:
\begin{definition}[Negative metric gradient]
Given $\mathcal{F}:(\mathcal{P}(\Rd))^2\to\R$ and $\rhoup \in (\mathcal{P}_p(\Rd))^2$, we define the \emph{negative metric gradient} of $\mathcal{F}$ at $\rhoup$ by
\baqs
    l_\rhoup(\grad^-\mathcal{F}(\rhoup))[\bar{\varjup}] \coloneqq -\diff \mathcal{F}(\rhoup)[\bar{\varjup}], \qquad \text{for all } \bar{\varjup} \in T_\rhoup(\mathcal{P}_p(\Rd))^2.
\eaqs
\end{definition}
Since $l_\rhoup(\cdot)$ is not antisymmetric, in general $\grad^-\mathcal{F}(\rhoup) \neq -\grad\mathcal{F}(\rhoup)$.
To define the (unique) direction of steepest descent at $\rhoup$ we use the following criterion, as in the Riemannian case:
\begin{definition}
Given $\rhoup$ and $\mathcal{F}$ such that $\diff \mathcal{F}(\rhoup)\neq 0$, we define the direction of steepest descent as
\baq\label{eq:gradE minimization}
\jup^\ast :=  \argmin \left\{\diff\mathcal{F}(\rhoup)[\bar{\varjup}]\,\big|\, \bar{\varjup}\in T_\rhoup(\mathcal{P}_p(\Rd))^2, \,\text{s.t.}\, l_\rhoup(\bar{\varjup})[\bar{\varjup}] = 1\right\},
\eaq
if it exists.
\end{definition}
Note that $\diff \mathcal{F}(\rhoup)=0$ implies $\grad^-\mathcal{F}(\rhoup)=0$. Otherwise, the negative metric gradient determines the direction of steepest descent, as the next lemma shows.
\begin{lemma}\label{lem:neg grad steepest descent}
Let $\jup^\ast$ be as in \eqref{eq:gradE minimization}. Then, there exists $C>0$ such that $\jup^\ast = C\grad^-\mathcal{F}(\rhoup)$ holds.
\end{lemma} 
\begin{proof}
We argue similar to \cite[Subsection 3.1]{EPSS2021} and start by adding the constraint of the optimization problem with Lagrange multiplier $C\in\R$. For $\jup\in T_\rhoup(\mathcal{P}_p(\Rd))^2$ we define the functional
\baqs
    \mathcal{H}(C,\jup) \coloneqq \diff\mathcal{F}(\rhoup)[\jup]+\frac{C}{p}(l_\rhoup(\jup)[\jup]-1).
\eaqs
We employ that, by Proposition \ref{prop:l_derives_from_F}, for any $\jup, \bar{\varjup}\in T_\rhoup(\mathcal{P}_p(\Rd))^2$, we have the equality
\baqs
    \frac{\dd}{\dd \tau} l_\rhoup(\jup+ \tau\bar{\varjup})[\jup+ \tau\bar{\varjup}]\Big|_{\tau=0}=p l_\rhoup(\jup)[\bar{\varjup}].
\eaqs
Then, using the linearity of the differential $\diff\mathcal{F}(\rhoup)$, any minimizer $(C^\ast, \jup^\ast)$ of $\mathcal{H}$ must satisfy the condition
\baqs
    \diff\mathcal{F}(\rhoup)[\cdot] = -C^\ast l_{\rhoup}(\jup^\ast)[\cdot].
\eaqs
By the linearity of the map $\jup \mapsto -C^\ast l_{\rhoup}(j^\ast)[\jup]$, and the symmetry of the constraint, we find $0>\diff\mathcal{F}(\rhoup)[j^\ast]=-C^\ast l_\rhoup(\jup^\ast)[j^\ast]$, which implies $C^\ast> 0$. Thus, the previously proven injectivity and positive 1-homogeneity of $l_\rhoup$ yield
\baqs
j^\ast =l_\rhoup^{-1}\left(-\frac{1}{C^\ast}\diff\mathcal{F}(\rhoup)\right)=\frac{1}{C^\ast}l_\rhoup^{-1}(-\diff\mathcal{F}(\rhoup))=\frac{1}{C^\ast}\grad^-\mathcal{F}(\rhoup).
\eaqs
\end{proof}
In light of Lemma \ref{lem:neg grad steepest descent}, it makes sense to write metric gradient flows with respect to $\mathcal{F}$ in the Finsler space $((\mathcal{P}_p(\Rd))^2,\mathcal{T}_\betaup)$ as
\baqs
\partial_t\rhoup_t=\babla\cdot\grad^-\mathcal{F}(\rhoup_t).
\eaqs
Since the previous considerations did not use any specific structure of $\mathcal{F}$, they stay valid for general functionals $\mathcal{F}:(\mathcal{P}_p(\Rd))^2\to\R \cup \{+\infty\}$. However, even though by Theorem \ref{thm:grad unique} we know that the (negative) metric gradient of a functional is unique, we have yet to show its existence. For the case where $\mathcal{F}$ is the nonlocal cross-interaction energy \eqref{eq:Energy}, the following theorem ensures existence.
\begin{theorem}[Existence of the negative metric gradient for the nonlocal cross-interaction energy]\label{thm:grad existence}
Let $\mathcal{E}$ be the nonlocal cross-interaction energy. Then, for any $\rhoup\in(\mathcal{P}_p(\Rd))^2$ and $\eta(\mu\otimes\mu)$-a.e. the negative metric gradient $\dd\grad^-\mathcal{E}(\rhoup)=\sum_{i=1}^2(\grad^-\mathcal{E}(\rhoup))^{(i)}\dd(\mu\otimes\mu)+(\grad^-\mathcal{E}(\rhoup))^{(i)\perp}\dd\varsigma^{(i)}$ is given for $i=1,2$ by
\baq\label{eq_def:grad-}
    (\grad^-\mathcal{E}(\rhoup))^{(i)} &= \big(\mathfrak{m}^{(i)}\big)^\top\big((-\beta^{(i)}\babla\delta_{\rho^{(i)}}\mathcal{E}(\rhoup))_-\big)^{q-1} 
    -\mathfrak{m}^{(i)}\big((-\beta^{(i)}\babla\delta_{\rho^{(i)}}\mathcal{E}(\rhoup))_+\big)^{q-1},\\
    (\grad^-\mathcal{E}(\rhoup))^{(i)\perp} &= \big(\mathfrak{m}_\infty^{(i)}\big)^\top\big((-\beta^{(i)}\babla\delta_{\rho^{(i)}}\mathcal{E}(\rhoup))_-\big)^{q-1}-\mathfrak{m}_\infty^{(i)}\big((-\beta^{(i)}\babla\delta_{\rho^{(i)}}\mathcal{E}(\rhoup))_+\big)^{q-1},
\eaq
if $R=S=\infty$. If $R\land S<\infty$, then we have $(\grad^-\mathcal{E}(\rhoup))^{(i)\perp}=0$ for $i=1,2$.
\end{theorem}
\begin{proof}
We calculate the differential according to Definition \ref{def:diff and grad}. For $\rhoup\in(\mathcal{P}_p(\Rd))^2$ and $\jup\in T_\rhoup(\mathcal{P}_p(\Rd))^2$ take any curve $\bs{\tilde{\rhoup}}\in\AC^p([0,T];((\mathcal{P}_p(\Rd))^2,\mathcal{T}_{m,\betaup,\mu}))$, s.t. $\tilde{\rhoup}_0 = \rhoup$ and $\frac{\dd}{\dd t}\tilde{\rhoup}_t =-\babla\cdot \bar{\varjup}_t$ according to Definition \ref{def:cont eq} and $\bar{\varjup}_0=\jup$. Then, using the equality $K^{(21)}=K^{(12)}$ and Lemma \ref{lem:continuous representative}, we find
\baqs
-\diff \mathcal{E}(\rhoup)[\jup] &= -\frac{\dd}{\dd t} \mathcal{E}(\tilde{\rho}_t)\Big|_{t=0} = 
-\lim_{\tau\to 0} \frac{\mathcal{E}(\tilde{\rhoup}_\tau)-\mathcal{E}(\tilde{\rhoup}_0)}{\tau} \\
&= -\lim_{\tau\to 0}\frac{1}{2\tau}\sum_{i,k=1}^2\Bigg[\int_\Rd (K^{(ik)}\ast\tilde{\rho}_\tau^{(k)})(x)\dd\tilde{\rho}_\tau^{(i)}(x)-\int_\Rd  (K^{(ik)}\ast\tilde{\rho}_0^{(k)})(x)\dd\tilde{\rho}_0^{(i)}(x)\Bigg]\\
&= -\lim_{\tau\to 0}\frac{1}{2\tau}\sum_{i,k=1}^2\Bigg[\int_0^\tau\frac{\dd}{\dd t}\iint_\RdRd K^{(ik)}(x,y)\dd\tilde{\rho}_t^{(k)}(y)\dd\tilde{\rho}_t^{(i)}(x)\dd t\Bigg]\\
&=-\frac{1}{2}\sum_{i,k=1}^2\iint_G\babla(K^{(ik)}\ast\rho^{(k)})(x,y)\eta(x,y)\dd j^{(i)}(x,y)\\
&= \frac{1}{2}\sum_{i,k=1}^2\iint_G -\babla(K^{(ik)}\ast\rho^{(k)})(x,y)\eta(x,y)\dd j^{(i)}(x,y).
\eaqs
%
Since $\delta_{\rho^{(1)}}\mathcal{E}(\rhoup) = K^{(11)}\ast\rho^{(1)}+K^{(12)}\ast\rho^{(2)}$ and $\delta_{\rho^{(2)}}\mathcal{E}(\rhoup) = K^{(22)}\ast\rho^{(2)}+K^{(21)}\ast\rho^{(1)}$, and $\beta^{(i)}>0$ for $i=1,2$, we rewrite this in terms of the negative metric gradient of the energy functional:
\baqs
-\diff \mathcal{E}(\rhoup)[\jup]&=\frac{1}{2}\sum_{i=1}^2\Bigg[\iint_G -\babla\delta_{\rho^{(i)}}\mathcal{E}(\rhoup)(x,y)\eta(x,y)j^{(i)}(x,y)\dd \mu(x)\dd\mu(y)\\
&\hphantom{=\frac{1}{2}\sum_{i=1}^2\Bigg[}+\iint_G -\babla\delta_{\rho^{(i)}}\mathcal{E}(\rhoup)(x,y)\eta(x,y)j^{(i)\perp}(x,y)\dd \varsigma^{(i)}(x,y)\Bigg]\\
&=\frac{1}{2}\sum_{i=1}^2\frac{1}{\beta^{(i)}}\iint_G j^{(i)}(x,y)\Bigg(\frac{\mathfrak{m}^{(i)}(x,y)(\beta^{(i)}(-\babla\delta_{\rho^{(i)}}\mathcal{E}(\rhoup))_+(x,y))^{q-1}}{\mathfrak{m}^{(i)}(x,y)}\\
&\hphantom{=\frac{1}{2}\sum_{i=1}^2\iint_G j^{(i)\perp}(x,y)\Bigg(} -\frac{\mathfrak{m}^{(i)}(y,x)(\beta^{(i)}(-\babla\delta_{\rho^{(i)}}\mathcal{E}(\rhoup))_-(x,y))^{q-1}}{\mathfrak{m}^{(i)}(y,x)}\Bigg)^{p-1}\eta(x,y)\dd \mu(x)\dd\mu(y)\\
&\hphantom{=\frac{1}{2}\sum_{i=1}^2\frac{1}{\beta^{(i)}}}+\iint_G j^{(i)\perp}(x,y)\Bigg(\frac{\mathfrak{m}_\infty^{(i)}(x,y)(\beta^{(i)}(-\babla\delta_{\rho^{(i)}}\mathcal{E}(\rhoup))_+(x,y))^{q-1}}{\mathfrak{m}_\infty^{(i)}(x,y)}\\
&\hphantom{=\frac{1}{2}\sum_{i=1}^2\iint_G j^{(i)\perp}(x,y)\Bigg(} -\frac{\mathfrak{m}_\infty^{(i)}(y,x)(\beta^{(i)}(-\babla\delta_{\rho^{(i)}}\mathcal{E}(\rhoup))_-(x,y))^{q-1}}{\mathfrak{m}_\infty^{(i)}(y,x)}\Bigg)^{p-1}\eta(x,y)\dd \varsigma^{(i)}(x,y).
\eaqs
Comparing this expression with the definition of the gradient,  \eqref{eq_def:grad-} follows.
\end{proof}
\begin{remark}\label{rem:grad^-E_in_dense_subset_of_T_rho}
For $R\land S<\infty$ or $m_\infty\equiv 0$ the structure of the negative gradient closely resembles the structure in \eqref{eq:dense flux structure} with $\varphi^{(i)} = -\beta^{(i)}\delta_{\rho^{(i)}}\mathcal{E}(\rhoup)$ since $-\beta^{(i)}\babla\delta_{\rho^{(i)}}\mathcal{E}(\rhoup)=-\babla\beta^{(i)}\delta_{\rho^{(i)}}\mathcal{E}(\rhoup)$.
\end{remark}

\subsection{Variational characterization for the nonlocal nonlocal cross-interaction equation}\label{Variational characterization for the nonlocal nonlocal cross-interaction equation}

We now want to characterize \eqref{eq:NL2CIE} as a gradient flow in the sense of curves of maximal slope and start by defining the one-sided strong upper gradient.
\begin{definition}\label{def:One-sided strong upper gradient}(One-sided strong upper gradient).
A function $h:(\mathcal{P}_p(\Rd))^2\to[0,\infty]$ is called a \emph{one-sided strong upper gradient} for $\mathcal{E}$ if for every $\bs\rhoup\in \AC^p([0,T];((\mathcal{P}_p(\Rd))^2,\mathcal{T}_{m,\betaup,\mu}))$ the function $h\circ\bs\rhoup:[0,T]\to[0,\infty]$ is measurable and we have
\baqs
\mathcal{E}(\rhoup_{t})-\mathcal{E}(\rhoup_{s})\geq -\int_{s}^{t} h(\rhoup_\tau)\abs{\rhoup'_\tau}\dd\tau, \qquad\text{for all }0\leq s\leq t\leq T.
\eaqs
As before $\abs{\rhoup'_\tau}$ denotes the metric derivative of $\rhoup_\tau$ with respect to $\mathcal{T}_{m,\betaup,\mu}$.
\end{definition}
The one-sided strong upper gradient is sufficient to characterize curves of maximal slope:
\begin{definition}\label{def:Curve of maximal slope}(Curve of maximal slope).
Given a strong one-sided upper gradient $h$ for $\mathcal{E}$, a curve $\bs\rhoup\in \AC^p([0,T];((\mathcal{P}_p(\Rd))^2,\mathcal{T}_{m,\betaup,\mu}))$ is called a \emph{curve of maximal slope} for $\mathcal{E}$ with respect to $h$ if and only if
\baq\label{eq_def:max slope}
\mathcal{E}(\rhoup_{t})-\mathcal{E}(\rhoup_{s})+\int_{s}^{t} \frac{1}{q}(h(\rhoup_\tau))^q+\frac{1}{p}\abs{\rhoup'_\tau}^p\dd\tau\leq 0,\qquad\text{for all }0\leq {s}\leq {t}\leq T.
\eaq
\end{definition}
\begin{remark}\label{rem:link one-sided upper gradient - curve of maximal slope}
Note that inequality \eqref{eq_def:max slope} implies that $t\mapsto\mathcal{E}(\rhoup_t)$ is nonincreasing. Further, observe that by Young's inequality we immediately see that any strong one-sided upper gradient for $\mathcal{E}$ satisfies 
\baqs
\mathcal{E}(\rhoup_{t})-\mathcal{E}(\rhoup_{s})+\int_{s}^{t} \frac{1}{q}(h(\rhoup_\tau))^q+\frac{1}{p}\abs{\rhoup'_\tau}^p\dd\tau\geq 0,\qquad\text{for all }0\leq {s}\leq {t}\leq T,
\eaqs
i.e., if $\bs \rhoup$ is a curve of maximal slope for $\mathcal{E}$ with respect to its strong one-sided upper gradient $h$, then we have equality in \eqref{eq_def:max slope}.
\end{remark}
Our next goal is to derive a chain rule. However, we have seen in Theorem \ref{thm:grad existence}, the relation between $(\grad^-\mathcal{E}(\rhoup))^{(i)}$ and $-\beta^{(i)}\babla\delta_{\rho^{(i)}}\mathcal{E}(\rhoup)$ is not linear, but contains a power $q-1$. To account for this, we make the following definition:
\begin{definition}\label{def:tilde_l_rho}
    Given two maps $\vup=(v^{(1)},v^{(1)\perp},v^{(2)},v^{(2)\perp}),\bar{\vup}=(\bar{v}^{(1)},\bar{v}^{(1)\perp},\bar{v}^{(2)},\bar{v}^{(2)\perp}):G\to\R^4$, we define
    \baqs 
        \tilde{l}_\rhoup(\vup)[\bar{\vup}]&= \frac{1}{2}\sum_{i=1}^2\frac{1}{\beta^{(i)}}\Bigg[\iint_G \bar{v}^{(i)}\left(\mathfrak{m}^{(i)}\big(v^{(i)}_+(\rho)\big)^{q-1}-(\mathfrak{m}^{(i)})^\top \big(v^{(i)}_-\big)^{q-1}\right)\eta\dd(\mu\otimes\mu)\\
        &\hphantom{\frac{1}{2}\sum_{i=1}^2\frac{1}{\beta^{(i)}}\Bigg[}+\iint_G \bar{v}^{(i)\perp}\left(\mathfrak{m}_\infty^{(i)} \big(v^{(i)\perp}_+\big)^{q-1}-(\mathfrak{m}_\infty^{(i)})^\top \big(v^{(i)\perp}_-\big)^{q-1}\right)\eta\dd\varsigma^{(i)}\Bigg],
    \eaqs
    if $R=S=\infty$. For $R\land S<\infty$, we define $\tilde{l}$ by
    \baqs 
        \tilde{l}_\rhoup(\vup)[\bar{\vup}]&= \frac{1}{2}\sum_{i=1}^2\frac{1}{\beta^{(i)}}\iint_G \bar{v}^{(i)}\left(\mathfrak{m}^{(i)}(v^{(i)}_+)^{q-1}-(\mathfrak{m}^{(i)})^\top (v^{(i)}_-)^{q-1}\right)\eta\dd(\mu\otimes\mu).
    \eaqs
\end{definition}
\begin{remark}\label{rem:connection_l_and_tilde_l}
    Let $\vup, \bar \vup$ be associated to $\jup, \bar{\varjup}$ as in \eqref{eq:dj_mit_v} and \eqref{eq:dj^perp_mit_v}. Then, in general $l_\rhoup(\jup)[\bar{\varjup}]\ne\tilde{l}_\rhoup(\vup)[\bar{\vup}]$. However, $\tilde{l}_\rhoup(\vup)[\bar{\vup}]$ is linear in $\bar{\vup}$ and it still holds that 
    \baq\label{eq:l_tildeg=A_tilde}
        \tilde{l}_\rhoup(\vup)[\vup] = \tilde{\A}_{m, \betaup}(\mu;\rhoup,\vup) = \A_{m, \betaup}(\mu;\rhoup,\jup) = l_\rhoup(\jup)[\jup].
    \eaq
\end{remark}
With this machinery in place, we can now adapt \eqref{eq:continuous representative} as follows:
\begin{lemma}[Chain rule for test functions]\label{lem:chain_rule_test_function}
    For $\bs\rhoup \in \AC^p([0,T];((\mathcal{P}_p(\Rd))^2, \mathcal{T}_{m,\betaup,\mu}))$ let $\jupbold\subset T_\rhoup(\mathcal{P}_p(\Rd))^2$ such that $(\bs\rhoup,\jupbold)\in \CE_T$ and $\abs{\rhoup'_t}^p=\A_{m,\betaup}(\mu;\rhoup_t;\jup_t)$ for a.e. $t\in[0,T]$, as given in Proposition \ref{prop:met vel}. For this $\jupbold$, let $\vupbold= (\vup_t)_t, \vup_t=(v_t^{(1)},v_t^{(1)\perp},v_t^{(2)},v_t^{(2)\perp}):G\to\R^4$ be as in \eqref{eq:dj_mit_v} and \eqref{eq:dj^perp_mit_v}. Then, for any $\varphiup = (\varphi^{(1)},\varphi^{(1)\perp},\varphi^{(2)},\varphi^{(2)\perp})\in (C_c^\infty(\Rd))^4$, $i=1,2$, $0\leq s\leq t\leq T$ and $i=1,2$ it holds
    \baqs
        \sum_{i=1}^2\int_\Rd\varphi^{(i)}(x)\dd\rho^{(i)}_{t}(x)-\int_\Rd\varphi^{(i)}(x)\dd\rho^{(i)}_{s}(x)=\int_{s}^{t}\tilde{l}_{\rhoup_\tau}( v_\tau)[\betaup\babla\varphi]\dd\tau.
    \eaqs
\end{lemma}
\begin{proof}
    Let $i\in\{1,2\}$. Starting from the continuity equation  \eqref{eq:cont} we calculate
    \baqs
        &\int_\Rd\varphi^{(i)}(x)\dd\rho^{(i)}_{t}(x)-\int_\Rd\varphi^{(i)}(x)\dd\rho^{(i)}_{s}(x) =\frac{1}{2}\int_{s}^{t}\iint_G\babla\varphi(x,y)\eta(x,y)\dd j^{(i)}_\tau(x,y)\dd\tau\\
        =&\;\frac{1}{2\beta^{(i)}}\int_{s}^{t}\iint_G\beta^{(i)}\babla\varphi^{(i)}(x,y)\eta(x,y)\left(\big(\big(v^{(i)}_\tau\big)_+(x,y)\big)^{q-1}\dd\gamma_{1,\tau}^{(i)}(x,y)-\big(\big(v^{(i)}_\tau\big)_-(x,y)\big)^{q-1}\dd\gamma_{2,\tau}^{(i)}(x,y)\right)\dd\tau.
    \eaqs
    From this, we conclude by summing over both species.
\end{proof}
As for $l$, for $\tilde{l}$ we too have a H{\"o}lder-type inequality:
\begin{lemma}[H{\"o}lder-type inequality]\label{lem:Holder-type_inequality_for_tilde_l}
For all $v,\bar{v}\in \tilde{T}_\rhoup(\mathcal{P}_p(\Rd))^2$ we have
\baq\label{eq:Holder-type_inequality_for_tilde_l}
\tilde{l}_\rhoup(\vup)[\bar{\vup}] \leq (\tilde{l}_\rhoup(\vup)[\vup])^{1/p}(\tilde{l}_\rhoup(\bar{\vup})[\bar{\vup}])^{1/q},
\eaq
with equality if and only if, for some $\lambda >0$, for $i=1,2$ we have $v^{(i)}_+ =\lambda \bar{v}^{(i)}_+$, $\eta\gamma_1^{(i)}$-a.e. as well as $v^{(i)\perp}_+ =\lambda \bar{v}^{(i)\perp}_+$, $\eta\gamma_1^{(i)\perp}$-a.e. and hence, by antisymmetry, also $v^{(i)}_- =\lambda \bar{v}^{(i)}_-$, $\eta\gamma_2^{(i)}$-a.e. as well as $v^{(i)\perp}_- =\lambda \bar{v}^{(i)\perp}_-$, $\eta\gamma_2^{(i)\perp}$-a.e.
\end{lemma}
\begin{proof}
The argument is analogous to that for $l_\rhoup(\jup)[\bar{\varjup}]$ in the proof of Lemma \ref{lem:Holder-type_inequality_for_l}.
\end{proof}

\begin{definition}\label{def:Dissipation_and_De_Giorgi_functional}(Dissipation and De Giorgi functional).
For $\rhoup \in (\mathcal{P}_p(\Rd))^2$, we define the \emph{dissipation} at $\rho$ by
\baqs
\mathcal{D}(\rhoup)\coloneqq\tilde{\A}_{m, \betaup}(\mu;\rhoup,-\betaup\babla\delta_\rhoup\mathcal{E}(\rhoup)) = 
\tilde{l}_\rho(-\betaup\babla\delta_\rhoup\mathcal{E}(\rhoup)[-\betaup\babla\delta_\rhoup\mathcal{E}(\rhoup)],
\eaqs
where $\delta_\rhoup\mathcal{E}(\rhoup)\coloneqq (\delta_{\rho^{(1)}}\mathcal{E}(\rhoup),\delta_{\rho^{(2)}}\mathcal{E}(\rhoup),\delta_{\rho^{(1)}}\mathcal{E}(\rhoup),\delta_{\rho^{(2)}}\mathcal{E}(\rhoup))$. For any $\bs\rhoup \in \AC^p([0,T];((\mathcal{P}_p(\Rd))^2,\mathcal{T}_{m,\betaup,\mu}))$, we define the \emph{De Giorgi functional} at $\bs\rhoup$ by
\baqs
\mathcal{G}_T(\bs\rhoup)\coloneqq \mathcal{E}(\rhoup_T)-\mathcal{E}(\rhoup_0)+\int_0^T \frac{1}{q}\mathcal{D}(\rhoup_t)+\frac{1}{p}\abs{\rhoup'_t}^p\dd t.
\eaqs
When the dependence on the base measure needs to be made explicit, we write $\mathcal{D}(\mu;\bs\rhoup)$ and $\mathcal{G}_T(\mu;\bs\rhoup)$.
\end{definition}

\subsection{Characterization of weak solutions}
In this subsection we show that weak solutions of \eqref{eq:NL2CIE} can be characterized as minimizers for the De Giorgi functional $\mathcal{G}_T$ introduced in Definition \ref{def:Dissipation_and_De_Giorgi_functional}. To achieve this, we need the chain rule for the gradient velocity of $\mathcal{E}$. Its proof is based on a mollification and truncation argument (and can be found in Appendix \ref{A:Chain}). 

\begin{proposition}[Chain rule for $\mathcal{E}$]\label{prop:Chain rule}
Let $K^{(ik)}$, $i,k=1,2$ satisfy \eqref{K2}, \eqref{K3} and $K^{(21)} = K^{(12)}$, let $\bs\rhoup \subset \AC^p([0,T];((\mathcal{P}_p(\Rd))^2,\mathcal{T}_{m,\betaup,\mu}))$ and let $0\leq {s}\leq {t}\leq T$. Denote by $\bs v$ the unique velocity in $\tilde{T}_{\bs\rhoup}(\mathcal{P}_p(\Rd))^2$, which is associated to $\bs\rhoup$ by Proposition \ref{prop:met vel} and Lemma \ref{lem:connection_flux_velocity}. Then, we have the chain rule identity
\baq\label{eq:chain rule velocity}
    \mathcal{E}(\rhoup_{t})-\mathcal{E}(\rhoup_{s})= \int_{s}^{t}\tilde{l}_{\rhoup_\tau}(\vup_\tau)[\betaup\babla\delta_\rhoup\mathcal{E}(\rhoup_\tau)]\dd\tau.
\eaq
\end{proposition}
Using the chain rule, we infer that $\mathcal{D}^{1/q}$ is a one-sided strong upper gradient for $\mathcal{E}$.
\begin{corollary}\label{cor sqrtD is one-sided strong upper gradient}
For any curve $\bs\rhoup \in \AC^p([0,T];((\mathcal{P}_p(\Rd))^2,\mathcal{T}_{m,\betaup,\mu}))$ and any $0\leq {s} \leq {t}\leq T$ we have
\baqs
\mathcal{E}(\rhoup_{t})-\mathcal{E}(\rhoup_{s})\geq -\int_{s}^{t}(\mathcal{D}(\rhoup_\tau))^{1/q}\abs{\rhoup'_\tau}\dd\tau,
\eaqs
i.e., $\mathcal{D}^{1/q}$ is a one-sided strong upper gradient for $\mathcal{E}$ in the sense of Definition \ref{def:One-sided strong upper gradient}.
\end{corollary}
\begin{proof}
Without loss of generality, assume that $\int_{s}^{t}(\mathcal{D}(\rhoup_\tau))^{1/q}\abs{\rhoup'_\tau}\dd\tau < \infty$ as otherwise there is nothing to show. We employ \eqref{eq:chain rule velocity} from Proposition \ref{prop:Chain rule} and apply the H{\"o}lder-type inequality from Lemma \ref{lem:Holder-type_inequality_for_tilde_l}. For $0\leq {s}\leq {t} \leq T$, we have
\baqs
\mathcal{E}(\rhoup_{t})-\mathcal{E}(\rhoup_{s})&= \int_{s}^{t}\tilde{l}_{\rhoup_\tau}(\vup_\tau)[\betaup\babla\delta_\rhoup\mathcal{E}(\rhoup_\tau)]\dd\tau =-\int_{s}^{t}\tilde{l}_{\rhoup_\tau}(\vup_\tau)[-\betaup\babla\delta_\rhoup\mathcal{E}(\rhoup_\tau)]\dd\tau\\
&\ge -\int_{s}^{t}(\tilde{l}_{\rhoup_\tau}(\vup_\tau)[\vup_\tau])^{1/p}(\tilde{l}_{\rhoup_\tau}(-\betaup\babla\delta_\rhoup\mathcal{E}(\rhoup_\tau))[-\betaup\babla\delta_\rhoup\mathcal{E}(\rhoup_\tau)])^{1/q}\dd\tau\\
&=-\int_{s}^{t}(\tilde{\A}_{m,\betaup}(\rhoup_\tau,\vup_\tau))^{1/p}(\mathcal{D}(\rhoup_\tau))^{1/q}\dd\tau =-\int_{s}^{t}\abs{\rhoup'_\tau}(\mathcal{D}(\rhoup_\tau))^{1/q}\dd\tau.
\eaqs
Here, the last two inequalities are provided by \eqref{eq:l_tildeg=A_tilde} and Proposition \ref{prop:met vel}.
\end{proof}
Now we are ready to identify weak solutions to \eqref{eq:NL2CIE} as minimizers of $\mathcal{G}_T$.
\begin{theorem}[Characterization of weak solutions to the nonlocal nonlocal cross-interaction system]\label{thm:characterization of weak solutions to NL2CIE}
Suppose $\mu$ satisfies \eqref{MB2}, \eqref{MB} and \eqref{BC}, and the kernels $K^{(ik)}$ satisfy \eqref{K2}, \eqref{K3} for $i,k=1,2$ as well as $K^{(21)}=K^{(12)}$. A curve $\bs\rhoup:[0,T]\to(\mathcal{P}_p(\Rd))^2$ is a weak solution of \eqref{eq:NL2CIE} according to Definition \ref{def:weak solution to NL2CIE} if and only if $\bs\rhoup \in \AC^p([0,T];((\mathcal{P}_p(\Rd))^2,\mathcal{T}_{m,\betaup,\mu}))$ is a curve of maximal slope for $\mathcal{E}$ with respect to $(\mathcal{D}(\bs\rhoup))^{1/q}$ in the sense of Definition \ref{def:Curve of maximal slope}, i.e., it satisfies
\baq\label{eq:characterization of weak solutions to NL2CIE}
\mathcal{G}_T(\bs\rhoup) = 0,
\eaq
where $\mathcal{G}_T$ is the De Giorgi functional given in Definition \ref{def:Dissipation_and_De_Giorgi_functional}.
\end{theorem}
\begin{proof}
Assume that $\bs\rhoup$ is a weak solution to \eqref{eq:NL2CIE} according to Definition \ref{def:weak solution to NL2CIE}. To construct a weak solution for the continuity equation \eqref{eq:cont}, we define the flux $\jupbold$ by
\baqs
    \dd j_t^{(i)\mu}&=(\beta^{(i)}(\babla\delta_{\rho^{(i)}}\mathcal{E}(\rhoup))_-)^{q-1}\dd\gamma_{1,t}^{(i)}-(\beta^{(i)}(\babla\delta_{\rho^{(i)}}\mathcal{E}(\rhoup))_+)^{q-1}\dd\gamma_{2,t}^{(i)},\\
    \dd j_t^{(i)\perp}&=(\beta^{(i)}(\babla\delta_{\rho^{(i)}}\mathcal{E}(\rhoup))_-)^{q-1}\dd\gamma_{1,t}^{(i)\perp}-(\beta^{(i)}(\babla\delta_{\rho^{(i)}}\mathcal{E}(\rhoup))_+)^{q-1}\dd\gamma_{2,t}^{(i)\perp},
\eaqs
for $i=1,2$. Using the abbreviation $v_t^{\mathcal{E},(i)} \coloneqq \beta^{(i)}\babla\delta_{\rho^{(i)}}\mathcal{E}(\rhoup)$, we immediately obtain
\baqs
\int_0^T\A_{m, \betaup}(\mu ; \rhoup_t,\jup_t)\dd t = \int_0^T\tilde{\A}_{m, \betaup}(\mu;\rhoup_t,\vup_t^{\mathcal{E}})\dd t=\int_0^T\mathcal{D}(\rhoup_t)\dd t<\infty.
\eaqs
The first two equalities are clear from the definitions. For the finiteness, recall that due to the concavity and finiteness of the mobility $m$, for any $B\in\mathcal{B}(G)$ we have the bound 
\baqs
    (\gamma_{1,\tau}^{(i)}+\gamma_{1,\tau}^{(i)\perp})(B) \le M(\mu\otimes\mu+\rho^{(i)}\otimes\mu+\mu\otimes\rho^{(i)})(B),
\eaqs
where $M$ only depends on $m$ and $G$. With this, Jensen's inequality, \eqref{K3}, \eqref{MB2} and \eqref{MB}, we obtain
\baqs
\mathcal{D}(\rhoup_t)&=\sum_{i=1}^2(\beta^{(i)})^{q-1}\Bigg[\iint_G\left(\left(\sum_{k=1}^2\babla(K^{(ik)}\ast\rho_t^{(k)})\right)_-\right)^q\eta\dd\gamma_{1,t}^{(i)}+\iint_G\left(\left(\sum_{k=1}^2\babla(K^{(ik)}\ast\rho_t^{(k)})\right)_-\right)^q\eta\dd\gamma_{1,t}^{(i)\perp}\Bigg]\\
&\leq ML_K^q\sum_{i,k=1}^2(\beta^{(i)})^{q-1}\int_\Rd\iint_G\left(\abs{x-y}^q\lor\abs{x-y}^{pq}\right)\eta(x,y)\dd\mu(y)\dd\rho_t^{(k)}(x)\dd(\mu+2\rho_t^{(i)})(z)\\
&\leq ML_K^qC_\eta\sum_{i,k=1}^2(\beta^{(i)})^{q-1}\int_\Rd\int_\Rd\dd\rho_t^{(k)}(x)\dd(\mu+2\rho_t^{(i)})(z)\\
&= 2(C_\mu+2)\big((\beta^{(1)})^{q-1}+(\beta^{(2)})^{q-1}\big)ML_K^qC_\eta<\infty.
\eaqs
By Proposition \ref{prop:met vel}, this also proves that $\bs\rhoup\in\AC^p([0,T];((\mathcal{P}_p(\Rd))^2,\mathcal{T}_\betaup))$ and that $\abs{\rhoup'_t}^p\le \mathcal{D}(\rhoup_t)$ for a.e. $t \in [0,T]$. The latter together with Proposition \ref{prop:Chain rule} yields
\baqs
    \mathcal{E}(\rhoup_t)-\mathcal{E}(\rhoup_s)&= \int_{s}^{t}\tilde{l}_{\rhoup_\tau}(\vup^\mathcal{E}_\tau)[\betaup\babla\delta_\rhoup\mathcal{E}(\rhoup_\tau)]\dd\tau=-\int_{s}^{t}\tilde{l}_{\rhoup_\tau}(\vup^\mathcal{E}_\tau)[-\betaup\babla\delta_\rhoup\mathcal{E}(\rhoup_\tau)]\dd\tau\\
    &=-\int_{s}^{t}\tilde{l}_{\rhoup_\tau}(-\betaup\babla\delta_\rhoup\mathcal{E}(\rhoup_\tau))[-\betaup\babla\delta_\rhoup\mathcal{E}(\rhoup_\tau)]\dd\tau\\
    &=-\int_{s}^{t}\mathcal{D}(\rhoup_\tau)\dd\tau\leq-\int_{s}^{t}\frac{1}{q}\mathcal{D}(\rhoup_\tau)+\frac{1}{p}\abs{\rhoup'_\tau}^p\dd\tau.
\eaqs
Hence, Corollary \ref{cor sqrtD is one-sided strong upper gradient} in conjunction with Remark \ref{rem:link one-sided upper gradient - curve of maximal slope} yields
\baqs
\mathcal{E}(\rhoup_t)-\mathcal{E}(\rhoup_s)+\int_{s}^{t}\frac{1}{q}\mathcal{D}(\rhoup_\tau)+\frac{1}{p}\abs{\rhoup'_\tau}^p\dd\tau = 0.
\eaqs
Therefore, the first implication of the theorem follows for the choices $s=0$ and $t=T$ implying $\mathcal{G}_T(\bs\rhoup)=0$.

To prove the converse implication, now consider $\bs\rhoup \in \AC^p([0,T];((\mathcal{P}_p(\Rd))^2,\mathcal{T}_\betaup))$ satisfying \eqref{eq:characterization of weak solutions to NL2CIE}. We verify that $\bs\rhoup$ is a weak solution of \eqref{eq:NL2CIE} according to Definition \ref{def:weak solution to NL2CIE}. By Proposition \ref{prop:met vel}, there exists a unique family $\jupbold\subset T_{\bs\rhoup}((\mathcal{P}(\Rd))^2)$,  such that $(\bs\rhoup,\jupbold)\in\CE_T$, $\int_0^T\A_{m, \betaup}^{1/p}(\rhoup_t,\jup_t)\dd t <\infty$ and $\abs{\rhoup'_t}^p=\A_{m, \betaup}(\rhoup_t,\jup_t)$, for a.e. $t\in[0,T]$. Moreover, by Lemma \ref{lem:connection_flux_velocity} we find a family of antisymmetric measurable vector fields $\vupbold=(\bs v^{(1)},\bs v^{(1)\perp},\bs v^{(2)},\bs v^{(2)\perp}):[0,T]\times G\to\R^4$ such that for every $t\in[0,T]$ and $i=1,2$ we have
\baqs
\dd j^{(i)\mu} &= (v^{(i)}_+)^{q-1}\dd\gamma_1^{(i)}-(v^{(i)}_-)^{q-1}\dd\gamma_2^{(i)},\\
\dd j^{(i)\perp} &= (v^{(i)\perp}_+)^{q-1}\dd\gamma_1^{(i)\perp}-(v^{(i)\perp}_-)^{q-1}\dd\gamma_2^{(i)\perp}.
\eaqs
Employing Proposition \ref{prop:Chain rule}, the H{\"o}lder-type inequality \eqref{eq:Holder-type_inequality_for_tilde_l}, the identity \eqref{eq:l_tildeg=A_tilde}, Definition \ref{def:Dissipation_and_De_Giorgi_functional}, and Young's inequality, we obtain
\baqs
    \mathcal{E}(\rhoup_T)-\mathcal{E}(\rhoup_0)&= \int_0^T\tilde{l}_{\rhoup_\tau}(\vup_\tau)[\betaup\babla\delta_\rhoup\mathcal{E}(\rhoup_\tau)]\dd\tau=-\int_{s}^{t}\tilde{l}_{\rhoup_\tau}(\vup_\tau)[-\betaup\babla\delta_\rhoup\mathcal{E}(\rhoup_\tau)]\dd\tau\\
    &\geq-\int_0^T(\A_{m, \betaup}(\rhoup_\tau,\jup_\tau))^{1/p}(\mathcal{D}(\rhoup_\tau))^{1/q}\dd\tau = -\int_0^T\abs{\rhoup'_\tau}(\mathcal{D}(\rhoup_\tau))^{1/q}\dd\tau\\
    &\geq -\int_0^T\frac{1}{q}\mathcal{D}(\rhoup_\tau)+\frac{1}{p}\abs{\rhoup'_\tau}^p\dd\tau.
\eaqs
Equation \eqref{eq:characterization of weak solutions to NL2CIE} implies that the inequalities are actually equalities. By Lemma \ref{lem:Holder-type_inequality_for_l}, equality holds if and only if for $i=1,2$ and a.e. $t\in[0,T]$ we have 
\baqs
    (v^{(i)}_t)_+ &= -\beta^{(i)}\babla\delta_{\rho^{(i)}}\mathcal{E}(\rhoup_t)_+, \quad \gamma_{1,t}^{(i)}\text{-a.e. on }G,\qquad (v^{(i)}_t)_- = -\beta^{(i)}\babla\delta_{\rho^{(i)}}\mathcal{E}(\rhoup_t)_-, \quad \gamma_{2,t}^{(i)}\text{-a.e. on }G,\\
    (v^{(i)\perp}_t)_+ &= -\beta^{(i)}\babla\delta_{\rho^{(i)}}\mathcal{E}(\rhoup_t)_+,  \quad \gamma_{1,t}^{(i)\perp}\text{-a.e. on }G, \qquad (v^{(i)\perp}_t)_- = -\beta^{(i)}\babla\delta_{\rho^{(i)}}\mathcal{E}(\rhoup_t)_-,  \quad  \gamma_{2,t}^{(i)\perp}\text{-a.e. on }G 
\eaqs
Hence, $(\bs\rhoup,\jupbold)\in\CE_T$ is a weak solution to \eqref{eq:NL2CIE}.
\end{proof}
\subsection{Stability and existence of weak solutions}\label{Stability and existence of weak solutions}
In this section, we utilize the characterization of weak solutions to \eqref{eq:NL2CIE} as minimizers of $\mathcal{G}_T$, attaining $\mathcal{G}_T=0$. To show the existence of minimizers, we employ the direct method of calculus of variations. This way, we will prove the compactness and stability of gradient flows, which we will then utilize to approximate the desired problem by discrete problems. The existence of solutions is easy to show.

\begin{lemma}\label{lem:lsc D}
Let $(\mu^n)_{n\in\mathbb{N}}\subset \Mloc^+(\Rd)$ and suppose $\mu^n \rightharpoonup^\ast \mu$ for some $\mu\in\Mloc^+(\Rd)$ as $n\to\infty$. Assume that $\mu^n$ and $\mu$ safisfy \eqref{MB2}, \eqref{MB} and \eqref{BC} uniformly in $n$. For $i,k=1,2$, let $K^{(ik)}$ satisfy \eqref{K2}, \eqref{K3} and $K^{(21)}=K^{(12)}$. Moreover, let $(\rhoup^n)_{n\in\mathbb{N}}$ be a sequence in $(\mathcal{P}_p(\Rd))^2$, which satisfies $\sup_{n\in\mathbb{N}}M_p(\rho^{n,(i)})<\infty$ and is such that $\rhoup^n\rightharpoonup \rhoup$ for some $\rhoup \in (\mathcal{P}_p(\Rd))^2$, as $n\to\infty$. Then, we have
\baqs
\liminf_{n\to\infty}\mathcal{D}(\mu^n;\rhoup^n)\geq \mathcal{D}(\mu;\rhoup).
\eaqs
\end{lemma}
\begin{proof}
For every $n\in\mathbb{N}$ and $i=1,2$, we define $u^{n,(i)} \coloneqq \beta^{(i)}\sum_{k=1}^2 \babla(K^{(ik)}\ast\rho^{n,(k)})$ and $u^{(i)} \coloneqq \beta^{(i)}\sum_{k=1}^2\babla( K^{(ik)}\ast\rho^{(k)})$. Further, we define the convex and continuous map $f:\R\to\R,r\mapsto(r_-)^q$ and note that we have
\baqs
\mathcal{D}(\mu^n;\rhoup^n)&=\sum_{i=1}^2\frac{1}{\beta^{(i)}}\iint_G f(u^{n,(i)})\eta\dd(\gamma_1^{n,(i)}+\gamma_1^{n,(i)\perp}),\\
\mathcal{D}(\mu;\rhoup)&=\sum_{i=1}^2\frac{1}{\beta^{(i)}}\iint_G f(u^{(i)})\eta\dd(\gamma_1^{(i)}+\gamma_1^{(i)\perp}),
\eaqs
where $\gamma_k^{(i)}$ and $\gamma_k^{(i)\perp}$ are as in Lemma \ref{def:V^as} and Remark \ref{rem:V^as}. We want to employ \cite[Theorem 5.4.4 (ii)]{AmbrosioGigliSavare2008} to prove the desired inequality. To this end, we observe that $u^{(i)}\in L^q(\eta\gamma_1^{(i)})$ and $u^{n,(i)}\in L^q(\eta\gamma_1^{n,(i)})$. Indeed, \eqref{K3}, \eqref{MB2} and \eqref{MB} and the bound
\baqs
    (\gamma_1^{n,(i)}+\gamma_1^{n,(i)\perp})(B)\le M(\mu\otimes\mu+\rho^{n,(i)}\otimes\mu+\mu\otimes\rho^{n,(i)})(B)\quad\forall B\in\mathcal{B}(G),
\eaqs
from Lemma \ref{lem:bound_by_A} imply
\baqs
    &\,\iint_G {\abs{u^{n,(i)}(x,y)}}^q\eta(x,y)\dd\gamma_1^{n,(i)}(x,y)\\
    =&\,\left(\beta^{(i)}\right)^q\iint_G {\abs{\sum_{k=1}^2K^{(ik)}\ast\rho^{n,(k)}(y)-K^{(ik)}\ast\rho^{n,(k)}(x)}}^q\eta(x,y)\dd(\gamma_1^{n,(i)}+\gamma_1^{n,(i)\perp})(x,y)\\
\leq&\, 2M(C_\mu+2)\left(\beta^{(i)}\right)^q L_K^q C_\eta.
\eaqs
Now, let $\varphi\in C_c^\infty(G)$. We find for $i=1,2$:
\baqs
&\,\left(\beta^{(i)}\right)^{-q}\iint_G u^{n,(i)}(x,y)\varphi(x,y)\eta(x,y)\dd(\gamma_1^{n,(i)}+\gamma_1^{n,(i)\perp})(x,y)\\
=&\,\sum_{k=1}^2\iint_G\int_\Rd \left(K^{(ik)}(y,z)- K^{(ik)}(x,z)\right)\dd\rho^{n,(k)}(z)\varphi(x,y)\eta(x,y)\dd(\gamma_1^{n,(i)}+\gamma_1^{n,(i)\perp})(x,y)\\
=&\,\sum_{k=1}^2\iint_{\supp\varphi}\int_{\Rd\cap B_R} \left(K^{(ik)}(y,z)- K^{(ik)}(x,z)\right)\dd\rho^{n,(k)}(z)\varphi(x,y)\eta(x,y)\dd(\gamma_1^{n,(i)}+\gamma_1^{n,(i)\perp})(x,y)\\
+&\,\sum_{k=1}^2\iint_{\supp\varphi}\int_{\Rd\setminus B_R} \left(K^{(ik)}(y,z)- K^{(ik)}(x,z)\right)\dd\rho^{n,(k)}(z)\varphi(x,y)\eta(x,y)\dd(\gamma_1^{n,(i)}+\gamma_1^{n,(i)\perp})(x,y).
\eaqs
The terms which are integrated over $\Rd\setminus B_R$ vanish as $R\to\infty$ since $\rho^{n,(k)}(\Rd\setminus B_R)\xrightarrow{R\to\infty}0$ by Prokhorov's Theorem. \eqref{K3} together with \eqref{MB2} and \eqref{MB} yields
\baqs
&\,{\abs{\iint_{\supp\varphi}\int_{\Rd\setminus B_R} \left(K^{(ik)}(y,z)- K^{(ik)}(x,z)\right)\varphi(x,y)\eta(x,y)\dd(\rho^{n,(k)}\otimes(\gamma_1^{n,(i)}+\gamma_1^{n,(i)\perp}))(z,x,y)}}\\
\leq &\,M L_K{\norm{\varphi}_\infty}\rho^{n,(k)}(\Rd\setminus B_R)\iint_{\supp\varphi}({\abs{x-y}\lor\abs{x-y}^p})\eta(x,y)\dd(\mu+2\rho^{n,(i)})(x)\dd\mu^n(y)\\
\leq &\,\frac{(C_\mu+2)M L_K C_\eta\norm{\varphi}_\infty\rho^{n,(k)}(\Rd\setminus B_R)}{\inf_{\supp\varphi}(\abs{x-y}^{q/p}\lor\abs{x-y}^q)}.
\eaqs
By \eqref{W} and \eqref{K3}, the function $(z,y,x)\mapsto(K(y,z)-K(x,z))\varphi(x,y)\eta(x,y)$ is continuous and bounded on $(\Rd\cap B_R)\times G$. On the other hand we have $\rho^{n,(k)}\otimes(\gamma_1^{n,(i)}+\gamma_1^{n,(i)\perp})\rightharpoonup \rho^{(k)}\otimes(\gamma_1^{(i)}+\gamma_1^{(i)\perp})$ in $\mathcal{P}(\Rd)\times\Mloc^+(G)$ for $i,k=1,2$. Therefore, for any $R>0$ and $i,k=1,2$, we obtain
\baqs
\lim_{n\to\infty}&\,\iint_{\supp\varphi}\int_{\Rd\cap B_R} \left(K^{(ik)}(y,z)- K^{(ik)}(x,z)\right)\varphi(x,y)\eta(x,y)\dd(\rho^{n,(k)}\otimes(\gamma_1^{n,(i)}+\gamma_1^{n,(i)\perp}))(z,x,y)\\
=&\,\iint_{\supp\varphi}\int_{\Rd\cap B_R} \left(K^{(ik)}(y,z)- K^{(ik)}(x,z)\right)\varphi(x,y)\eta(x,y)\dd(\rho^{(k)}\otimes(\gamma_1^{(i)}+\gamma_1^{(i)\perp}))(z,x,y).
\eaqs
Letting $R\to\infty$, we obtain
\baqs
\lim_{n\to\infty}\iint_G u^{n,(i)}\varphi\eta\dd(\gamma_1^{n,(i)}+\gamma_1^{n,(i)\perp})=\iint_G u^{(i)}\varphi\eta\dd(\gamma_1^{(i)}+\gamma_1^{(i)\perp}).
\eaqs
Therefore, $u^{n,(i)}$ converges weakly to $u^{(i)}$ in the sense of \cite[Definition 5.4.3]{AmbrosioGigliSavare2008}. This allows the application of \cite[Theorem 5.4.4 (ii)]{AmbrosioGigliSavare2008} to conclude
\baqs
\liminf_{n\to\infty}\mathcal{D}(\mu^n;\rhoup^n)&=\liminf_{n\to\infty}\sum_{i=1}^2\frac{1}{\beta^{(i)}}\iint_G f(u^{n,(i)})\eta\dd(\gamma_1^{n,(i)}+\gamma_1^{n,(i)\perp})\\
&\ge \sum_{i=1}^2\frac{1}{\beta^{(i)}}\iint_G f(u^{(i)})\eta\dd(\gamma_1^{(i)}+\gamma_1^{(i)\perp}) = \mathcal{D}(\mu;\rhoup),
\eaqs
which finishes the proof.
\end{proof}
\begin{lemma}[Compactness and lower semicontinuity of the De Giorgi functional]\label{lem:Compactness and lower semicontinuity of the De Giorgi functional}
Let $(\mu^n)_{n\in\mathbb{N}}\subset \Mloc^+(\Rd)$ and suppose $\mu^n \rightharpoonup^\ast \mu$ for some $\mu\in\Mloc^+(\Rd)$ as $n\to\infty$. Assume that $\mu^n$ and $\mu$ safisfy \eqref{MB2}, \eqref{MB} and \eqref{BC} uniformly in $n$. For $i,k=1,2$, let $K^{(ik)}$ satisfy \eqref{K2}, \eqref{K3} and $K^{(21)}=K^{(12)}$. Moreover, let $(\bs\rhoup^n)_{n\in\mathbb{N}}$ be such that $\bs\rhoup^n\in\AC^p([0,T];((\mathcal{P}_p(\Rd))^2,\mathcal{T}_{m,\betaup,\mu^n}))$, for all $n\in\mathbb{N}$ with $\sup_{n\in\mathbb{N}}M_p(\rho^{n,(i)}_0)<\infty$ for $i=1,2$ and $\sup_{n\in\mathbb{N}}\mathcal{G}_T(\mu^n;\bs\rhoup^n)<\infty$. Then, there exists $\bs\rhoup \in \AC^p([0,T];((\mathcal{P}_p(\Rd))^2,\mathcal{T}_{m,\betaup,\mu}))$ such that up to a subsequence we have $\rhoup^n_t\rightharpoonup \rhoup_t$ as $n\to\infty$ for all $t\in[0,T]$ and it holds
\baqs
\liminf_{n\to\infty}\mathcal{G}_T(\mu^n;\bs\rhoup^n)\geq \mathcal{G}_T(\mu;\bs\rhoup).
\eaqs
\end{lemma}
\begin{proof}
Let $n\in\mathbb{N}$. Recall
\baqs
\mathcal{G}_T(\mu^n;\bs\rhoup^n)= \mathcal{E}(\rhoup^n_T)-\mathcal{E}(\rhoup_0^n)+\int_0^T \frac{1}{q}\mathcal{D}(\mu^n;\rhoup^n_t)+\frac{1}{p}\abs{(\rhoup^n_t)'}^p_{\mathcal{T}_{m,\betaup,\mu^n}}\dd t,
\eaqs
where the metric derivative of $\rhoup^n_t$ is taken with respect to $\mathcal{T}_{m,\betaup,\mu^n}$. Since the domain of $\mathcal{E}$ is all of $(\mathcal{P}_p(\Rd))^2$ and $\mathcal{D}$ is nonnegative, the bound $\sup_{n\in\mathbb{N}}\mathcal{G}_T(\mu^n;\bs\rhoup^n)<\infty$ ensures that
\baqs
\sup_{n\in\mathbb{N}}\int_0^T \abs{(\rhoup^n_t)'}^p_{\mathcal{T}_{m,\betaup,\mu^n}} \dd t<\infty.
\eaqs
Since for any $n$ we have $\bs\rhoup^n\in\AC^p([0,T];((\mathcal{P}_p(\Rd))^2,\mathcal{T}_{m,\betaup,\mu^n}))$, Proposition \ref{prop:met vel} yields the existence of a unique flux $\jupbold^n$ such that $(\bs\rhoup^n,\jupbold^n)\in\CE_T$ and $\abs{(\rhoup^n_t)'}^p_{\mathcal{T}_{m,\betaup,\mu^n}}=\A_{m, \betaup}(\mu^n;\rhoup^n_t,\jup^n_t)$ for a.e. $t\in[0,T]$. We therefore obtain
\baqs
\sup_{n\in\mathbb{N}}\int_0^t \A_{m, \betaup}(\mu^n;\rhoup^n_t,\jup^n_t)\dd t=\sup_{n\in\mathbb{N}}\int_0^T \abs{(\rhoup^n_t)'}^p_{\mathcal{T}_{m,\betaup,\mu^n}} \dd t<\infty.
\eaqs
Thus, by Proposition \ref{prop:compactness}, there exists $(\bs\rhoup,\jupbold)\in\CE_T$ such that, up to subsequences, $\rhoup^n_t\rightharpoonup\rhoup_t$ and $\jup^n_t\rightharpoonup^\ast \jup_t$ as $n\to\infty$ for a.e. $t\in[0,T]$, and we have
\baqs
\int_0^t \A_{m, \betaup}(\mu;\rhoup_t,\jup_t)\dd t\le\liminf_{n\to\infty}\int_0^t \A_{m, \betaup}(\mu^n;\rhoup^n_t,\jup^n_t)\dd t<\infty.
\eaqs
Hence, Proposition \ref{prop:met vel} implies that $\bs\rhoup\in\AC^p([0,T];((\mathcal{P}_p(\Rd))^2,\mathcal{T}_{m,\betaup,\mu}))$ and $\abs{\rhoup_t'}^p_{\mathcal{T}_{m,\betaup,\mu}} \leq \A_{m, \betaup}(\mu;\rhoup_t,\jup_t)$ for a.e. $t\in[0,T]$, which now yields
\baq\label{eq:met vel lsc}
\int_0^T\abs{\rhoup_t'}^p_{\mathcal{T}_{m,\betaup,\mu}}\dd t\le\liminf_{n\to\infty}\int_0^T\abs{(\rhoup^n_t)'}^p_{\mathcal{T}_{m,\betaup,\mu^n}}\dd t.
\eaq
By Proposition \ref{prop:Continuity of the energy}, we have that the $\mathcal{E}$ is narrowly continuous, i.e.
\baq\label{eq:E conv}
\lim_{n\to\infty}\mathcal{E}(\rhoup^n_0)=\mathcal{E}(\rhoup_0)\text{ and }\lim_{n\to\infty}\mathcal{E}(\rhoup^n_T)=\mathcal{E}(\rhoup_T).
\eaq
Lastly, Fatou's lemma and the narrow lower semicontinuity of $\mathcal{D}$, shown in Lemma \ref{lem:lsc D}, give us
\baq\label{eq:D lsc}
\int_0^T \mathcal{D}(\mu;\rho_t)\dd t\le\int_0^T \liminf_{n\to\infty}\mathcal{D}(\mu^n;\rhoup^n_t)\dd t\le \liminf_{n\to\infty}\int_0^T \mathcal{D}(\mu^n;\rhoup^n_t)\dd t.
\eaq
Combining \eqref{eq:met vel lsc}, \eqref{eq:E conv} and \eqref{eq:D lsc}, we finally obtain
\baqs
\mathcal{G}_T(\mu;\bs\rhoup)=\mathcal{E}(\rhoup_T)-\mathcal{E}(\rhoup_0)+\int_0^T \frac{1}{q}\mathcal{D}(\mu;\rho_t) +\frac{1}{p}\abs{\rhoup_t'}^p_{\mathcal{T}_{m,\betaup,\mu}}\dd t \le\liminf_{n\to\infty}\mathcal{G}_T(\mu^n;\bs\rhoup^n),
\eaqs
which finishes he proof.
\end{proof}

\begin{theorem}[Closedness of the Null Space of the DeGiorgi Functional]\label{thm:Stability of gradient flows}
Let $(\mu^n)_{n\in\mathbb{N}}\subset \Mloc^+(\Rd)$ and suppose $\mu^n \rightharpoonup^\ast \mu$ for some $\mu\in\Mloc^+(\Rd)$ as $n\to\infty$. Assume that $\mu^n$ and $\mu$ safisfy \eqref{MB2}, \eqref{MB} and \eqref{BC} uniformly in $n$. For $i,k=1,2$ let $K^{(ik)}$ satisfy \eqref{K2}, \eqref{K3} and $K^{(21)}=K^{(12)}$. Let $\bs\rhoup^n$ be a gradient flow of $\mathcal{E}$ with respect to $\mu^n$ for all $n\in\mathbb{N}$, i.e.
\baqs
\mathcal{G}_T(\mu^n;\bs\rhoup^n)=0,\qquad\text{for all }n\in\mathbb{N}.
\eaqs
Additionally, assume $\sup_{n\in\mathbb{N}}M_p(\rho^{n,(i)}_0)<\infty$ for $i=1,2$ and $\rhoup^n_t\rightharpoonup \rhoup_t$ for all $t\in[0,T]$ for some $\bs\rhoup \subset (\mathcal{P}_p(\Rd))^2$ as $n\to\infty$. Then, $\bs\rhoup \in \AC^p([0,T];((\mathcal{P}_p(\Rd))^2,\mathcal{T}_{m,\betaup,\mu}))$ is a gradient flow of $\mathcal{E}$ with respect to $\mu$, i.e.
\baqs
\mathcal{G}_T(\mu;\bs\rhoup)=0.
\eaqs
\end{theorem}
\begin{proof}
By Lemma \ref{lem:Compactness and lower semicontinuity of the De Giorgi functional}, we immediately obtain that $\bs\rhoup \in \AC^p([0,T];((\mathcal{P}_p(\Rd))^2,\mathcal{T}_{m,\betaup,\mu}))$ and that, up to a subsequence, we have
\baqs
\liminf_{n\to\infty}\mathcal{G}_T(\mu^n;\bs\rhoup^n)\geq\mathcal{G}_T(\mu;\bs\rhoup).
\eaqs
Finally, since $\mathcal{G}_T(\mu;\bs\rhoup) \geq 0$, Young's inequality and Corollary \ref{cor sqrtD is one-sided strong upper gradient} yield $\mathcal{G}_T(\mu;\bs\rhoup) = 0$.
\end{proof}
\begin{theorem}[Existence of weak solutions]\label{thm:Existence of weak solutions}
Let $m$ satisfy assumption \eqref{A} from Proposition \ref{prop:Absolutely continuous curves stay supported in supp mu} and for $i,k=1,2$, let $K^{(ik)}$ satisfy \eqref{K2}, \eqref{K3} and $K^{(21)}=K^{(12)}$. Suppose that $\mu\in\Mloc^+(\Rd)$ satisfies \eqref{MB2} and \eqref{BC}. Assume further that there exists $C'_\eta>0$ such that we have
\begin{equation}
    \label{eq:(A1)-Ersatz}\tag{MB2$^\prime$}
    \sup_{(x,y)\in G\cap\supp\mu\otimes\mu}\left(\abs{x-y}^q\lor\abs{x-y}^{pq}\right)\eta(x,y)\leq C'_\eta.
\end{equation}
Let $\rhoinit\in (\mathcal{P}_p(\Rd))^2$ be $\mu$-absolutely continuous. Then, there exists a weakly continuous curve $\bs\rhoup:[0,T]\to(\mathcal{P}_p(\Rd))^2$ s.t. $\supp \rho_t\subseteq \supp\mu$ for all $t\in[0,T]$, which is a weak solution of \eqref{eq:NL2CIE} and satisfies the initial condition $\rhoup_0=\rhoinit$.
\end{theorem}
\begin{proof}
Let $(\mu^n)_{n\in\mathbb{N}} \subset\Mloc^+(\Rd)$ be a sequence of atomic measures with finitely many atoms that narrowly converges to $\mu$. This means every $\mu^n$ is of the form
\baq\label{eq_def:mu^n}
\mu^n = \sum_{l=1}^{N_n} \mu^n_l \delta_{x^n_l},
\eaq
for some $N_n\in\mathbb{N}$, $\mu^n_l \in \R\setminus\{0\}$ and $x^n_l\in\Rd$. We further assume, without loss of generality, for any $n\in\mathbb{N}$ that $\mu^n(\Rd)\leq\mu(\Rd)$ and $\supp\mu^n\subset\supp\mu$.\\
Since every $\mu^n$ consists of finitely many atoms and their limit $\mu$ satisfies \eqref{BC}, the family $(\mu^n)_{n\in\mathbb{N}}$ satisfies \eqref{BC} uniformly in $n$. Indeed, as $\mu^n\rightharpoonup\mu$, for any $\varepsilon>0$ and $N \in \mathbb{N}$, there exists $\tilde{\varepsilon}=\tilde{\varepsilon}(\varepsilon,N) > 0$ s.t. $\tilde{\varepsilon}\to 0$ when $\varepsilon\to 0$ and $N\to\infty$ and s.t. we have
\baqs
&\sup_{n\geq N}\sup_{x\in\Rd}\int_{B_\varepsilon(x)\setminus\{x\}}{\abs{x-y}}^q\eta(x,y)\dd\mu^n(y)\\
\leq&\sup_{x\in\Rd}\left(\sup_{n\geq N}\int_{B_\varepsilon(x)\setminus\{x\}}{\abs{x-y}}^q\eta(x,y){\abs{\dd\mu^n-\dd\mu}}(y)+\int_{B_\varepsilon(x)\setminus\{x\}}{\abs{x-y}}^q\eta(x,y)\dd\mu(y)\right)\\
\leq&\,\tilde{\varepsilon}+\sup_{x\in\Rd}\int_{B_\varepsilon(x)\setminus\{x\}}{\abs{x-y}}^q\eta(x,y)\dd\mu(y).
\eaqs
On the other hand, since all the $\mu^n$ consist of only finitely many atoms, for any $\varepsilon > 0$ there exists $N=N(\varepsilon)\in\mathbb{N}$ such that $N\to\infty$ when $\varepsilon\to 0$ and such that we have
\baqs
\sup_{n< N}\sup_{x\in\Rd}\int_{B_\varepsilon(x)\setminus\{x\}}\abs{x-y}^q\eta(x,y)\dd\mu^n(y)=0.
\eaqs
Thus, choosing $N(\varepsilon)$ and $\tilde\varepsilon(\varepsilon,N(\varepsilon))$ as above, letting $\varepsilon \to 0$ and using the fact that $\mu$ satisfies \eqref{BC}, we obtain
\baqs
\lim_{\varepsilon\to 0}\sup_{n\in\mathbb{N}}\sup_{x\in\Rd}\int_{B_\varepsilon(x)\setminus\{x\}}\abs{x-y}^q\eta(x,y)\dd\mu^n(y)=0.
\eaqs
Next, denote by $\tilde{\mu}^n$ the normalization of $\mu^n$, i.e.
\baqs
\tilde{\mu}^n = \frac{\mu(\Rd)}{\mu^n(\Rd)}\mu^n,
\eaqs
and let $\pi^n$ be an optimal transportation plan between $\mu$ and $\tilde{\mu}^n$ for the quadratic cost. 
Since we have $\tilde{\mu}^n\rightharpoonup\mu$ narrowly, we have $\pi^n\rightharpoonup (\id\times\id)_{\#}\mu$ narrowly. For $i=1,2$ let $\tilde{\varrho}_0^{(i)}$ be the density of $\varrho_0^{(i)}$ with respect to $\mu$ and let $\varrho_0^{n,(i)}$ be the second marginal of $(\tilde{\varrho}_0^{(i)}\times \mathbb{1})\dd\pi^n$, i.e., for any $B\in\mathcal{B}(\Rd)$, we have $\varrho_0^{n,(i)}(B)=\int_{\Rd\times B}\tilde{\varrho}_0^{(i)}(x)\dd\pi^n(x,y)$. Then, by construction, for any $n\in\mathbb{N}$ and $i=1,2$, we find $\varrho_0^{n,(i)}(\Rd)=\varrho_0^{(i)}(\Rd)$ and $\varrho_0^{n,(i)}\ll\mu^n$. Also, by the convergence of $\pi^n$ and the fact that $(\tilde{\varrho}_0^{(i)}\times \mathbb{1})\pi^n$ is a transport plan between $\varrho_0^{(i)}$ and $\varrho_0^{n,(i)}$, we find that $\varrho_0^{n,(i)}\rightharpoonup\varrho_0^{(i)}$ for $i=1,2$ as $n\to\infty$. By \eqref{eq:(A1)-Ersatz}, for all $n\in\mathbb{N}$, we obtain the bound
\baqs
\mu-\esssup_{x\in\Rd}\int_\Rd\left(\abs{x-y}^q\lor\abs{x-y}^{pq}\right)\eta(x,y)\dd\mu^n(y)\leq C'_\eta\mu^n(\Rd)\leq C'_\eta\mu(\Rd).
\eaqs
Since, by construction, $\varrho_0^{n,(i)}\ll\mu^n$, we have $\supp\varrho_0^{n,(i)} \subset\supp\mu^n \subset\supp\mu$. By Proposition \ref{prop:Absolutely continuous curves stay supported in supp mu}, the nested support is preserved in time for any $\bs\rhoup^n \in\AC^p([0,T];((\mathcal{P}_p(\Rd))^2,\mathcal{T}_{m,\betaup,\mu}))$ with $\rho^{n,(i)}_0=\varrho^{n,(i)}_0$, i.e., we have $\supp\rho_t^{n,(i)}\subset\supp\mu^n \subset\supp\mu$ for any $t\in[0,T]$ and any $n\in\mathbb{N}$. Therefore, \eqref{eq:(A1)-Ersatz} can be used to replace \eqref{MB} uniformly in $n$, when employing Lemma \ref{lem:Compactness and lower semicontinuity of the De Giorgi functional} and Theorem \ref{thm:Stability of gradient flows} later in this proof. Further note that $\{\mu^n\}_n$ satisfy \eqref{MB2} uniformly in $n$, since $\mu$ satisfies \eqref{MB2}. These considerations now allow us to construct curves $\bs\rhoup^n \in\AC^p([0,T];((\mathcal{P}_p(\Rd))^2,\mathcal{T}_{m,\betaup,\mu}))$, which are gradient flows and converge to a gradient flow $\bs\rhoup$. Indeed, since any $\mu^n$ is a counting measure, we have $\rhoup^{n,(i)}_t\ll\mu^n$ for any $i=1,2$, $t\in[0,T]$, and $n\in\mathbb{N}$. Thus, we can write
\baq\label{eq_def:rho^n,(i)}
\rho_t^{n,(i)}=\sum_{\nu=1}^{N_n} \rho_\nu^{n,(i)}(t)\mu^n_\nu \delta_{x_\nu^n},
\eaq
for suitable functions $\rho_\nu^{n,(i)}:[0,T]\to\R$ and the points $x_\nu^n\in\Rd$ from \eqref{eq_def:mu^n}.
Now, let $\varphi^n_\nu\in C^\infty_c(\Rd)$, $\nu\in\{1,...,N_n\}$ satisfy $\varphi^n_\nu(x_\nu^n) \neq 0$ and $\varphi^n_\nu(x_\kappa^n) = 0$ for $x_\kappa\neq \nu$. Then, inserting \eqref{eq_def:mu^n} and \eqref{eq_def:rho^n,(i)} into Equation \eqref{eq_def:flux_weak_solution}, we find for $i=1,2$ and $\nu\in\{1,...,N_n\}$
\baq\label{eq:discrete NL2CIE}
    \partial_t \rho_l^{n,(i)} = &\sum_{m=1}^{N_n}\Bigg[\Bigg(\beta^{(i)} m\left(\rho_l^{n,(i)},\rho_m^{n,(i)}\right)\Bigg(\sum_{k=1}^2\sum_{h=1}^{N_n}\left(K^{(ik)}(x^n_m,x^n_h)-K^{(ik)}(x^n_l,x^n_h)\right)\rho_h^{n,(k)}\mu^n_h\Bigg)_+\Bigg)^{q-1}\\
    &-\;\Bigg(\beta^{(i)} m\left(\rho_m^{n,(i)},\rho_l^{n,(i)}\right)\Bigg(\sum_{k=1}^2\sum_{h=1}^{N_n}\left(K^{(ik)}(x^n_m,x^n_h)-K^{(ik)}(x^n_l,x^n_h)\right)\rho_h^{n,(k)}\mu^n_h\Bigg)_-\Bigg)^{q-1}\Bigg]\eta(x_l^n,x_m^n)\mu^n_m.
\eaq
Since $m(0,s)=0$, we see that the simplex defined by
\baq\label{eq:vertex}
\rho_\nu^{n,(i)}\in\bigg[0,\Big(\min_{1\leq m\leq N_n}\mu^n_\kappa\Big)^{-1}\bigg],\quad \sum_{\nu=1}^{N_n}\mu^n_\nu\rho_\nu^{n,(i)} =1.
\eaq
is an invariant region of the dynamics. Due to the continuity of $m$, the right-hand side of \eqref{eq:discrete NL2CIE} is continuous with respect to $\rho^{n,(i)}$ for any $n\in\mathbb{N}$ and $i=1,2$. With this, the Peano existence theorem provides us with a strong solution $\bs\rhoup^n$ of \eqref{eq:discrete NL2CIE} on an interval $[0,\tau_n]$ for some $\tau_n > 0$. Due to \eqref{eq:vertex}, $\tau_n$ only depends on $n$ and $\mu^n$. Thus, by a standard continuation argument, a piecewise $C^1$ solution exists on the whole interval $[0,T]$.
By construction, this solution is a weak solution for \eqref{eq:NL2CIE} in the sense of Definition \ref{def:weak solution to NL2CIE} with respect to $\mu^n$ starting from $\varrhoup^n_0$. Therefore, by Theorem \ref{thm:characterization of weak solutions to NL2CIE}, $\bs\rhoup^n$ is a gradient flow of $\mathcal{E}$ with respect to $\mu^n$ with initial datum $\varrhoup^n_0$, for any $n\in\mathbb{N}$. This allows us to apply the compactness from Lemma \ref{lem:Compactness and lower semicontinuity of the De Giorgi functional} and the stability from Theorem \ref{thm:Stability of gradient flows} to find that, up to a subsequence, $\rhoup^n_t\rightharpoonup\rhoup_t$ as $n\to\infty$ for all $t\in[0,T]$, where $\bs\rhoup \in\AC^p([0,T];((\mathcal{P}_p(\Rd))^2,\mathcal{T}_{m,\betaup,\mu}))$ is a gradient flow of $\mathcal{E}$ with respect to $\mu$ starting from $\rhoinit$.
\end{proof}
\begin{remark}
Assumption \eqref{eq:(A1)-Ersatz} is needed to obtain an atomic approximating sequence $(\mu^n)_n$ for $\mu$, which satisfies \eqref{MB2}, \eqref{MB} and \eqref{BC} uniformly in $n$. There might be cases, where it is possible drop this assumption if one is able to explicitly construct a sequence $(\mu^n)_n$ satisfying these bounds uniformly in $n$.
\end{remark}
\appendix
\section*{Appendix}
\section{Chain rule}\label{A:Chain}
\begin{remark}\label{rem:Approximate energies}(Approximate energies).
Let $K^{(ik)}$, $i,k=1,2$ satisfy \eqref{K2}, \eqref{K3} and $K^{(21)} = K^{(12)}$. Let $m\in C^\infty_c(\Rd\times\Rd)$ be a standard mollifier. For $\varepsilon >0$ and $z\in\R^{2d}$, we set $m_\varepsilon(z)\coloneqq \frac{1}{\varepsilon^{2d}}m(z/\varepsilon)$. Moreover, for $R>0$ let $\varphi_R\in C^\infty_c(\R^{2d})$ be a cut-off function with $\supp\varphi_R\subset B_{2R}(0)$, $\varphi_R|_{B_R(0)} \equiv 1$ and $\abs{\nabla\varphi_R}\leq \frac{2}{R}$. Using the mollifier and the cut-off we define for $i,k=1,2$
\baqs
K^{\varepsilon,(ik)}_R(x,y) \coloneqq \varphi_R(x,y)(K^{(ik)}\ast m_\varepsilon)(x,y).
\eaqs
Note that the functions $K^{\varepsilon,(ik)}_R$ still satisfy \eqref{K2}, \eqref{K3} and, additionally, lie in $C^\infty_c(\R^{2d})$. For $\rhoup \in (\mathcal{P}_p(\Rd))^2$, we define the approximate energies
\baqs
\mathcal{E}^\varepsilon_R(\rhoup)& \coloneqq \frac{1}{2}\sum_{i,k=1}^2\iint_{\Rd\times\Rd}K^{\varepsilon,(ik)}_R(x,y)\dd\rho^{(i)}(x)\dd\rho^{(k)}(y).
\eaqs
\end{remark}
\begin{lemma}[Mollification]\label{lem:Mollified measures}
Let $(\bs\rhoup,\jupbold) \in\CE_T$ with $\int_0^T\A_{m,\betaup}(\mu;\rhoup_t,\jup_t)\dd t <\infty$. Let $n \in C^\infty_c(\R)$ be a standard mollifier with $\supp n \subseteq[-1,1]$. For $\bar{\varepsilon} >0 $ set $n_{\bar{\varepsilon}}(t) \coloneqq \frac{1}{\bar{\varepsilon}}n\left(\frac{t}{\bar{\varepsilon}}\right)$. Extending $\bs\rhoup$ and $\jupbold$ periodically to $[-T,2T]$, i.e., $\rhoup_{-t}=\rhoup_{T-t}$ and $\rhoup_{T+t}=\rhoup_t$ for any $t\in(0,T]$ and likewise for $\jup$, we define the regularizations $\rhoup^{\bar{\varepsilon}}_t = (\rho_t^{\bar{\varepsilon},(1)},\rho_t^{\bar{\varepsilon},(2)})^\top$ and $\jup^{\bar{\varepsilon}}_t = (j_t^{\bar{\varepsilon},(1)},j_t^{\bar{\varepsilon},(2)})^\top$ by
\baqs
\rho_t^{\bar{\varepsilon},(i)}(A) &\coloneqq (n_{\bar{\varepsilon}}\ast\rho_t^{(i)})(A) = \int_{-\bar{\varepsilon}}^{\bar{\varepsilon}} n_{\bar{\varepsilon}}(s)\rho_{t-s}(A)\dd s,\qquad\forall A\subseteq\Rd,\\
j_t^{\bar{\varepsilon},(i)}(B) &\coloneqq (n_{\bar{\varepsilon}}\ast j_t^{(i)})(B) = \int_{-\bar{\varepsilon}}^{\bar{\varepsilon}} n_{\bar{\varepsilon}}(s)j_{t-s}(B)\dd s,\qquad\forall B\subseteq G,
\eaqs
for $i=1,2$ and any $\bar{\varepsilon}\in(0,T)$. Then, we obtain that the integral $\int_0^T\A_{m,\betaup}(\mu;\rhoup^{\bar{\varepsilon}}_t, \jup^{\bar{\varepsilon}}_t)\dd t$ is uniformly bounded with respect to $\bar{\varepsilon}$ and that the pair $(\bs\rhoup^{\bar{\varepsilon}},\jupbold^{\bar{\varepsilon}})$ lies in $\CE_T$. Furthermore, if $\rhoup_t \in (\mathcal{P}_p(\Rd))^2$, then $(\rhoup^{\bar{\varepsilon}}_t)_{\bar{\varepsilon}}\subset (\mathcal{P}_p(\Rd))^2$ with uniformly bounded $p$-th moments. 
\end{lemma}
\begin{proof}
If $\rhoup_t \in (\mathcal{P}_p(\Rd))^2$ for all $t\in[0,T]$, it is immediate that $\rhoup^{\bar{\varepsilon}}_t \in (\mathcal{P}_p(\Rd))^2$ for all $t\in[0,T]$. Indeed, for $i\in\{1,2\}$ and $f:\Rd\rightarrow\R$ integrable with respect to $\rho_t$ for $t\in[0,T]$, Fubini's theorem gives us
\baqs
\int_\Rd f(x)\dd\rho_t^{\bar{\varepsilon},(i)}(x) = \int_{-\bar{\varepsilon}}^{\bar{\varepsilon}} n_{\bar{\varepsilon}}(s)\int_\Rd f(x)\dd\rho_{t-s}^{(i)}(x)\dd s.
\eaqs
In particular, for $f \equiv 1$ the inner integral is equal to 1 and hence the whole expression, while for $f(x)=|x|^p$ the inner integral and consequently the whole expression are both finite. In particular, the family $(\rhoup^{\bar{\varepsilon}}_t)_{\bar{\varepsilon}}$ has uniformly bounded $p$-th moments.

Next, we prove the uniform bound. To shorten notation, we only consider the case $R\land S<\infty$. The recession terms in the case $R=S=\infty$ can then be treated analogously by employing $\sigma_t^{(i)}=\varsigma_t^{(i)}=\rho_t^{(i)\perp}\otimes\mu+\mu\otimes\rho_t^{(i)\perp}$ for $i=1,2$. By the joint convexity of the density function $\alpha$, Jensen's inequality and Fubini's theorem, we obtain
\baqs
    \int_0^T\A_{m,\betaup}(\mu;\rhoup^{\bar{\varepsilon}}_t,\jup^{\bar{\varepsilon}}_t)\dd t &=\frac{1}{2}\sum_{i=1}^2\frac{1}{\beta^{(i)}}\int_0^T\iint_G\Bigg[\alpha\Bigg(\int_{-\bar{\varepsilon}}^{\bar{\varepsilon}}\frac{\dd j^{(i)}_{t-s}}{\dd(\mu\otimes\mu)}n_{\bar{\varepsilon}}(s)\dd s,\int_{-\bar{\varepsilon}}^{\bar{\varepsilon}}\frac{\dd \gamma_{1,t-s}^{(i)}}{\dd(\mu\otimes\mu)}n_{\bar{\varepsilon}}(s)\dd s\Bigg)\\
    &\hphantom{=\frac{1}{2}\sum_{i=1}^2\frac{1}{\beta^{(i)}}}+\alpha\Bigg(-\int_{-\bar{\varepsilon}}^{\bar{\varepsilon}}\frac{\dd j^{(i)}_{t-s}}{\dd(\mu\otimes\mu)}n_{\bar{\varepsilon}}(s)\dd s,\int_{-\bar{\varepsilon}}^{\bar{\varepsilon}}\frac{\dd \gamma_{2,t-s}^{(i)}}{\dd(\mu\otimes\mu)}n_{\bar{\varepsilon}}(s)\dd s\Bigg)\Bigg]\eta\dd(\mu\otimes\mu)\dd t\\
    &\le\frac{1}{2}\sum_{i=1}^2\frac{1}{\beta^{(i)}}\int_{-\bar{\varepsilon}}^{\bar{\varepsilon}}\int_0^T\iint_G\Bigg[\alpha\Bigg(\frac{\dd j^{(i)}_{t-s}}{\dd(\mu\otimes\mu)}n_{\bar{\varepsilon}}(s),\frac{\dd \gamma_{1,t-s}^{(i)}}{\dd(\mu\otimes\mu)}n_{\bar{\varepsilon}}(s)\Bigg)\\
    &\hphantom{\le\frac{1}{2}\sum_{i=1}^2\frac{1}{\beta^{(i)}}}+\alpha\Bigg(-\frac{\dd j^{(i)}_{t-s}}{\dd(\mu\otimes\mu)}n_{\bar{\varepsilon}}(s),\frac{\dd \gamma_{2,t-s}^{(i)}}{\dd(\mu\otimes\mu)}n_{\bar{\varepsilon}}(s)\Bigg)\Bigg]n_{\bar{\varepsilon}}(s)\dd s\eta\dd \mu\otimes\mu\\
    &=\int_{-\bar{\varepsilon}}^{\bar{\varepsilon}}\int_0^T\A_{m,\betaup}(\mu;\rhoup_{t-s},\jup_{t-s})\dd t\dd s
    \leq \int_{-T}^{2T}\A_{m,\betaup}(\mu;\rhoup_t,\jup_t)\dd t\\
    &= 3 \int_0^T\A_{m,\betaup}(\mu;\rhoup_t,\jup_t)\dd t <\infty.
\eaqs
Let us now check that $(\bs\rhoup^{\bar{\varepsilon}},\jupbold^{\bar{\varepsilon}})\in \CE_T$. The first two requirements are immediate. Hence, it only remains to check that the continuity equation \eqref{eq:cont} holds. To this end, let $i\in \{1,2\}$, $\varphi \in C^\infty_c([0,T]\times\Rd)$, and periodically extend $\varphi$ to $[-T,2T]$. Then, for $i=1,2$ and denoting by $\bar{\bs\rho}^{(i)}$ the continuous representative of $\bs\rho^{(i)}$ defined in Lemma \ref{lem:continuous representative}, we obtain
\baqs
&\int_0^T\int_\Rd \partial_t\varphi_t(x)\dd\rho_t^{\bar{\varepsilon},(i)}(x) \dd t+\frac{1}{2}\int_0^T\iint_G(\babla\varphi_t)(x,y)\eta(x,y)\dd j_t^{\bar{\varepsilon},(i)}(x,y)\dd t\\
=&\int_0^T\int_\Rd \partial_t\varphi_t(x)\int_{-\bar{\varepsilon}}^{\bar{\varepsilon}} n_{\bar{\varepsilon}}(s)\dd\rho^{(i)}_{t-s}(x)\dd s \dd t+\frac{1}{2}\int_0^T\iint_G(\babla\varphi_t)(x,y)\eta(x,y)\int_{-\bar{\varepsilon}}^{\bar{\varepsilon}} n_{\bar{\varepsilon}}(s)\dd j^{(i)}_{t-s}(x,y)\dd s\dd t\\
=&\int_{-\bar{\varepsilon}}^{\bar{\varepsilon}} n_{\bar{\varepsilon}}(s)\Bigg[\int_0^T\int_\Rd \partial_t\varphi_t(x)\dd\rho^{(i)}_{t-s}(x)\dd t+\frac{1}{2}\int_0^T\iint_G(\babla\varphi_t)(x,y)\eta(x,y) \dd j^{(i)}_{t-s}(x,y)\dd t\Bigg]\dd s\\
=&\int_{-\bar{\varepsilon}}^{\bar{\varepsilon}} n_{\bar{\varepsilon}}(s)\Bigg[\int_{-s}^{T-s}\int_\Rd \partial_t\varphi_{t+s}(x)\dd\rho^{(i)}_{t}(x)\dd t+\frac{1}{2}\int_{-s}^{T-s}\iint_G(\babla\varphi_{t+s})(x,y)\eta(x,y) \dd j^{(i)}_{t}(x,y)\dd t\Bigg]\dd s\\
=&\int_{-\bar{\varepsilon}}^{\bar{\varepsilon}} n_{\bar{\varepsilon}}(s)\Bigg[\int_\Rd \varphi_T(x)\dd\bar\rho^{(i)}_{T-s}(x)-\int_\Rd \varphi_0(x)\dd\bar\rho^{(i)}_{-s}(x)\Bigg]\dd s=\int_\Rd \varphi_T(x)\dd\bar\rho_T^{\bar{\varepsilon},(i)}(x)-\int_\Rd \varphi_0(x)\dd\bar\rho_0^{\bar{\varepsilon},(i)}(x).
\eaqs
\end{proof}
\begin{remark}[Chain rule in the mollified case]
For $(\bs\rhoup^{\bar{\varepsilon}},\jupbold^{\bar{\varepsilon}}) \in \CE_T$ as in Lemma \ref{lem:Mollified measures}, by the regularity of $\bs\rhoup^{\bar{\varepsilon}}$ as well as the continuity equation \eqref{eq:cont} in conjunction with Remark \ref{rem:eq cont}, we have
\baqs
\frac{\dd}{\dd t}\mathcal{E}_R^\varepsilon(\rhoup^{\bar{\varepsilon}}_t) &= 
\sum_{i,k=1}^2\int_\Rd(K^{\varepsilon,(ik)}_R\ast\rho_t^{\bar{\varepsilon},(i)})(x)\partial_t\rho_t^{\bar{\varepsilon},(k)}(x)\dd\mu(x)\\
&= \frac{1}{2}\sum_{i,k=1}^2\iint_G\babla(K^{\varepsilon,(ik)}_R\ast\rho_t^{\bar{\varepsilon},(i)})(x,y)\eta(x,y)\dd j_t^{\bar{\varepsilon},(k)}(x,y)\\
&= \frac{1}{2}\sum_{i=1}^2\iint_G \babla\delta_{\rho^{(i)}}\mathcal{E}_R^\varepsilon(\rho_t^{\bar{\varepsilon}})(x,y)\eta(x,y)\dd j_t^{\bar{\varepsilon},(i)}(x,y).
\eaqs
\end{remark}
Now, we are in the position to prove Proposition \ref{prop:Chain rule}.
\begin{proof}[Proof of Proposition \ref{prop:Chain rule}]
Recall that choosing $\bs\rhoup$, $\jupbold$ and $\bs v$ as above, we have $(\bs\rhoup,\jupbold)\in\CE_T$, $\abs{\rhoup'_t}^p = \A_{m,\betaup}(\mu;\rhoup_t,\jup_t)$ for a.e. $t\in[0,T]$, $\int_0^T \A_{m,\betaup}(\mu;\rhoup_t,\jup_t) \dd t <\infty$ and $\dd j_t^{(i)\mu} = ((v_t^{(i)})_+)^{q-1}\dd\gamma^{(i)}_{1,t}-((v_t^{(i)})_-)^{q-1}\dd\gamma^{(i)}_{2,t}$, $(\mu\otimes\mu)$-a.e. in $G$ as well as $\dd j_t^{(i)\perp} = ((v_t^{(i)\perp})_+)^{q-1}\dd\gamma^{(i)\perp}_{1,t}-((v_t^{(i)\perp})_-)^{q-1}\dd\gamma^{(i)\perp}_{2,t}$, $\varsigma_t^{(i)}$-a.e. in $G$, for a.e. $t\in[0,T]$ and $i=1,2$. Inserting the definition of $\tilde{l}$ as well as the connection between $\vupbold$ and $\jupbold$, we can rewrite \eqref{eq:chain rule velocity} as
\baq\label{eq:chain rule flux}
    \mathcal{E}(\rhoup_t)-\mathcal{E}(\rhoup_s)&=\int_{s}^{t}\tilde{l}_\rho(v_\tau)[\betaup\babla\delta_{\rhoup}\mathcal{E}(\rhoup_\tau)]\\ 
    &=\frac{1}{2}\sum_{i=1}^2\frac{1}{\beta^{(i)}}\iint_G \beta^{(i)}\babla\delta_{\rho^{(i)}}\mathcal{E}(\rhoup_\tau)\eta\Big[((v^{(i)}_\tau)_+)^{q-1}\dd\gamma_{1,\tau}^{(i)}-((v^{(i)}_\tau)_-)^{q-1}\dd\gamma_{2,\tau}^{(i)}\\
    &\hphantom{=\frac{1}{2}\sum_{i=1}^2\frac{1}{\beta^{(i)}}\iint_G}+((v^{(i)\perp}_\tau)_+)^{q-1}\dd\gamma_{1,\tau}^{(i)\perp}-((v^{(i)\perp}_\tau)_-)^{q-1}\dd\gamma_{2,\tau}^{(i)\perp}\Big]\\
    &=\frac{1}{2}\sum_{i=1}^2\int_{s}^{t}\iint_G \babla\delta_{\rho^{(i)}}\mathcal{E}(\rhoup_\tau)(x,y)\eta(x,y)\dd j_\tau^{(i)}(x,y)\dd\tau.
\eaq
We regularize $(\bs\rhoup,\jupbold)$ to obtain $(\bs\rhoup^{\bar{\varepsilon}}, \jupbold^{\bar{\varepsilon}})$ as in Lemma \ref{lem:Mollified measures}, approximate the energies as in Remark \ref{rem:Approximate energies} and integrate in time to obtain
\baq\label{eq:chain rule flux regularized}
\mathcal{E}_R^\varepsilon(\rhoup^{\bar{\varepsilon}}_t)-\mathcal{E}_R^\varepsilon(\rhoup^{\bar{\varepsilon}}_s) =\frac{1}{2}\sum_{i=1}^2\int_{s}^{t}\iint_G \babla\delta_{\rho^{(i)}}\mathcal{E}_R^\varepsilon(\rho_\tau^{\bar{\varepsilon}})(x,y)\eta(x,y)\dd j_\tau^{\bar{\varepsilon},(i)}(x,y)\dd\tau.
\eaq
Our goal is to pass to the limit as $\varepsilon\to 0$, $\bar{\varepsilon}\to 0$ and $R\to\infty$, which will yield \eqref{eq:chain rule flux}. Due to Proposition \ref{prop:Continuity of the energy}, we immediately obtain $\mathcal{E}_R^\varepsilon(\rhoup^{\bar{\varepsilon}}_t) \to \mathcal{E}_R^\varepsilon(\rho_t)$ as $\bar{\varepsilon}\to 0$. By definition we also have that $K_R^{\varepsilon,(ik)}\to  \varphi_RK^{(ik)}\eqqcolon K_R^{(ik)}$, uniformly as $\varepsilon\to 0$. Then, letting $R\to\infty$, we obtain convergence of the left-hand side of \eqref{eq:chain rule flux regularized} to the left-hand side of \eqref{eq:chain rule flux}.

It remains to show convergence of the right-hand side. To this end, we use a truncation argument. Let $\tilde{\varepsilon} > 0$ and set $N_{\tilde{\varepsilon}}\coloneqq\overline{B}_{{\tilde{\varepsilon}}^{-1}}\times\overline{B}_{{\tilde{\varepsilon}}^{-1}}$, where $B_{{\tilde{\varepsilon}}^{-1}} = \{x\in\Rd:|x|<{\tilde{\varepsilon}}^{-1}\}$, and set $G_{\tilde{\varepsilon}} \coloneqq \{(x,y)\in G:{\tilde{\varepsilon}}\leq \abs{x-y}\}$. Let $(\varphi_{\tilde{\varepsilon}})_{{\tilde{\varepsilon}}>0} \subset C_c^\infty(\Rd\times G;[0,1])$ be a family of truncation functions, which is s.t.,for any ${\tilde{\varepsilon}} >0$, we have $\{\varphi_{\tilde{\varepsilon}}=1\}\supset\overline{B}_{{\tilde{\varepsilon}}^{-1}}\times G_{\tilde{\varepsilon}}\cap N_{\tilde{\varepsilon}}$. We add and subtract $\varphi_{\tilde{\varepsilon}}$ on the right-hand side of \eqref{eq:chain rule flux regularized}. Since $\rhoup^{\bar{\varepsilon}}_t\otimes \jup^{\bar{\varepsilon}}_t\rightharpoonup\rhoup_t\otimes \jup_t$ for any $T\in[0,T]$ as $\bar{\varepsilon}\to 0$, and $K_R^{\varepsilon,(ik)}\to K_R^{(ik)}$ uniformly as $\varepsilon\to 0$, we can pass to the limit in $\bar{\varepsilon}$ and $\varepsilon$ for any $R,{\tilde{\varepsilon}}>0$:
\baqs
\lim_{\substack{\bar{\varepsilon}\to 0\\ \varepsilon\to 0}}\,&\frac{1}{2}\int_{s}^{t}\iint_G\int_\Rd\varphi_{\tilde{\varepsilon}}(z,x,y)\left(K^{\varepsilon,(ik)}_R(y,z)-K^{\varepsilon,(ik)}_R(x,z)\right)\eta(x,y)\dd\rho_\tau^{\bar{\varepsilon},(i)}(z)\dd j_\tau^{\bar{\varepsilon},(k)}(x,y)\dd\tau\\
=\,&\frac{1}{2}\int_{s}^{t}\iint_G\int_\Rd\varphi_{\tilde{\varepsilon}}(z,x,y)\left(K^{(ik)}_R(y,z)-K^{(ik)}_R(x,z)\right)\eta(x,y)\dd\rho_\tau^{(i)}(z)\dd j_\tau^{(k)}(x,y)\dd\tau,
\eaqs
for $i,k=1,2$. By using  $\varphi_{\tilde{\varepsilon}}\leq 1$, \eqref{K3} and Corollary \ref{cor:bound_by_A} in conjunction with \eqref{MB2} and \eqref{MB}, for any $\tau\in[s,t]$, we obtain the bound
\baqs
&\,\Bigg|\frac{1}{2}\iint_G\int_\Rd\varphi_{\tilde{\varepsilon}}(z,x,y)\left(K^{(ik)}_R(y,z)-K^{(ik)}_R(x,z)\right)\eta(x,y)\dd\rho_\tau^{(i)}(z)\dd j_\tau^{(k)}(x,y)\Bigg|\\
\leq&\,\frac{1}{2}\iint_G\int_\Rd\frac{\Big|K^{(ik)}_R(y,z)-K^{(ik)}_R(x,z)\Big|}{\abs{x-y}\lor\abs{x-y}^p} |x-y|\lor|x-y|^p\eta(x,y)\dd\rho_\tau^{(i)}(z)\dd j_\tau^{(k)}(x,y)\\
\leq&\,L_K M C_\eta^{1/q}\A_{m, \betaup}^{1/p}(\mu;\rhoup_\tau,\jup_\tau).
\eaqs
Hence, the integral is bounded in time uniformly with respect to ${\tilde{\varepsilon}}$ and $R$. Since, by definition of $K^{(ik)}_R$, we also have uniform boundedness in space, this allows us to apply Lebesgue's dominated convergence theorem to pass to the limit in ${\tilde{\varepsilon}}$ and $R$ to obtain:
\baqs
\lim_{\substack{{\tilde{\varepsilon}}\to 0\\ R\to \infty}}\,&\frac{1}{2}\int_{s}^{t}\iint_G\int_\Rd\varphi_{\tilde{\varepsilon}}(z,x,y)\left(K^{(ik)}_R(y,z)-K^{(ik)}_R(x,z)\right)\eta(x,y)\dd\rho_\tau^{(i)}(z)\dd j_\tau^{(k)}(x,y)\dd\tau\\
=\,&\frac{1}{2}\int_{s}^{t}\iint_G\int_\Rd\left(K^{(ik)}(y,z)-K^{(ik)}(x,z)\right)\eta(x,y)\dd\rho_\tau^{(i)}(z)\dd j_\tau^{(k)}(x,y)\dd\tau.
\eaqs
The remaining step of the proof is to control the integral involving $1-\varphi_{\tilde{\varepsilon}}(z,x,y)$. To do this, note that for ${\tilde{\varepsilon}}>0$, we have
\baqs
(\Rd\times G)\setminus\{\varphi_{\tilde{\varepsilon}}=1\}\subseteq(\overline{B}^c_{{\tilde{\varepsilon}}^{-1}}\times G)\cup(G\setminus(G_{\tilde{\varepsilon}}\cap N_{\tilde{\varepsilon}})))\eqqcolon M_{\tilde{\varepsilon}}
\eaqs
For $i,k=1,2$, as before, using \eqref{K3} and splitting the contributions, we obtain
\baqs
&{\abs{\iint_G\int_\Rd(1-\varphi_{\tilde{\varepsilon}}(z,x,y))\left(K^{\varepsilon,(ik)}_R(y,z)-K^{\varepsilon,(ik)}_R(x,z)\right)\eta(x,y)\dd\rho_t^{\bar{\varepsilon},(i)}(z)\dd j_t^{\bar{\varepsilon},(k)}(x,y)}}\\
\leq&\,L_K\iiint_{M_{\tilde{\varepsilon}}}|x-y|\lor|x-y|^p2\eta(x,y)\dd\rho_t^{\bar{\varepsilon},(i)}(z)\dd\big|j_t^{\bar{\varepsilon},(k)}\big|(x,y)\\
\leq&\,L_K\int_{\overline{B}^c_{\tilde{\varepsilon}}}\dd\rho_t^{\bar{\varepsilon},(i)}(z)\iint_{G}|x-y|\lor|x-y|^p\eta(x,y)\dd\big|j_t^{\bar{\varepsilon},(k)}\big|(x,y)\\
+& \,L_K\int_\Rd\dd\rho_t^{\bar{\varepsilon},(i)}(z)\iint_{G^c_\delta}|x-y|\lor|x-y|^p\eta(x,y)\dd\big|j_t^{\bar{\varepsilon},(k)}\big|(x,y)\\
+& \,L_K\int_\Rd\dd\rho_t^{\bar{\varepsilon},(i)}(z)\iint_{N^c_{\tilde{\varepsilon}}}|x-y|\lor|x-y|^p\eta(x,y)\dd\big|j_t^{\bar{\varepsilon},(k)}\big|(x,y).
\eaqs

Thus, we rid ourselves of the dependence on $R$. In the first term we apply Corollary \ref{cor:bound_by_A} together with \eqref{MB2} and \eqref{MB} to obtain
\baqs
L_K\int_{\overline{B}^c_{\tilde{\varepsilon}}}\dd\rho_t^{\bar{\varepsilon},(i)}(z)\iint_{G}|x-y|\lor|x-y|^p\eta(x,y)\dd\big|j_t^{\bar{\varepsilon},(k)}\big|(x,y) \leq L_K M C_\eta^{1/q}\A_{m, \betaup}^{1/p}(\mu;\rhoup^{\bar{\varepsilon}}_t,\jup^{\bar{\varepsilon}}_t)\rho_t^{\bar{\varepsilon},(i)}(\overline{B}^c_{\tilde{\varepsilon}}),
\eaqs
hence this term vanishes as ${\tilde{\varepsilon}}\to 0$. To show that the second term also vanishes as ${\tilde{\varepsilon}}\to 0$, we assume, without loss of generality, that ${\tilde{\varepsilon}} \leq 1$, which implies that $|x-y|\lor|x-y|^p =|x-y|$ on $G^c_{\tilde{\varepsilon}}$. Applying Lemma \ref{lem:bound_by_A} with $\Phi(x,y)=|x-y|\mathbb{1}_{G^c_{\tilde{\varepsilon}}}(x,y)$ yields
\baqs
&\,L_K\int_\Rd\dd\rho_t^{\bar{\varepsilon},(i)}(z)\iint_{G^c_{\tilde{\varepsilon}}}|x-y|\lor|x-y|^p\eta(x,y)\dd\big|j_t^{\bar{\varepsilon},(k)}\big|(x,y)\\
\leq &\,L_K M \A_{m, \betaup}^{1/p}(\mu;\rhoup^{\bar{\varepsilon}}_t,\jup^{\bar{\varepsilon}}_t)\sum_{l=1}^2\left(\iint_{G^c_{\tilde{\varepsilon}}}|x-y|^q\eta(x,y) \dd(\mu\otimes\mu+\rho^{(l)}\otimes\mu+\mu\otimes\rho^{(l)})(x,y)\right)^{1/q}.
\eaqs
By the local blow-up control Assumption \eqref{BC}, this also vanishes as ${\tilde{\varepsilon}}\to 0$. Similarly, for the third term, we apply Lemma \ref{lem:bound_by_A} with $\Phi(x,y)=|x-y|\mathbb{1}_{N^c_{\tilde{\varepsilon}}}(x,y)$ and obtain
\baqs
&\,L_K\int_\Rd\dd\rho_t^{\bar{\varepsilon},(i)}(z)\iint_{N^c_{\tilde{\varepsilon}}}|x-y|\lor|x-y|^p\eta(x,y)\dd\big|j_t^{\bar{\varepsilon},(k)}\big|(x,y)\\
\leq &\,L_K M C_\eta \A_{m, \betaup}^{1/p}(\mu;\rhoup^{\bar{\varepsilon}}_t,\jup^{\bar{\varepsilon}}_t)\Bigg(\mu\left(\overline{B}^c_{{\tilde{\varepsilon}}^{-1}}\right)+\sum_{l=1}^2\rho_t^{\bar{\varepsilon},(l)}\left(\overline{B}^c_{{\tilde{\varepsilon}}^{-1}}\right)\Bigg).
\eaqs
The uniform $p$-th moment bound of the family $(\rho_t^{\bar{\varepsilon},(l)})_{\tilde{\varepsilon}}$ implies tightness by the de la Vallee-Poussin theorem and the single measure $\mu\in\Mloc^+(\Rd)$ is tight as well. Thus, the third term also vanishes as ${\tilde{\varepsilon}}\to 0$, which concludes the proof.
\end{proof}
\bibliographystyle{plain}
\bibliography{bibliography}

\begin{thebibliography}{10}

\bibitem{Agueh2012_finsler}
Martial Agueh.
\newblock Finsler structure in the p-wasserstein space and gradient flows.
\newblock {\em Comptes Rendus Mathematique}, 350(1):35--40, 2012.

\bibitem{Ambrosio2000FunctionsOB}
Luigi Ambrosio, Nicola Fusco, and Diego Pallara.
\newblock {\em Functions of bounded variation and free discontinuity problems}.
\newblock Oxford Mathematical Monographs. Oxford University Press, London,
  England, March 2000.

\bibitem{AmbrosioGigliSavare2008}
Luigi Ambrosio, Nicola Gigli, and Giuseppe Savar\'{e}.
\newblock {\em Gradient flows in metric spaces and in the space of probability
  measures}.
\newblock Lectures in Mathematics ETH Z\"{u}rich. Birkh\"{a}user Verlag, Basel,
  second edition, 2008.

\bibitem{armstrong2006continuum}
Nicola~J Armstrong, Kevin~J Painter, and Jonathan~A Sherratt.
\newblock A continuum approach to modelling cell--cell adhesion.
\newblock {\em Journal of theoretical biology}, 243(1):98--113, 2006.

\bibitem{bailo2020convergence}
Rafael Bailo, Jos{\'e}~A Carrillo, Hideki Murakawa, and Markus Schmidtchen.
\newblock Convergence of a fully discrete and energy-dissipating finite-volume
  scheme for aggregation-diffusion equations.
\newblock {\em Mathematical Models and Methods in Applied Sciences}, pages
  1--36, 2020.

\bibitem{BCH2021}
Rafael Bailo, José~A. Carrillo, and Jingwei Hu.
\newblock Bound-preserving finite-volume schemes for systems of continuity
  equations with saturation, 2021.

\bibitem{B.C.H2020}
Rafael Bailo, José~Antonio Carrillo, and Jingwei Hu.
\newblock Fully discrete positivity-preserving and energy-dissipating schemes
  for aggregation-diffusion equations with a gradient flow structure.
\newblock {\em Commun. Math. Sci.}, 18, 2020.

\bibitem{barre2019modelling}
Julien Barre, Pierre Degond, Diane Peurichard, and Ewelina Zatorska.
\newblock Modelling pattern formation through differential repulsion.
\newblock {\em arXiv preprint arXiv:1906.00704}, 2019.

\bibitem{BenamouBrenier2000}
Jean-David Benamou and Yann Brenier.
\newblock A computational fluid mechanics solution to the {Monge}-{Kantorovich}
  mass transfer problem.
\newblock {\em Numerische Mathematik}, 84(3):375--393, Jan 2000.

\bibitem{berendsen2017cross}
Judith Berendsen, Martin Burger, and Jan-Frederik Pietschmann.
\newblock On a cross-diffusion model for multiple species with nonlocal
  interaction and size exclusion.
\newblock {\em Nonlinear Analysis}, 159:10--39, 2017.

\bibitem{bessemoulin2012finite}
Marianne Bessemoulin-Chatard and Francis Filbet.
\newblock A finite volume scheme for nonlinear degenerate parabolic equations.
\newblock {\em SIAM Journal on Scientific Computing}, 34(5):B559--B583, 2012.

\bibitem{bonaschi2015equivalence}
Giovanni~A Bonaschi, Jos{\'e}~A Carrillo, Marco Di~Francesco, and Mark~A
  Peletier.
\newblock Equivalence of gradient flows and entropy solutions for singular
  nonlocal interaction equations in 1d.
\newblock {\em ESAIM: Control, Optimisation and Calculus of Variations},
  21(2):414--441, 2015.

\bibitem{burger2020segregation}
Martin Burger, Jos{\'e}~A Carrillo, Jan-Frederik Pietschmann, and Markus
  Schmidtchen.
\newblock Segregation effects and gap formation in cross-diffusion models.
\newblock {\em Interfaces and Free Boundaries}, 22(2):175--203, 2020.

\bibitem{CARRILLO20101273}
J.A. Carrillo, S.~Lisini, G.~Savaré, and D.~Slepčev.
\newblock Nonlinear mobility continuity equations and generalized displacement
  convexity.
\newblock {\em Journal of Functional Analysis}, 258(4):1273--1309, 2010.

\bibitem{Carrillo2011}
Jos{\'e}~A Carrillo, Marco DiFrancesco, Alessio Figalli, Thomas Laurent, and
  Dejan Slepčev.
\newblock {Global-in-time weak measure solutions and finite-time aggregation
  for nonlocal interaction equations}.
\newblock {\em Duke Mathematical Journal}, 156(2):229 -- 271, 2011.

\bibitem{carrillo2020convergence}
Jos{\'e}~A Carrillo, Francis Filbet, and Markus Schmidtchen.
\newblock Convergence of a finite volume scheme for a system of interacting
  species with cross-diffusion.
\newblock {\em Numerische Mathematik}, 145:473--511, 2020.

\bibitem{carrillo2010particle}
Jos{\'e}~A Carrillo, Massimo Fornasier, Giuseppe Toscani, and Francesco Vecil.
\newblock Particle, kinetic, and hydrodynamic models of swarming.
\newblock In {\em Mathematical modeling of collective behavior in
  socio-economic and life sciences}, pages 297--336. Springer, 2010.

\bibitem{carrillo2018zoology}
Jos{\'e}~A Carrillo, Yanghong Huang, and Markus Schmidtchen.
\newblock Zoology of a nonlocal cross-diffusion model for two species.
\newblock {\em SIAM Journal on Applied Mathematics}, 78(2):1078--1104, 2018.

\bibitem{CJLV2016}
Jos{\'e}~A Carrillo, Francois James, Fr{\'e}d{\'e}ric Lagouti{\`e}re, and
  Nicolas Vauchelet.
\newblock The {F}ilippov characteristic flow for the aggregation equation with
  mildly singular potentials.
\newblock {\em Journal of Differential Equations}, 260(1):304--338, 2016.

\bibitem{carrillo2019population}
Jose~A Carrillo, Hideki Murakawa, Makoto Sato, Hideru Togashi, and Olena Trush.
\newblock A population dynamics model of cell-cell adhesion incorporating
  population pressure and density saturation.
\newblock {\em Journal of theoretical biology}, 474:14--24, 2019.

\bibitem{carrillo_chertock_huang_2015}
José~A. Carrillo, Alina Chertock, and Yanghong Huang.
\newblock A finite-volume method for nonlinear nonlocal equations with a
  gradient flow structure.
\newblock {\em Communications in Computational Physics}, 17(1):233–258, 2015.

\bibitem{1078-0947_2020_2_1191}
José~Antonio Carrillo, Marco Di~Francesco, Antonio Esposito, Simone Fagioli,
  and Markus Schmidtchen.
\newblock Measure solutions to a system of continuity equations driven by
  newtonian nonlocal interactions.
\newblock {\em Discrete \& Continuous Dynamical Systems}, 40(2):1191--1231,
  2020.

\bibitem{chang1970practical}
Christopher~J. Chang and Kenneth~G. Cooper.
\newblock A practical difference scheme for {F}okker-{P}lanck equations.
\newblock {\em Journal of Computational Physics}, 6(1):1--16, 1970.

\bibitem{chizat2018interpolating}
Lenaic Chizat, Gabriel Peyr{\'e}, Bernhard Schmitzer, and Fran{\c{c}}ois-Xavier
  Vialard.
\newblock An interpolating distance between optimal transport and fisher--rao
  metrics.
\newblock {\em Foundations of Computational Mathematics}, 18(1):1--44, 2018.

\bibitem{Chow2012}
Shui-Nee Chow, Wen Huang, Yao Li, and Haomin Zhou.
\newblock {Fokker}-{Planck} equations for a free energy functional or {Markov}
  process on a graph.
\newblock {\em Archive for Rational Mechanics and Analysis}, 203(3):969--1008,
  2012.

\bibitem{DEF2018}
Marco Di~Francesco, Antonio Esposito, and Simone Fagioli.
\newblock Nonlinear degenerate cross-diffusion systems with nonlocal
  interaction.
\newblock {\em Nonlinear Analysis}, 169:94--117, 2018.

\bibitem{di2018nonlinear}
Marco Di~Francesco, Antonio Esposito, and Simone Fagioli.
\newblock Nonlinear degenerate cross-diffusion systems with nonlocal
  interaction.
\newblock {\em Nonlinear Analysis}, 169:94--117, 2018.

\bibitem{di2021many}
Marco Di~Francesco, Antonio Esposito, and Markus Schmidtchen.
\newblock Many-particle limit for a system of interaction equations driven by
  newtonian potentials.
\newblock {\em Calculus of Variations and Partial Differential Equations},
  60(2):1--44, 2021.

\bibitem{di2016nonlocal}
Marco Di~Francesco and Simone Fagioli.
\newblock A nonlocal swarm model for predators--prey interactions.
\newblock {\em Mathematical Models and Methods in Applied Sciences},
  26(02):319--355, 2016.

\bibitem{dobrushin1979vlasov}
Roland~L’vovich Dobrushin.
\newblock Vlasov equations.
\newblock {\em Functional Analysis and Its Applications}, 13(2):115--123, 1979.

\bibitem{Dolbeault_2008}
Jean Dolbeault, Bruno Nazaret, and Giuseppe Savaré.
\newblock A new class of transport distances between measures.
\newblock {\em Calculus of Variations and Partial Differential Equations},
  34(2):193–231, Jun 2008.

\bibitem{d2006self}
Maria~R D’Orsogna, Yao-Li Chuang, Andrea~L Bertozzi, and Lincoln~S Chayes.
\newblock Self-propelled particles with soft-core interactions: patterns,
  stability, and collapse.
\newblock {\em Physical review letters}, 96(10):104302, 2006.

\bibitem{erbar2012gradient}
Matthias Erbar.
\newblock {Gradient flows of the entropy for jump processes}.
\newblock {\em Annales de l'Institut Henri Poincaré, Probabilités et
  Statistiques}, 50(3):920 -- 945, 2014.

\bibitem{ErbarMaasGradientFlowPorousMedium2014}
Matthias Erbar and Jan Maas.
\newblock Gradient flow structures for discrete porous medium equations.
\newblock {\em Discrete Contin. Dyn. Syst.}, 34(4):1355--1374, 2014.

\bibitem{EPSS2021}
Antonio Esposito, Francesco~S Patacchini, Andr{\'e} Schlichting, and Dejan
  Slep{\v{c}}ev.
\newblock Nonlocal-interaction equation on graphs: gradient flow structure and
  continuum limit.
\newblock {\em Archive for Rational Mechanics and Analysis}, pages 1--62, 2021.

\bibitem{eymard2000finite}
Robert Eymard, Thierry Gallou{\"e}t, and Rapha{\`e}le Herbin.
\newblock Finite volume methods.
\newblock {\em Handbook of numerical analysis}, 7:713--1018, 2000.

\bibitem{DiFrancesco2013}
Marco~Di Francesco and Simone Fagioli.
\newblock Measure solutions for non-local interaction {PDEs} with two species.
\newblock {\em Nonlinearity}, 26(10):2777--2808, sep 2013.

\bibitem{golse2016dynamics}
Fran{\c{c}}ois Golse.
\newblock On the dynamics of large particle systems in the mean field limit.
\newblock In {\em Macroscopic and large scale phenomena: coarse graining, mean
  field limits and ergodicity}, pages 1--144. Springer, 2016.

\bibitem{Lisini2010}
Stefano Lisini and Antonio Marigonda.
\newblock On a class of modified wasserstein distances induced by concave
  mobility functions defined on bounded intervals.
\newblock {\em manuscripta mathematica}, 133(1-2):197--224, June 2010.

\bibitem{LZ2021}
Nadia Loy and Mattia Zanella.
\newblock Structure preserving schemes for fokker--planck equations with
  nonconstant diffusion matrices.
\newblock {\em Mathematics and Computers in Simulation}, 188:342--362, 2021.

\bibitem{MaasGradFlowEntropyFiniteMarkov2011}
Jan Maas.
\newblock Gradient flows of the entropy for finite {M}arkov chains.
\newblock {\em J. Funct. Anal.}, 261(8):2250--2292, 2011.

\bibitem{mogilner1999non}
Alexander Mogilner and Leah Edelstein-Keshet.
\newblock A non-local model for a swarm.
\newblock {\em Journal of mathematical biology}, 38(6):534--570, 1999.

\bibitem{Monsaingeon2021}
L{\'e}onard Monsaingeon.
\newblock A new transportation distance with bulk/interface interactions and
  flux penalization.
\newblock {\em Calculus of Variations and Partial Differential Equations},
  60(3):101, Apr 2021.

\bibitem{neunzert1984introduction}
Helmut Neunzert.
\newblock An introduction to the nonlinear boltzmann-vlasov equation.
\newblock In {\em Kinetic theories and the Boltzmann equation}, pages 60--110.
  Springer, 1984.

\bibitem{painter2003modelling}
Kevin~J Painter and Jonathan~A Sherratt.
\newblock Modelling the movement of interacting cell populations.
\newblock {\em Journal of theoretical biology}, 225(3):327--339, 2003.

\bibitem{PZ2018}
Lorenzo Pareschi and Mattia Zanella.
\newblock Structure preserving schemes for nonlinear fokker--planck equations
  and applications.
\newblock {\em Journal of Scientific Computing}, 74(3):1575--1600, 2018.

\bibitem{schlichting2020scharfetter}
Andr{\'e} Schlichting and Christian Seis.
\newblock The {S}charfetter--{G}ummel scheme for aggregation-diffusion
  equations.
\newblock {\em arXiv preprint arXiv:2004.13981}, 2020.

\bibitem{Topaz2006}
Chad~M Topaz, Andrea~L Bertozzi, and Mark~A Lewis.
\newblock A nonlocal continuum model for biological aggregation.
\newblock {\em Bulletin of mathematical biology}, 68(7):1601—1623, October
  2006.

\bibitem{volkening2020modeling}
Alexandria Volkening, Madeline~R Abbott, N~Chandra, B~Dubois, F~Lim, D~Sexton,
  and Bjorn Sandstede.
\newblock Modeling stripe formation on growing zebrafish tailfins.
\newblock {\em Bulletin of mathematical biology}, 82(5):1--33, 2020.

\bibitem{volkening2015modelling}
Alexandria Volkening and Bj{\"o}rn Sandstede.
\newblock Modelling stripe formation in zebrafish: an agent-based approach.
\newblock {\em Journal of the Royal Society Interface}, 12(112):20150812, 2015.

\bibitem{volkening2018iridophores}
Alexandria Volkening and Bj{\"o}rn Sandstede.
\newblock Iridophores as a source of robustness in zebrafish stripes and
  variability in danio patterns.
\newblock {\em Nature communications}, 9(1):1--14, 2018.

\end{thebibliography}
\end{document}